\documentclass[reqno,11pt]{amsart}
\usepackage[colorlinks=true, linkcolor=blue, citecolor=blue]{hyperref}

\usepackage{amssymb}
\usepackage{amsmath, graphicx, rotating}
\usepackage{color}     
\usepackage{soul}
\usepackage[dvipsnames]{xcolor}   
\usepackage{tcolorbox}

\usepackage{ifthen}
\usepackage{xkeyval}
\usepackage{todonotes}
\usepackage[T1]{fontenc}
\usepackage{lmodern}
\usepackage[english]{babel}

\usepackage{ upgreek }
\usepackage{stmaryrd}
\SetSymbolFont{stmry}{bold}{U}{stmry}{m}{n}
\usepackage{amsthm}
\usepackage{float}

\usepackage{ bbm }
\usepackage{ stmaryrd }
\usepackage{ mathrsfs }
\usepackage{ frcursive }
\usepackage{ comment }

\usepackage{pgf, tikz}
\usetikzlibrary{shapes}
\usepackage{varioref}
\usepackage{enumitem}
\usepackage{longtable}

\usepackage{amsmath,amscd}

\usepackage{mathtools}

\usepackage{dsfont}

\definecolor{rouge}{rgb}{0.7,0.00,0.00}
\definecolor{vert}{rgb}{0.00,0.5,0.00}
\definecolor{bleu}{rgb}{0.00,0.00,0.8}

\usepackage{color,graphics}

\usepackage{tikz-cd}

\usepackage[margin=1.3in]{geometry}

\newtheorem{theorem}{Theorem}[section]
\newtheorem*{theorem*}{Theorem}
\newtheorem{lemma}[theorem]{Lemma}

\newtheorem{corollary}[theorem]{Corollary}
\newtheorem{proposition}[theorem]{Proposition}

\labelformat{hypothesis}{\textbf{M\kern-0.1mm#1}}

\newtheorem{condition}{Condition}

\labelformat{conditionA}{\textbf{A\kern-0.1mm#1}}

\renewcommand\dots{\hbox to 1em{.\hss.\hss.}}

\theoremstyle{definition}
\newtheorem{example}[theorem]{Example}
\newtheorem{remark}[theorem]{Remark}

\numberwithin{equation}{section}

\def\bb#1{\mathbb{#1}}

\def\bf#1{\mathbf{#1}}
\def\scr#1{\mathscr{#1}}

\def\geq{\geqslant}
\def\leq{\leqslant}

\def\wt{\widetilde}
\def\wh{\widehat}

\newcommand\ee{\varepsilon}

\DeclareMathOperator{\supp}{supp}

\DeclareMathOperator{\End}{End}

\def\geq{\geqslant}
\def\leq{\leqslant}
\def\Rd {\mathbb{R}^d}
\def\Rd*{(\mathbb{R}^d)^*}
\def\Pd{{\mathbb{P}}^{d-1}}
\def\Pd*{(\mathbb{P}^{d-1})^*}

\def\bb#1{\mathbb{#1}}

\begin{document}
\title[Local limit theorem for the operator norm]{Local limit theorem for the operator norm of products of random matrices} 

\author{Ion Grama}
\author{Jean-Fran\c cois Quint}
\author{Hui Xiao}

\curraddr[Grama, I.]{Univ Bretagne Sud, CNRS UMR 6205, LMBA, Vannes, France.}
\email{ion.grama@univ-ubs.fr}

\curraddr[Quint, J.-F.]{IMAG, Univ Montpellier, CNRS UMR 5149, Montpellier, France.}
\email{Jean-Francois.Quint@umontpellier.fr}

\curraddr[Xiao, H.]{State Key Laboratory of Mathematical Sciences, Academy of Mathematics and Systems Science, Chinese Academy of Sciences, Beijing 100190, China.}
\email{xiaohui@amss.ac.cn}

\date{\today}

\begin{abstract}
We prove a local limit theorem for the quantity $\| g_n \cdots g_1 \|$, where $g_1, g_2, \ldots$ is a sequence 
of independent and identically distributed random invertible square matrices. 
The norm $\| \cdot \|$ is arbitrary and we make no proximality assumption.  
\end{abstract}

\maketitle

\tableofcontents


\section{Introduction and results}
\subsection{Notation and background}

Let $V$ be a finite-dimensional real vector space. Denote by ${\rm GL}(V)$ the group of linear automorphisms of $V$.
We equip the space $\End(V)$ of endomorphisms on $V$ with an arbitrary norm $\| \cdot \|$; 
for instance, one may take the operator norm induced by a chosen norm on $V$. 
Let $\mu$ be a Borel probability measure on the group ${\rm GL}(V)$. 
Denote by $\Gamma_{\mu}$ the closed subsemigroup of ${\rm GL}(V)$ spanned by the support of the measure $\mu$. 
We say that $\Gamma_{\mu}$ is strongly irreducible if no finite union of proper non-zero subspaces 
of $V$ is $\Gamma_{\mu}$-invariant. 
The measure $\mu$ is said to admit a finite first moment if  $\int_{{\rm GL}(V)}  \log \max \{ \|g\|,  \|g^{-1} \| \} \mu(dg) < \infty$,  
and a finite exponential moment if there exists a constant $\alpha >0$ such that 
\begin{align}\label{Exponential-moment} 
\int_{{\rm GL}(V)} \max \{ \|g\|,  \|g^{-1} \| \}^{\alpha} \mu(dg) < \infty. 
\end{align}

Throughout the paper, we fix a probability space $(\Omega, \scr F, \bb P)$,
and a sequence $g_1, g_2, \ldots$ of 
independent and identically distributed (i.i.d.) random elements of ${\rm GL}(V)$ with law $\mu$. 
It is well known (see \cite{FK60, Boug-Lacr85, BQ16b}) that if the measure $\mu$ 
admits a finite first moment, 
then there exists a real number $\lambda_{\mu}$, called the first Lyapunov exponent of $\mu$,
such that,  $\bb P$-almost surely, 
\begin{align*}
\frac{1}{n}  \log \|g_n \cdots g_1 \|  \xrightarrow[n\to\infty]{}  \lambda_{\mu}. 
\end{align*}
Note that $\lambda_{\mu}$ does not depend on the choice of the norm $\|\cdot \|$. 

According to the results in \cite{BQ16b}, when $\Gamma_{\mu}$ is strongly irreducible 
and the measure $\mu$ admits a finite exponential moment,  the asymptotic variance   
\begin{align}\label{def-variance-aa}
\upsilon_{\mu}^2 = \lim_{n \to \infty} \frac{1}{n} \bb E \left[ (\log \|g_n \cdots g_1 \| - n \lambda_{\mu})^2 \right]
\end{align}
exists and does not depend on the choice of the norm. 
This asymptotic variance $\upsilon_{\mu}^2$ is equal to zero if and only if the set 
$e^{-\lambda_{\mu}} \supp \mu$ is contained in a compact subgroup of ${\rm GL}(V)$, where $\supp \mu$ denotes the support of the measure $\mu$. 
In particular, if $\lambda_{\mu} = 0$,
then $\upsilon_{\mu}$ is positive if and only if $\Gamma_{\mu}$ is unbounded in ${\rm GL}(V)$. 

The main result of this article is the following local limit theorem for the norm of products of random matrices.

\begin{theorem}\label{Main-Thm-LLT}
Assume that $\Gamma_{\mu}$ is strongly irreducible 
and that the image of $\Gamma_{\mu}$ in the projective linear group 
${\rm PGL}(V)$ is not contained in any compact subgroup of ${\rm PGL}(V)$. 
Assume also that $\mu$ admits a finite exponential moment \eqref{Exponential-moment}. 
Then, for any norm $\| \cdot \|$ on $\End (V)$ and any real numbers $a < b$, we have 
\begin{align}\label{LLT formula}
\upsilon_{\mu} \sqrt{2 \pi n} \,  \bb P  \left( \log \|g_n \cdots g_1 \| - n \lambda_{\mu} \in [a, b] \right)
\xrightarrow[n\to\infty]{}  b-a. 
\end{align}
\end{theorem}

\begin{remark}
Recall that an endomorphism of $V$ is said to be proximal if it admits a unique eigenvalue of maximal modulus and this eigenvalue 
is a simple root of its characteristic polynomial.  
In the case where $\Gamma_{\mu}$ contains a proximal element, 
Theorem \ref{Main-Thm-LLT} is essentially due to Le Page \cite{LePage82}, see also Guivarc'h \cite{Gui08}, Benoist and Quint \cite{BQ16b}. 
More precisely, Le Page \cite{LePage82} established Theorem \ref{Main-Thm-LLT} 
for the operator norm associated with a Euclidean norm on $V$,
under an additional non-arithmeticity assumption on the measure $\mu$. 
Thanks to the work of Benoist \cite{Benoist97, Be2}, Guivarc'h \cite{Gui08} was able to remove this kind of assumption 
in order to prove a multidimensional central limit theorem,
which simultaneously handles norms in several proximal representations of the same semigroup. 
The main difficulty in the present paper is to further get rid of the proximality assumption entirely.
\end{remark}

\begin{remark}
In Theorem \ref{Main-Thm-LLT}, 
one cannot replace the strong irreducibility assumption on $\Gamma_{\mu}$ by mere irreducibility without changing the conclusion. 
As a counterexample, let $V = \bb R^2$ and define random matrices
$$g_i = \begin{pmatrix} 0 & 1 \\  1 & 0  \end{pmatrix}^{\ee_i}  \begin{pmatrix} e^{t_i} & 0 \\  0 & e^{-t_i}  \end{pmatrix},$$
where $(\ee_i)_{i \geq 1}$ is an i.i.d.\ sequence of Bernoulli random variables taking values $0$ and $1$,
and $(t_i)_{i \geq 1}$ is a real-valued i.i.d.\ sequence, independent of $(\ee_i)_{i \geq 1}$, 
with a finite second moment and satisfying a non-arithmeticity condition. 
In this case, the subsemigroup $\Gamma_{\mu}$ is irreducible (it does not preserve any nontrivial subspace) 
but not strongly irreducible (it preserves the union of the two coordinate axes).  
For the standard Euclidean operator norm on $\bb R^2$, the local limit theorem then becomes:
for any real numbers $a < b$, 
\begin{align*}
\upsilon_{\mu} \sqrt{2 \pi n} \,  \bb P  \left( \log \|g_n \cdots g_1 \|  \in [a, b] \right) 
\xrightarrow[n\to\infty]{}  2 (\max\{b, 0\} - \max\{a, 0\}). 
\end{align*}
Thus the limiting distribution is not the Lebesgue measure on $[a, b]$ but rather a reflected version concentrated on the positive half-line.

Using the ideas developed in this paper, 
one could describe the local limit behavior of $\log \|g_n \cdots g_1 \|$ under the sole assumption
that $\Gamma_{\mu}$ is irreducible, or even totally reducible. 
However, the resulting statement would be considerably more intricate than Theorem \ref{Main-Thm-LLT}. 
\end{remark}

\begin{remark}
If the image of $\Gamma_{\mu}$ in ${\rm PGL}(V)$ is contained in a compact subgroup of ${\rm PGL}(V)$, 
then one may find a compact subgroup $K \subset {\rm GL}(V)$ that acts strongly irreducibly on $V$
and such that $\Gamma_{\mu}$ is contained in the group $\bb R_+^* K$ of linear automorphisms which are positive multiples
of an element in $K$. 
This group is isomorphic to the product group $\bb R \times K$ and the projection of the random walk $g_n \cdots g_1$
on the real component is a sum $S_n = X_1 + \ldots + X_n$ of i.i.d.\ real-valued random variables. 
In other words, we may write $g_n \cdots g_1 \in e^{S_n} K$. 
If the random walk $S_n$ is non-lattice, then the local limit theorem \eqref{LLT formula} remains valid. 
In general, however, the asymptotic behavior of the left-hand side of \eqref{LLT formula} depends on the lattice properties of $S_n$. 
\end{remark}

\subsection{Sketch of the proof}
As in the original work by Le Page \cite{LePage82}, the strategy is to approximate the quantity 
$\log \|g_n \cdots g_1 \|$ by a quantity of the form $\sigma(g_n \cdots g_1, x)$, where $\sigma: \Gamma_{\mu} \times \bb X \to \bb R$
is an additive cocycle on a compact space $\bb X$ equipped with an action of $\Gamma_{\mu}$ and $x$ is an element of $\bb X$. 
This approximation is achieved in Section \ref{Sec-LLT-approximation} and we are brought back to proving a local limit theorem for the cocycle 
$\sigma(g_n \cdots g_1, x)$. 

The main difficulty compared to \cite{LePage82, Gui08, BQ16b} is that the action of $\Gamma_{\mu}$ on $\bb X$
is not contracting, that is, starting from two different points $x, x' \in \bb X$, the two random trajectories $g_n \cdots g_1 x$ and $g_n \cdots g_1 x'$
do not in general get close to each other.  
This reflects the fact that we have not assumed that $\Gamma_{\mu}$ contains a proximal element. 
Fortunately, this lack of proximality is precisely described by the fact that the space $\bb X$ is also equipped with an action of a compact group $M$, 
commuting with the action of $\Gamma_{\mu}$, such that the induced action of $\Gamma_{\mu}$ on
the orbit space $M \backslash \bb X$ possesses the contraction property used in \cite{LePage82, Gui08, BQ16b}. 
Consequently, using the representation theory of the compact group $M$, one can reduce the study of functions on $\bb X$ 
to the study of sections of certain vector bundles over $M \backslash \bb X$. 
This viewpoint is employed throughout the paper. 

In Section \ref{Section Vector bundles}, we recall the definition of vector bundles
and prove a spectral gap theorem 
for Markov operators acting on spaces of sections of vector bundles whose base space is contracting.  
This is an extension of the spectral gap properties in spaces of functions considered in \cite{LePage82, Gui08, BQ16b}. 

In Section \ref{Section principal bundles}, we discuss $M$-principal bundles.
This will be useful later since the quotient map $\bb X \to M \backslash \bb X$ is an $M$-principal bundle.
We use this language to relate $M$-equivariant maps from $\bb X$ to representation spaces of $M$ with sections of vector bundles
over $M \backslash \bb X$. 

In Section \ref{Sec-Existence of limit measures}, 
we combine the above concepts with the representation theory of the compact group $M$
to prove an equidistribution property for random walks on the total space of an $M$-principal bundle whose base is contracting. 
Still within this abstract framework, in Section \ref{Section-LLT}, 
we establish a local limit theorem for real-valued cocycles over such random walks on the total space of principal bundles.

In Section \ref{Sec-action-Strong-irre}, we introduce the specific principal bundle relevant to our problem
and make explicit, for this example, the conclusions of the previous two sections. 
This yields a local limit theorem for certain cocycles,  
which is then used in Section \ref{Sec-LLT-approximation} to deduce Theorem \ref{Main-Thm-LLT} 
thanks to the aforementioned approximation of $\log \|g_n \cdots g_1 \|$ by $\sigma(g_n \cdots g_1, x)$.

\section{Vector bundles and spectral gap}\label{Section Vector bundles}

\subsection{Generalities on vector bundles}\label{Subsec-vector-bundles}
We first recall the definition of vector bundles that will be used throughout this section. 
Let $B$ be a compact metric space equipped with a distance $\bf{d}$; this space will serve as the base space of the bundle. 
A Lipschitz vector bundle of dimension $d$ over $B$ is the data of another topological space $E$ 
together with a continuous surjective map $\pi: E \to B$, an open covering $(U_i)_{i \in I}$ of $B$ and a collection $(\phi_i)_{i \in I}$ of maps
such that 
\begin{enumerate}
\item
for any $i \in I$, $\phi_i$ is a continuous map from $\pi^{-1} (U_i)$ to $\bb C^d$,
and the map $(\pi \times \phi_i): w \mapsto (\pi w, \phi_i w)$ is a homeomorphism from $\pi^{-1} (U_i)$ onto $U_i \times \bb C^d$. 

\item
for any $i, j \in I$, there exists a Lipschitz continuous map $\theta_{i, j}: U_i \cap U_j \to {\rm GL}_d(\bb C)$
such that for every $w \in \pi^{-1} (U_i \cap U_j)$, we have 
\begin{align}\label{transition function}
\phi_i w = \left(  \theta_{i, j} \left( \pi w \right) \right) \left( \phi_j  w \right). 
\end{align}
\end{enumerate}

Such a collection $(U_i, \phi_i, \theta_{i, j})_{i, j \in I}$ will be called a local trivialization of the bundle. 
Two local trivializations $(U_i, \phi_i, \theta_{i, j})_{i, j \in I}$ and $(U'_{i'}, \phi'_{i'}, \theta_{i', j'}')_{i', j' \in I'}$ are said to be compatible 
if for every $i \in I$ and $i' \in I'$, there exists a Lipschitz continuous map $\tau_{i, i'}: U_i \cap U'_{i'}  \to {\rm GL}_d(\bb C)$
such that for all $w \in \pi^{-1} (U_i \cap U'_{i'})$, 
\begin{align*}
\phi_i w = \left(  \tau_{i, i'} \left( \pi w \right) \right) \left( \phi'_{i'} w \right). 
\end{align*}
A local trivialization is said to be maximal if it contains every  local trivialization that is compatible with it. 
One easily verifies that any local trivialization is compatible with a unique maximal one
and therefore, in the definition of vector bundle, we always assume that 
the local trivialization is maximal.

The fact that the transition functions $\theta_{i, j}$ in \eqref{transition function} 
are Lipschitz continuous with respect to the distance $\bf{d}$ on $B$ 
is essential for the spectral gap arguments in the next subsection.

The space $B$ is called the base space, the space $E$ is called the total space, 
and the map $\pi$ is called the projection of the bundle. 
For each $b \in B$, the preimage $\pi^{-1}(b)$ is called the fiber of the bundle over $b \in B$.
Intuitively, one can view a vector bundle as a family of fibers $\pi^{-1}(b)$
parametrized by $b \in B$  and glued together according to the global topology of the total space $E$.
Note that, for each $b \in B$, the fiber $\pi^{-1} (b)$ is endowed with a natural structure of a $d$-dimensional vector space 
over $\bb C$. 
Indeed, if $b \in U_i$, for $w, w' \in \pi^{-1} (b)$ and $\lambda \in \bb C$, we can set 
\begin{align}\label{fiber-vector-space-001}
w + w' = (\pi \times \phi_i)^{-1} (b, \phi_i w + \phi_i w')  
\quad \mbox{and}  \quad 
\lambda w =  (\pi \times \phi_i)^{-1} (b, \lambda \phi_i w), 
\end{align}
where $(\pi \times \phi_i)^{-1} : U_i \times \mathbb{C}^d \to \pi^{-1}(U_i)$ is the homeomorphism inverse to 
$w \mapsto (\pi(w), \phi_i(w))$. 
These definitions do not depend on the choice of $(U_i)_{i \in I}$, 
thanks to the compatibility condition (2) for the transition functions.

 A continuous section of the bundle $E$ is a continuous map $s: B \to E$ such that $\pi (s(b)) = b$ (i.e. $s(b) \in \pi^{-1}(b)$) for every $b \in B$.
 For $\gamma \in (0, 1]$, the continuous section $s$ is said to be $\gamma$-H\"older continuous if 
for every $i \in I$, the map $\phi_i \circ s: U_i \to \bb C^d$ is locally $\gamma$-H\"older continuous, 
where $(U_i, \phi_i, \theta_{i, j})_{i, j \in I}$ is the maximal local trivialization which defines the bundle $E$. 
Note that it actually suffices to check the above condition for any local trivialization which is compatible with 
the maximal one $(U_i, \phi_i, \theta_{i, j})_{i, j \in I}$. 
 
We denote by $\boldsymbol{\Gamma}^0(E)$ the space of all continuous sections of $E$. 
For $\gamma \in (0, 1]$, we also denote by $\boldsymbol{\Gamma}^{\gamma}(E)$
the space of $\gamma$-H\"older continuous sections of $E$. 
These are vector spaces for the natural operations. 
In the sequel, we shall equip them with appropriate Banach space structures.

We fix a norm $\| \cdot \|$ on $\bb C^d$ and choose a local trivialization $(U_i, \phi_i, \theta_{i, j})_{i, j \in I}$
such that the covering $(U_i)_{i \in I}$ of $B$ is finite (i.e. the index set $I$ is finite). 
By the compactness of $B$, we can further select, for each $i \in I$, an open subset $V_i \subset U_i$ 
with the property that the closure of $V_i$ in $B$ is contained in $U_i$, 
and such that the family $(V_i)_{i \in I}$ still covers $B$: $B = \bigcup_{i \in I} V_i$.  
For a continuous section $s$ of $E$, we define its uniform norm by 
\begin{align}\label{def-norm-s-infty}
\|s\|_0 = \max_{i \in I} \sup_{b \in V_i} \| \phi_i \circ s (b) \|. 
\end{align}
For $\gamma \in (0, 1]$, if $s$ is a $\gamma$-H\"older continuous section of $E$, 
we define its $\gamma$-H\"older norm by
\begin{align}\label{def-norm-s-Holder}
\|s\|_{\gamma} = \max_{i \in I} \Bigg( \sup_{b \in V_i} \| \phi_i \circ s (b) \| 
+ \sup_{\substack{b, b' \in V_i\\ b \neq b'}}  \frac{\| \phi_i \circ s (b) - \phi_i \circ s (b') \|}{\bf{d}(b, b')^{\gamma}}  \Bigg), 
\end{align}
where $\bf{d}$ is the distance on $B$. 
One easily checks that these formulae define Banach norms on the spaces $\boldsymbol{\Gamma}^{\gamma}(E)$ 
for any $\gamma \in [0, 1]$.
Up to norm equivalence, these norms do not depend on the choices that we made for the norm $\|\cdot\|$ on $\bb C^d$, 
the local trivialization $(U_i, \phi_i, \theta_{i, j})_{i, j \in I}$, 
or the covering $(V_i)_{i \in I}$.  

As in the case of function spaces, we have the following lemma,
which is an extension of the classical fact that, over a compact space,
the embedding of the space of $\gamma$-H\"older continuous functions into the space of continuous functions
is a compact operator.

\begin{proposition}\label{Sobolev-injection-lemma}
Let $B$ be a compact metric space and let $\pi: E \to B$ be a 
Lipschitz vector bundle of dimension $d$ over $B$. 
Then, for any $\gamma \in (0,1]$, the natural injection $\boldsymbol{\Gamma}^{\gamma}(E) \to \boldsymbol{\Gamma}^{0}(E)$
is a compact operator with dense image. 
\end{proposition}

\begin{proof}
We first prove that the natural injection is a compact operator. 
Let $(s_n)_{n \geq 1}$ be a bounded sequence in $\boldsymbol{\Gamma}^{\gamma}(E)$. 
We need to show that it admits a subsequence converging in $\boldsymbol{\Gamma}^{0}(E)$. 
Choose a finite local trivialization $(U_i, \phi_i, \theta_{i, j})_{i, j \in I}$, 
and for each $i \in I$, an open subset $V_i \subset U_i$ such that the closure of $V_i$ in $B$ is contained in $U_i$ 
and $B = \bigcup_{i \in I} V_i$. 
For each $i \in I$ and $n \geq 1$, 
define the $\bb C^d$-valued function $f^i_n: \overline{V_i} \to \bb C^d$ by 
\begin{align*}
f^i_n(b) = \phi_i (s_n(b)).
\end{align*}
Since $(s_n)_{n \geq 1}$ is bounded in $\boldsymbol{\Gamma}^{\gamma}(E)$, 
the family 
$(f^i_n)_{n \geq 1}$ is uniformly bounded in the space of $\gamma$-H\"older continuous functions
on the compact space $\overline{V_i}$. 
By the Arzel\`a-Ascoli theorem, after passing to a subsequence (common to all $i$ by a diagonal argument), 
we may assume that, for each $i \in I$,
there exists a continuous function $f^i: \overline{V_i} \to \bb C^d$ such that 
$f^i_n \xrightarrow[n\to\infty]{}  f^i$ uniformly on $\overline{V_i}$. 
For any $i, j \in I$, $n \geq 1$ and $b \in V_i \cap V_j$, we have 
\begin{align*}
(\pi \times \phi_i)^{-1}(b, f_n^i (b)) = s_n(b) = (\pi \times \phi_j)^{-1}(b, f_n^j (b)). 
\end{align*}
Letting $n \to \infty$ and using the continuity of the maps $(\pi \times \phi_i)^{-1}$, we obtain 
\begin{align*}
(\pi \times \phi_i)^{-1}(b, f^i (b))  = (\pi \times \phi_j)^{-1}(b, f^j (b)). 
\end{align*}
Thus we can define a section $s : B \to E$ by setting, for any $b\in B$ and any $i \in I$ with $b \in V_i$,
\begin{align*}
s(b) = (\pi \times \phi_i)^{-1}(b, f^i(b)).
\end{align*}
This definition does not depend on the choice of $i \in I$ and hence we get $s \in \boldsymbol{\Gamma}^{0}(E)$. 
Moreover, from the uniform convergence of $f_n^i$ to $f^i$ on each $\overline{V_i}$,
we deduce that $s_n \xrightarrow[n\to\infty]{} s$ in the space $\boldsymbol{\Gamma}^{0}(E)$. 
Hence the natural injection $\boldsymbol{\Gamma}^{\gamma}(E) \to \boldsymbol{\Gamma}^{0}(E)$ is a compact operator.

Now we prove that $\boldsymbol{\Gamma}^{1}(E)$ is dense in $\boldsymbol{\Gamma}^{0}(E)$.
Let $s$ be an arbitrary continuous section of $E$.
Our goal is to approximate $s$ by a Lipschitz continuous section $s'$. 
To this end, take a finite Lipschitz partition of unity $(h_i)_{i \in I}$ associated with the open cover 
$B = \bigcup_{i \in I} V_i$,
thus for each $i \in I$, the function $h_i: B \to [0, 1]$ is Lipschitz continuous,  $\supp h_i \subset V_i$
and we have $\sum_{i \in I} h_i = 1$ on $B$. 
 For each $i \in I$, we set $s_i = h_i s$, so that $s = \sum_{i \in I} s_i$.
Next, define a continuous function $f_i: U_i \to \bb C^d$ by $f_i(b) = \phi_i (s_i(b))$. 
For each $i \in I$, 
as $f_i$ is continuous and compactly supported on $V_i$, for any $\ee>0$, 
we can find a Lipschitz continuous function $f_i': U_i \to \bb C^d$, 
also with compact support contained in $V_i$, such that $\sup_{b \in U_i} \|f_i(b) - f_i'(b) \| < \ee$.  
We then define a section $s_i'$ of $E$ by 
\begin{align}\label{def new section}
\phi_i (s_i'(b)) = f_i'(b) \quad \mbox{for}  \  b \in V_i,  
\quad 
s_i'(b) = 0  \quad \mbox{for}  \  b \in  B \setminus V_i. 
\end{align}

We claim that $s_i': B \to E$ is a Lipschitz continuous section. 
Indeed, since $s_i'$ vanishes outside of $V_i$,  we have that, for any $j \in I$,
\begin{align}\label{decom-Holder-norm-phi-s}
& \sup_{\substack{b, b' \in V_j \\  b \neq b'}}  \frac{\|\phi_j(s_i'(b)) - \phi_j(s_i'(b'))\|}{\bf d(b, b')}   \notag\\
 & = \sup_{\substack{b \in V_i \cap V_j, b' \in V_j \\  b \neq b'}}  \frac{\|\phi_j(s_i'(b)) - \phi_j(s_i'(b'))\|}{\bf d(b, b')} \notag\\
& = \max \Bigg\{\sup_{\substack{b \in V_i \cap V_j, b' \in U_i^c \cap V_j \\  b \neq b'}}  \frac{\|\phi_j(s_i'(b))\|}{\bf d(b, b')},
  \sup_{\substack{b \in V_i \cap V_j, b' \in U_i \cap V_j \\  b \neq b'}}  \frac{\|\phi_j(s_i'(b)) - \phi_j(s_i'(b'))\|}{\bf d(b, b')}
 \Bigg\}, 
\end{align}
where we have split the supremum according to whether $b'$ belongs to the set $U_i$ or not. 
For the first term in \eqref{decom-Holder-norm-phi-s}, 
by \eqref{transition function} and \eqref{def new section}, 
we have that for any $b \in V_i \cap V_j$, 
\begin{align}\label{identity transition}
\phi_j(s_i'(b)) = (\theta_{j,i}(b)) (\phi_i(s_i'(b))) = (\theta_{j,i}(b)) (f_i'(b)), 
\end{align}
and hence 
\begin{align*}
\sup_{\substack{b \in V_i \cap V_j, b' \in U_i^c \cap V_j \\  b \neq b'}}  \frac{\|\phi_j(s_i'(b))\|}{\bf d(b, b')}
 \leq  \sup_{b \in V_i \cap V_j}  \frac{\| (\theta_{j,i}(b)) (f_i'(b))\|}{\bf d(b, U_i^c)} 
 < \infty,
\end{align*}
where in the last inequality we have used the fact that $V_i \cap V_j$ has compact closure in $U_i \cap U_j$ 
and that the maps $\theta_{j,i}$, $f_i'$ and $\bf d(\cdot, U_i^c)$ are continuous on $V_i \cap V_j$, 
with $\bf d(b, U_i^c) > 0$ since $\overline{V_i} \subset U_i$. 
For the second term in \eqref{decom-Holder-norm-phi-s}, using again \eqref{identity transition} and the triangle inequality, we get 
\begin{align*}
\|\phi_j(s_i'(b)) - \phi_j(s_i'(b'))\| 
& = \| (\theta_{j,i}(b)) (f_i'(b))- (\theta_{j,i}(b'))(f_i'(b')) \|  \notag\\
& \leq  \| (\theta_{j,i}(b)) (f_i'(b) - f_i'(b')) \| + \| (\theta_{j,i}(b) - \theta_{j,i}(b')) (f_i'(b')) \|, 
\end{align*}
so that 
\begin{align*}
\frac{\|\phi_j(s_i'(b)) - \phi_j(s_i'(b'))\|}{\bf d(b, b')}
\leq  \frac{\| (\theta_{j,i}(b)) (f_i'(b) - f_i'(b')) \|}{\bf d(b, b')}  + \frac{\| (\theta_{j,i}(b) - \theta_{j,i}(b')) (f_i'(b')) \|}{\bf d(b, b')}. 
\end{align*}
Since $\theta_{j,i}$ is Lipschitz continuous on $U_i \cap U_j$ and $V_i \cap V_j$ has compact closure in $U_i \cap U_j$,
we have $\sup_{b \in V_i \cap V_j} \|\theta_{j,i}(b)\| < \infty$.
Moreover, as $f_i'$ is Lipschitz continuous on $U_i$, we get 
\begin{align*}
\sup_{\substack{b \in V_i \cap V_j, b' \in U_i \cap V_j \\  b \neq b'}}  \frac{\| (\theta_{j,i}(b)) (f_i'(b) - f_i'(b')) \|}{\bf d(b, b')} < \infty. 
\end{align*}
Besides, since $f_i'$ is Lipschitz continuous with compact support contained in $U_i$, we get that $\sup_{b' \in U_i} \|f_i'(b')\| < \infty$
and hence, as $\theta_{j,i}$ is Lipschitz continuous on $U_i \cap U_j$, 
\begin{align*}
\sup_{\substack{b \in V_i \cap V_j, b' \in U_i \cap V_j \\  b \neq b'}}  \frac{\| (\theta_{j,i}(b) - \theta_{j,i}(b')) (f_i'(b')) \|}{\bf d(b, b')} < \infty. 
\end{align*}
This finishes the proof that $s_i': B \to E$ is a Lipschitz continuous section.

To conclude the proof, we need to show that, for each $i \in I$, the section $s_i$ is well approximated by $s_i'$. 
Since both $s_i$ and $s_i'$ vanish outside of $V_i$, and $\sup_{b \in V_i} \|f_i(b) - f_i'(b) \| < \ee$, 
by \eqref{def-norm-s-infty} and \eqref{identity transition} we get 
\begin{align*}
\|s_i -s_i'\|_0 
& = \max_{j \in I} \sup_{b \in V_j} \| \phi_j(s_i(b)) - \phi_j(s_i'(b)) \|   \notag\\
& = \max_{j \in I} \sup_{b \in V_j \cap V_i} \| \phi_j(s_i(b)) - \phi_j(s_i'(b)) \|   \notag\\
& =  \max_{j \in I} \sup_{b \in V_j \cap V_i} \| (\theta_{j,i}(b)) (f_i(b)- f_i'(b)) \|   \notag\\
& \leq \ee \max_{j \in I} \sup_{b \in V_j \cap V_i} \| \theta_{j,i}(b) \|   \notag\\ 
& \leq C \ee,
\end{align*}
where $C= \max_{i, j \in I} \sup_{b \in V_j \cap V_i} \| \theta_{j,i}(b) \|$ is finite since for any $i, j \in I$, 
the closure of $V_i \cap V_j$ in $B$ is contained in $U_i \cap U_j$
and $\theta_{j,i}$ is continuous there. 
Finally, we set $s' = \sum_{i \in I} s_i'$.  Then 
\begin{align*}
\| s - s' \|_0 \leq \sum_{i \in I} \|s_i -s_i'\|_0 \leq  C |I|  \ee. 
\end{align*}
Since $\ee > 0$ is arbitrary, we can make $\| s - s' \|_0$ as small as desired, 
proving that $\boldsymbol{\Gamma}^{1}(E)$ is dense in $\boldsymbol{\Gamma}^{0}(E)$. 
\end{proof}

\begin{example} \label{Example-001}
A particularly simple (though not very interesting) example of a vector bundle is the \emph{trivial bundle}. 
In this case, the total space is $E = B \times \mathbb{C}^d$, 
and the projection $\pi : E \to B$ is given by $\pi(b, v) = b$. 
A continuous section $s : B \to E$ then necessarily has the form $s(b) = (b, \varphi(b))$ for some continuous function $\varphi : B \to \mathbb{C}^d$. 
Under this identification, for any $\gamma \in [0,1]$, the space $\boldsymbol{\Gamma}^{\gamma}(E)$ 
of $\gamma$-H\"older continuous sections is naturally identified with the space of $\gamma$-H\"older continuous functions 
from $B$ to $\mathbb{C}^d$.
\end{example}

\begin{example} \label{Example-002}
A non-trivial example that the reader may keep in mind is the case where $B = \bb R / \bb Z$ is the unit circle 
and $E = \widetilde{E}/\iota$ is the quotient of the space $\widetilde{E} = (\bb R / (2 \bb Z)) \times \bb C^d$ 
with the involution $\iota: (b, v) \mapsto (b+1, -v)$. 
The projection on the first component $\widetilde{E} \to \bb R / (2 \bb Z)$ factors as a map $\pi: E \to B$. 
Then $\pi : E \to B$ is a nontrivial vector bundle of rank \(d\) over the circle. 
For $\gamma \in [0, 1]$, the space $\boldsymbol{\Gamma}^{\gamma}(E)$ of $\gamma$-H\"older continuous sections 
can be identified with the space of $\gamma$-H\"older continuous 
functions $\varphi: \bb R / (2 \bb Z) \to \bb C^d$ satisfying the antiperiodicity condition: for any $b \in \bb R / (2 \bb Z)$, 
one has $\varphi(b+1) = - \varphi(b)$. 
Indeed, given such a $\varphi$, 
the map $b \mapsto (b, \varphi(b))$ from $\bb R / (2 \bb Z)$ to $\widetilde{E}$ factors as 
a well-defined section $s : B \to E$, and this section is $\gamma$-H\"older continuous.
See Section \ref{Sec-Equivariant maps} below where this example will fit into the general framework of principal bundles. 
\end{example}

\subsection{Group actions and spectral gap}\label{Subsection-Group actions-spectral gap}
We now define actions of groups on vector bundles. 
Let $G$ be a locally compact group that acts continuously on both the base space $B$ and the total space $E$ in such a way
that for every $g \in G$ and $w \in E$, we have 
\begin{align}\label{property-group-action}
\pi (gw) = g \pi(w), 
\end{align}
and that $g$ preserves the linear structure on each fiber. 
Note that if $s$ is a section of $E$, then, by \eqref{property-group-action}, 
for any $g \in G$, the map $b \mapsto g^{-1} s(gb)$ is also a section of $E$ (since $\pi(g^{-1} s(gb)) = g^{-1} \pi (s(gb)) = g^{-1} g b = b$). 
Besides, for any indices $i, j \in I$ (where $(U_i, \phi_i, \theta_{i, j})_{i, j \in I}$ is a local trivialization of $E$), 
there exists a unique map 
\begin{align*}
\sigma_{i, j}:  \{ (g, b) \in G \times U_i: gb \in U_j \}  \longrightarrow  {\rm GL}_d(\bb C)
\end{align*}
such that for every $g \in G$, $v \in \bb C^d$ and $b \in U_i \cap  (g^{-1} U_j)$, we have 
\begin{align}\label{def-sigma-ij}
g \left( \pi \times \phi_i \right)^{-1} (b, v) = \left( \pi \times \phi_j \right)^{-1} (gb, \sigma_{i, j} (g, b) v). 
\end{align}
Later we will be mainly concerned with the case where all the linear maps $\sigma_{i, j}(g, b)$ are isometries of $\bb C^d$, 
for any $i, j \in I$, $g \in G$ and $b \in U_i \cap (g^{-1} U_j)$.  
In this case, we say that the action of $G$ on $E$ is an \emph{isometric bundle action}. 
Moreover, we say that the action of $G$ on $E$ is a \emph{Lipschitz bundle action} 
if for every $g \in G$, the map $b \mapsto gb$ is Lipschitz continuous on $B$
and for all $i, j \in I$, the map $b \mapsto \sigma_{i, j}(g, b)$ is locally Lipschitz continuous on $U_i \cap  (g^{-1} U_j)$. 
These definitions do not depend on the choice of the local trivialization $(U_i, \phi_i, \theta_{i, j})_{i, j \in I}$.

Given a Lipschitz bundle action of $G$ on $E$, we define the Lipschitz constant of an element $g \in G$ as follows. 
Again we fix a norm $\| \cdot \|$ on $\bb C^d$ and choose a local trivialization $(U_i, \phi_i, \theta_{i, j})_{i, j \in I}$
such that the covering $(U_i)_{i \in I}$ of $B$ is finite. 
For each $i \in I$, 
pick an open subset $V_i \subset U_i$ with $\overline{V_i} \subset U_i$ and such that $B = \bigcup_{i \in I} V_i$.  
Then we set 
\begin{align}\label{def-Lip-norm-g}
{\rm Lip}_E(g) 
 = \sup_{\substack{b, b' \in B \\ b \neq b'}} \frac{\bf d (gb, gb')}{\bf d (b, b')}  
 + \max_{i, j \in I} 
  \sup_{\substack{b, b' \in  V_i  \cap  (g^{-1} V_j)    \\ b \neq b'}}   \frac{\| \sigma_{i, j}(g, b) - \sigma_{i, j}(g, b') \|}{\bf d (b, b')}. 
\end{align}
This definition depends on the choices made (norm, trivialization, and the sets $(V_i)_{i \in I}$) 
only up to a multiplicative constant that is independent of $g$. 
Hence, for our purpose, it defines an equivalent notion of Lipschitz constant.

Let $\mu$ be a Borel probability measure on $G$,
and let $(g_n)_{n \geq 1}$ be a sequence of independent random variables with law $\mu$. 
We say that the pair $(G, \mu)$ has the contraction property on $B$
if there exist constants $\delta \in (0, 1]$, $c>0$ and $a \in (0, 1)$ such that 
for all $b, b' \in B$ and all $n \geq 1$, 
\begin{align}\label{contraction-property-on-B}
\bb E \left( \bf d (g_n \cdots g_1 b, g_n \cdots g_1 b')^{\delta} \right) \leq c a^n \bf d(b, b')^{\delta}. 
\end{align} 
The main result of this section is the following theorem. 

\begin{theorem}\label{Thm-spectral-gap}
Let $B$ be a compact metric space and let $\pi: E \to B$ be a 
Lipschitz vector bundle of dimension $d$ over $B$. 
Let $G$ be a locally compact group acting on $E$ by an isometric Lipschitz bundle action. 
Let $\mu$ be a Borel probability measure on $G$. 
For any continuous section $s: B \to E$, define the operator $P_\mu$ by
\begin{align*}
P_{\mu} s (b) = \int_{G}  g^{-1} s (g b) \mu(dg), \quad b \in B. 
\end{align*}
Then $P_{\mu} s: B \to E$ is again a continuous section, and the operator $P_{\mu}$ is bounded on $\bf \Gamma^{0}(E)$. 

Assume furthermore that there exists a constant $\alpha >0$ such that $\bb E ({\rm Lip}_E(g_1)^{\alpha} ) < \infty$. 
Then, for any $\gamma \in (0, \alpha)$, 
the operator $P_{\mu}$ is bounded on $\bf \Gamma^{\gamma}(E)$. 
Moreover, if $(G, \mu)$ satisfies the contraction property \eqref{contraction-property-on-B} on $B$,
then for any $\gamma \in (0, \min \{ \delta, \alpha \})$, 
the operator $P_{\mu}$ has essential spectral radius strictly less than $1$ on $\bf \Gamma^{\gamma}(E)$. 
\end{theorem}

\begin{example} \label{Example-003}
In the case of the trivial bundle $E = B \times \bb C$ (see Example \ref{Example-001}), 
suppose that a group $G$ acts on $B$ by Lipschitz continuous homeomorphisms. 
We let $G$ act on $E$ by the formula $g(b, v) = (gb, v)$
for $g \in G$, $b \in B$ and $v \in \bb C$. 
Then, for a section $s(b) = (b, \varphi(b))$, the operator $P_{\mu}$ defined in Theorem \ref{Thm-spectral-gap} becomes 
$P_{\mu} s (b) = (b, \int_{G} \varphi(gb) \mu(dg))$, or equivalently, 
\begin{align}\label{operator-classical}
P_{\mu} \varphi(b) = \int_{G} \varphi(gb) \mu(dg). 
\end{align}
Thus $P_{\mu}$ can be identified with the classical transfer operator acting on continuous functions on $B$. 
In this setting, Theorem \ref{Thm-spectral-gap} then reduces to a well-known result; see \cite{LePage82} and \cite[Lemma 11.9]{BQ16b}. 
\end{example}

\begin{example}
In the case of Example \ref{Example-002} with $d=1$, the operator $P_{\mu}$ may be identified with the operator defined by 
the same formula \eqref{operator-classical} as in Example \ref{Example-003},
but now acting on continuous functions $\varphi: \bb R / (2 \bb Z) \to \bb C$ such that, for any $b \in \bb R / (2 \bb Z)$,
$\varphi(b+1) = - \varphi(b)$. 
Indeed, given an action of some group $G$ by Lipschitz continuous homeomorphisms on $B = \bb R / \bb Z$, 
by the theory of covering spaces (cf.\ \cite[Chapter 3]{Die08}, in particular, Theorem 3.5.2), 
up to replacing $G$ by a central extension by $\bb Z/(2 \bb Z)$,
we may assume that the action of $G$ on $B = \bb R / \bb Z$ is the quotient of an action of $G$ on $\bb R / (2 \bb Z)$
such that, for any $g \in G$ and $b \in \bb R / (2 \bb Z)$, $g(b+1) = gb + 1$. 
We then let $G$ act on $(\bb R / (2 \bb Z)) \times \bb C$ by the formula $g(b, v) = (gb, v)$,
so that $G$ acts on the quotient space $E$ by the quotient action. 
\end{example}

\begin{proof}[Proof of Theorem \ref{Thm-spectral-gap}]
We begin by proving that the operator $P_{\mu}$ is bounded on $\bf \Gamma^{0}(E)$. 
Fix a finite local trivialization $(U_i, \phi_i, \theta_{i, j})_{i, j \in I}$ 
and as above, for each $i \in I$, 
choose an open subset $V_i \subset U_i$ such that $\overline{V_i} \subset U_i$ and $B = \bigcup_{i \in I} V_i$. 
For $i , j \in I$, we let $\sigma_{i, j}:  \{ (g, b) \in G \times U_i: gb \in U_j \} \to {\rm GL}_d(\bb C)$ be the map defined by \eqref{def-sigma-ij}. 
Let $\| \cdot \|$ denote a fixed norm on $\bb C^d$. 
Then, for any section $s \in \bf \Gamma^{0}(E)$, any $i \in I$ and $b \in V_i$, we have
\begin{align}\label{identity phi P mu section}
\phi_i \circ P_{\mu} s (b) =  \phi_i  \left(  \int_{G}  g^{-1} s (g b)  \mu(dg)  \right)
= \int_{G}  \phi_i  \left( g^{-1} s (g b) \right)  \mu(dg), 
\end{align}
where $\phi_i$ is a continuous map from $\pi^{-1} (U_i)$ to $\bb C^d$. 
Using the finite covering $B = \bigcup_{j \in I} V_j$, we obtain 
\begin{align}\label{inequality P mu 001}
\| \phi_i \circ P_{\mu} s (b) \|
\leq \sum_{j \in I}    \int_{G}  \left\|   \phi_i \left( g^{-1} s (g b)   \right)  \right\| \mathds 1_{ \{ gb \in V_j \} } \mu(dg). 
\end{align}
Note that the formula \eqref{def-sigma-ij} says that, for any $i, j \in I$, $w \in \pi^{-1} (b)$ and $g \in G$ satisfying $gb \in U_j$, 
we have 
\begin{align*}
\phi_j (gw) = \sigma_{i, j} (g, b) \phi_i(w). 
\end{align*}
Applying this identity to $w = g^{-1} s (g b)$ gives 
\begin{align}\label{def-sigma-ij002}
\phi_i \big( g^{-1} s (g b) \big) =  \sigma_{i, j} (g, b)^{-1} \phi_j (s (g b)). 
\end{align}
Since the bundle action is isometric, every $\sigma_{i, j} (g, b)^{-1}$ is an isometry of $\bb C^d$.
Hence, substituting \eqref{def-sigma-ij002} into \eqref{inequality P mu 001}, 
 in view of \eqref{def-norm-s-infty}, we get 
\begin{align*}
\| \phi_i \circ P_{\mu} s (b) \| 
\leq  \sum_{j \in I}   \int_{G}   \left\|  \phi_j (s (g b))  \right\| \mathds 1_{ \{ gb \in V_j \} } \mu(dg)  
\leq |I| \|s\|_0.  
\end{align*}
Taking the supremum over $b \in V_i$ and then the maximum over $i \in I$, again by \eqref{def-norm-s-infty}, we obtain 
\begin{align}\label{boundedness-P-mu}
\|P_{\mu} s \|_0 \leq |I| \|s\|_0,
\end{align}
so that the operator $P_{\mu}$ is bounded on $\bf \Gamma^{0}(E)$. 

Next we assume that $\bb E ({\rm Lip}_E(g_1)^{\alpha} ) < \infty$ for some constant $\alpha>0$, 
and prove that, for any $\gamma \in (0, \alpha)$, 
the operator $P_{\mu}$ maps $\bf \Gamma^{\gamma}(E)$ to itself and is bounded. 
Fix $s \in \bf \Gamma^{\gamma}(E)$. 
For any $i \in I$ and $b, b' \in V_i$ with $b \neq b'$, by \eqref{identity phi P mu section} and splitting the integral into two parts, we have
\begin{align}\label{decom-Holder-norm-001}
& \| \phi_i \circ P_{\mu} s (b) - \phi_i \circ P_{\mu} s (b') \|  \notag\\
& = \left\|   \int_{G}  \left(  \phi_i  \left( g^{-1} s (g b) \right)  -  \phi_i  \left( g^{-1} s (g b') \right)  \right) \mu(dg)  \right\| \notag\\
& \leq \sum_{j \in I}    \int_{G}  \left\|  \phi_i  \left( g^{-1} s (g b) \right)  -  \phi_i  \left( g^{-1} s (g b') \right) \right\|
  \mathds 1_{\{(gb, gb') \in V_j \times V_j \}} \mu(dg)   \notag\\
  & \quad +   \int_{G}  \left\| \phi_i  \left( g^{-1} s (g b) \right)  -  \phi_i  \left( g^{-1} s (g b') \right) \right\|
  \mathds 1_{\{(gb, gb') \notin \bigcup_{j \in I} V_j \times V_j \}} \mu(dg).  
\end{align}
To handle the first term in \eqref{decom-Holder-norm-001}, 
using \eqref{def-sigma-ij002}, we derive that, for any $j \in I$ and $g \in G$ with $(gb, gb') \in V_j \times V_j$, 
\begin{align}\label{phi section difference 01}
& \left\|  \phi_i  \left( g^{-1} s (g b) \right)  -  \phi_i  \left( g^{-1} s (g b') \right) \right\|  \notag\\
& = \left\|  \sigma_{i, j} (g, b)^{-1} \phi_j (s (g b)) - \sigma_{i, j} (g, b')^{-1} \phi_j (s (g b')) \right\|  \notag\\
& \leq  \left\|  \sigma_{i, j} (g, b)^{-1} \left(  \phi_j (s (g b)) -  \phi_j (s (g b')) \right) \right\|  \notag\\
& \quad + \left\| \left(  \sigma_{i, j} (g, b)^{-1}  - \sigma_{i, j} (g, b')^{-1}  \right)\phi_j (s (g b')) \right\|  \notag\\
& \leq   \left\|   \phi_j (s (g b)) -  \phi_j (s (g b'))  \right\|
+ \| \sigma_{i, j} (g, b)  - \sigma_{i, j} (g, b') \| \|s\|_0,
\end{align}
where we have used the identity $x^{-1} (y - x) y^{-1} = x^{-1} - y^{-1}$ and the fact that the bundle action is isometric. 
Since $(b, b') \in V_i \times V_i$ and $(gb, gb') \in V_j \times V_j$,  we have, on one hand, by \eqref{def-norm-s-Holder} and \eqref{def-Lip-norm-g}, 
\begin{align}\label{phi section difference 02}
\left\|   \phi_j (s (g b)) -  \phi_j (s (g b'))  \right\| 
\leq \|s\|_{\gamma}  \mathbf d(gb, gb')^{\gamma}
\leq  \|s\|_{\gamma}  {\rm Lip}_E(g)^{\gamma}  \,  \mathbf d(b, b')^{\gamma}, 
\end{align}
and on the other hand, as $\sigma_{i, j}(g, b)$ is an isometry of $\bb C^d$ and again by \eqref{def-Lip-norm-g}, 
\begin{align}\label{phi section difference 03}
& \| \sigma_{i, j} (g, b)  - \sigma_{i, j} (g, b') \|  \notag\\
& = \| \sigma_{i, j} (g, b)  - \sigma_{i, j} (g, b') \|^{1-\gamma}  \| \sigma_{i, j} (g, b)  - \sigma_{i, j} (g, b') \|^{\gamma}  \notag\\
& \leq  2^{1-\gamma}   {\rm Lip}_E(g)^{\gamma}  \,  \mathbf d(b, b')^{\gamma}. 
\end{align}
Consequently, combining \eqref{phi section difference 01}, \eqref{phi section difference 02} and \eqref{phi section difference 03}, 
we obtain the upper bound for the first term in \eqref{decom-Holder-norm-001}: 
\begin{align}\label{bound-A1-001}
& \sum_{j \in I}    \int_{G}  \left\|  \phi_i  \left( g^{-1} s (g b) \right)  -  \phi_i  \left( g^{-1} s (g b') \right) \right\|
  \mathds 1_{\{(gb, gb') \in V_j \times V_j \}} \mu(dg)  \notag\\
& \leq   |I| \int_G  {\rm Lip}_E(g)^{\gamma} \mu(dg) \,  \left(  \|s\|_{\gamma} + 2^{1-\gamma} \|s\|_0 \right)  \mathbf d(b, b')^{\gamma}. 
\end{align}
Now we deal with the second term in \eqref{decom-Holder-norm-001}. 
As the set $\bigcup_{j \in I} (V_j \times V_j)$ is a neighborhood of the diagonal in $B \times B$, 
we may find $\ee >0$ such that for any $(u, u') \notin \bigcup_{j \in I} (V_j \times V_j)$, 
we have $\mathbf d(u, u') \geq \ee$. 
Note also that, by \eqref{def-norm-s-infty}, we have $\|  \phi_i  \left( g^{-1} s (g b) \right)  -  \phi_i  \left( g^{-1} s (g b') \right) \| \leq 2 \| s \|_0$. 
Therefore, using Markov's inequality and \eqref{def-Lip-norm-g}, we obtain 
\begin{align}\label{bound-A2-001}
& \int_{G}  \left\|  \phi_i  \left( g^{-1} s (g b) \right)  -  \phi_i  \left( g^{-1} s (g b') \right) \right\|
  \mathds 1_{\{(gb, gb') \notin \bigcup_{j \in I} V_j \times V_j \}} \mu(dg)  \notag\\ 
&  \leq  2 \| s \|_0 \   \mu \big( \{ g \in G: \mathbf d(gb, gb') \geq \ee \} \big)  \notag\\
& \leq   2  \ee^{-\gamma}  \| s \|_0     \int_{G}  \mathbf d(gb, gb')^{\gamma} \mu(dg) \notag\\
& \leq  2 \ee^{-\gamma}  \| s \|_0    \int_{G}   {\rm Lip}_E(g)^{\gamma}   \mu(dg)  \,  \mathbf d(b, b')^{\gamma}. 
\end{align}
Substituting \eqref{bound-A1-001} and \eqref{bound-A2-001} into \eqref{decom-Holder-norm-001}, taking into account \eqref{def-norm-s-Holder}, 
and using the assumption that $\bb E ({\rm Lip}_E(g_1)^{\alpha} ) < \infty$ for some $\alpha >0$, 
we get that, for any $\gamma \in (0, \alpha)$, the operator $P_{\mu}$ is bounded on $\bf \Gamma^{\gamma}(E)$. 

Finally, we assume that $(G, \mu)$ satisfies the contraction property \eqref{contraction-property-on-B} on $B$. 
For any $s \in \bf \Gamma^{\gamma}(E)$ and any $n \geq 1$, 
repeating the above computations leads to 
\begin{align*}
\| P_{\mu}^n s \|_{\gamma} 
& \leq |I| \|s\|_0
+ |I| \|s\|_{\gamma}  \sup_{\substack{b, b' \in B \\ b \neq b'}}  
 \frac{  \bb E \big( \mathbf d(g_n \cdots g_1 b, g_n \cdots g_1 b')^{\gamma} \big)  }{ \mathbf d(b, b')^{\gamma} }  \notag\\
& \quad + 2^{1-\gamma} |I| \|s\|_{0} \  \bb E  \big( {\rm Lip}_E(g_n \cdots g_1)^{\gamma}  \big)  \notag\\
 & \quad + 2 \ee^{-\gamma}  \| s \|_0 \   \sup_{\substack{b, b' \in B \\ b \neq b'}}  
 \frac{  \bb E  \big( \mathbf d(g_n \cdots g_1 b, g_n \cdots g_1 b')^{\gamma}  \big) }{ \mathbf d(b, b')^{\gamma} }. 
\end{align*}
Let $\delta \in (0, 1]$, $c>0$ and $a \in (0, 1)$ be constants as in \eqref{contraction-property-on-B}. 
Choosing $0 < \gamma < \min \{ \delta, \alpha \}$ and using H\"older's inequality, 
we get that, for any $b, b' \in B$ and $n \geq 1$, 
\begin{align*}
\bb E \big( \bf d (g_n \cdots g_1 b, g_n \cdots g_1 b')^{\gamma} \big) 
\leq \left( c a^n \bf d(b, b')^{\delta} \right)^{\gamma/\delta}
= c^{\frac{\gamma}{\delta}}  a^{ \frac{\gamma}{\delta} n}  \bf d(b, b')^{\gamma}, 
\end{align*} 
so that, by the submultiplicativity of ${\rm Lip}_E$, 
\begin{align*}
\| P_{\mu}^n s \|_{\gamma} 
 \leq \Big(  |I| + 2^{1-\gamma} |I| \Big[ \bb E \big( {\rm Lip}_E( g_1)^{\gamma} \big) \Big]^n 
 + 2 \ee^{-\gamma}  c^{\frac{\gamma}{\delta}}  a^{ \frac{\gamma}{\delta} n} \Big) \|s\|_0  
  +  c^{\frac{\gamma}{\delta}}  |I|  a^{\frac{\gamma}{\delta}n} \|s\|_{\gamma}. 
\end{align*}
By Proposition \ref{Sobolev-injection-lemma}, 
the natural injection $\boldsymbol{\Gamma}^{\gamma}(E) \to \boldsymbol{\Gamma}^{0}(E)$ is a compact operator. 
Hence, by the Ionescu-Tulcea and Marinescu theorem (see \cite{ITM50} and \cite{BQ16b}), the essential spectral radius
of the operator $P_{\mu}$ acting on the space $\bf \Gamma^{\gamma}(E)$ is at most $a^{\frac{\gamma}{\delta}}$. 
This completes the proof of the theorem. 
\end{proof}

\begin{corollary}\label{Corollary-001}
Let $B$ be a compact metric space and let $\pi: E \to B$ be a 
Lipschitz vector bundle of dimension $d$ over $B$. 
Let $G$ be a locally compact group acting on $E$ by an isometric Lipschitz bundle action
and let $\mu$ be a Borel probability measure on $G$. 
Assume that there exists a constant $\alpha >0$ such that $\bb E ({\rm Lip}_E(g_1)^{\alpha} ) < \infty$ 
and that the pair $(G, \mu)$ satisfies the contraction property \eqref{contraction-property-on-B} on $B$. 
Then, for each $b \in B$, there exists a continuous linear map $\zeta_b: \bf \Gamma^{0}(E) \to \pi^{-1} (b)$
such that, for any section $s \in \bf \Gamma^{0}(E)$, we have  
\begin{align*}
\frac{1}{n} \sum_{k=0}^{n-1} P_{\mu}^k s(b)  \xrightarrow[n\to\infty]{}  \zeta_b(s). 
\end{align*}
Moreover, this convergence holds uniformly over $b \in B$. 
\end{corollary}

\begin{proof}
We first observe that the estimate established in the proof of Theorem \ref{Thm-spectral-gap} 
can be extended to every iterate of the operator. 
More precisely, with the notation introduced there, following the proof of \eqref{boundedness-P-mu} actually gives that, 
for any $n \geq 0$ and any continuous section $s \in \bf \Gamma^{0}(E)$, we have 
\begin{align}\label{P-mu-n-001}
\| P_{\mu}^n s \|_{0} \leq |I| \|s\|_0. 
\end{align}
By Theorem \ref{Thm-spectral-gap}, we may find a constant $\gamma>0$ such that $P_{\mu}$ is bounded on $\bf \Gamma^{\gamma}(E)$
and has essential spectral radius strictly less than $1$. 
We then apply the spectral decomposition results from \cite[Section B.4, Proposition B.14]{BQ16b} 
to obtain a direct sum decomposition 
\begin{align*}
\bf \Gamma^{\gamma}(E) = X \oplus Y, 
\end{align*}
where $X$ and $Y$ are closed $P_{\mu}$-invariant subspaces,
the spectral radius of $P_{\mu}$ restricted to $X$ is strictly less than $1$,
and $Y$ is finite-dimensional and the spectrum of $P_{\mu}$ on $Y$ lies on the unit circle. 
 In particular, for any $s \in X$, we have $P_{\mu}^n s  \xrightarrow[n\to\infty]{}  0$.  
Now we decompose the finite-dimensional space $Y$ as a direct sum $Y = \bigoplus_{\lambda} Y_{\lambda}$ of generalized eigenspaces
associated with the eigenvalues of $P_{\mu}$. For each $\lambda$, we may find a nilpotent linear endomorphism $N_{\lambda}$ of $Y_{\lambda}$ such that 
the restriction of $P_{\mu}$ to $Y_{\lambda}$ is equal to $\lambda e^{N_{\lambda}}$.
Thus, for any $n\geq 0$, the restriction of $P_{\mu}^n$ to $Y_{\lambda}$ is equal to $\lambda^n  e^{n N_{\lambda}}$. 
On the other hand, due to \eqref{P-mu-n-001}, we get $\sup_{n \geq 0} \| e^{n N_{\lambda}} \| < \infty$ and therefore, since the map $t \mapsto e^{t N}$
is polynomial in $t$ and bounded on integers, it must be constant and hence $N_{\lambda} = 0$. 
It follows that the restriction of $P_\mu$ to $Y_\lambda$ reduces to multiplication by $\lambda$. 
This implies that, if $\lambda \neq 1$, then for any $s \in Y_{\lambda}$, 
it holds 
\begin{align*}
\frac{1}{n} \sum_{k=0}^{n-1} P_{\mu}^k s = \frac{1}{n} \frac{\lambda^n - 1}{\lambda -1} s  \xrightarrow[n\to\infty]{}  0,
\end{align*}
whereas, if $\lambda = 1$, then for any $s \in Y_{1}$,  it holds $\frac{1}{n} \sum_{k=0}^{n-1} P_{\mu}^k s = s$.  
We conclude that, for any $s \in \bf \Gamma^{\gamma}(E)$, as $n \to \infty$, 
the Birkhoff sum $\frac{1}{n} \sum_{k=0}^{n-1} P_{\mu}^k s$ converges in $\bf \Gamma^{\gamma}(E)$ towards the component of $s$ in $Y_1$, 
under the decomposition $\bf \Gamma^{\gamma}(E) = X \oplus \bigoplus_{\lambda} Y_{\lambda}$. 

Now let us fix $b \in B$ and denote by $e_b: \bf \Gamma^{0}(E) \to \pi^{-1} (b)$ the evaluation map, which is continuous and linear. 
We choose a norm on the finite-dimensional vector space $\pi^{-1} (b)$. 
For any $s \in \bf \Gamma^{0}(E)$ and $\ee >0$, 
Proposition \ref{Sobolev-injection-lemma} yields that there exists a $\gamma$-H\"older continuous section 
$s' \in \bf \Gamma^{\gamma}(E)$ such that $\|s - s'\|_0 < \ee$. 
For any $n \geq 1$, we estimate 
\begin{align*}
\left\| \frac{1}{n} \sum_{k=0}^{n-1} P_{\mu}^k s(b) - \frac{1}{n} \sum_{k=0}^{n-1} P_{\mu}^k s'(b)  \right\|
& = \left\|  \frac{1}{n} \sum_{k=0}^{n-1} e_b  P_{\mu}^k (s-s') \right\|  \notag\\
& \leq \|e_b\|   \left\|  \frac{1}{n} \sum_{k=0}^{n-1}   P_{\mu}^k \right\| \|s -s'\|_0  \notag\\
& \leq |I| \|e_b\| \ee, 
\end{align*}
where we have used the inequality \eqref{P-mu-n-001}. 
Since $\frac{1}{n} \sum_{k=0}^{n-1} P_{\mu}^k s'(b)$ converges as $n \to \infty$, 
the estimate above shows that for any $s \in \bf \Gamma^{0}(E)$, 
the Birkhoff sum $\frac{1}{n} \sum_{k=0}^{n-1} P_{\mu}^k s(b)$ also converges as $n \to \infty$.  
We denote its limit by $\zeta_b(s)$. 
The map $\zeta_b: \bf \Gamma^{0}(E) \to \pi^{-1} (b)$ is linear by construction,
and its continuity follows from \eqref{P-mu-n-001}. 
Finally, by refining the above arguments, we obtain that, for each $i \in I$, the convergence is uniform on $V_i$, and consequently uniform on $B$. 
\end{proof}

We will actually need a version of Theorem \ref{Thm-spectral-gap} for a more general family of operators, 
which can be defined by perturbing $P_{\mu}$ by means of a cocycle.  

Let $\sigma_0: G \times B \to \bb R$ be a continuous cocycle such that, for every $g \in G$, 
the function $\sigma_0(g, \cdot)$ is Lipschitz continuous on $B$.
For $g \in G$, we set 
\begin{align}\label{sigma-Lip-g}
\sigma_{\sup}(g) = \sup_{b \in B} |\sigma_0(g, b)|
\quad  \mbox{and}   \quad
\sigma_{\rm Lip}(g) 
 = \sup_{\substack{b, b' \in B \\ b \neq b'}} \frac{|\sigma_0(g, b) - \sigma_0(g, b')|}{\bf d (b, b')}. 
\end{align}
Throughout this subsection, we assume that there exists a constant $\alpha>0$ such that 
\begin{align}\label{moment-assumption-02-a}
\int_{G} e^{\alpha \sigma_{\sup}(g)} \mu(dg) < \infty
\quad  \mbox{and}   \quad
\int_{G} \sigma_{\rm Lip}(g)^{\alpha} \mu(dg) < \infty. 
\end{align}
For any $z \in \bb C$ with $|\Re(z)| < \alpha$ and any continuous section $s: B \to E$,
we define
\begin{align*}
P_{\mu, z} s (b) = \int_{G}  e^{z \sigma_0(g, b)} g^{-1} s (g b) \mu(dg), 
\quad b \in B. 
\end{align*}
The integrability assumption \eqref{moment-assumption-02-a} ensures that the integral is well defined. Moreover, 
$P_{\mu, z} s$ is again a continuous section of $E$. 

\begin{corollary}\label{Corollary-Spectral-gap-002}
Let $B$ be a compact metric space and let $\pi: E \to B$ be a 
Lipschitz vector bundle of dimension $d$ over $B$. 
Let $G$ be a locally compact group acting on $E$ by an isometric Lipschitz bundle action 
and let $\mu$ be a Borel probability measure on $G$. 
Assume that there exists a constant $\alpha >0$ such that $\bb E ({\rm Lip}_E(g_1)^{\alpha} ) < \infty$. 
Let $\sigma_0: G \times B \to \bb R$ be a continuous cocycle satisfying \eqref{moment-assumption-02-a}.
Then, for every $\alpha' \in (0, \alpha)$ and $\gamma >0$ close enough to $0$, 
the operator $P_{\mu, z}$ is bounded on $\bf \Gamma^{\gamma}(E)$, 
and the map $z \mapsto P_{\mu, z}$
with values in the space of bounded operators on $\bf \Gamma^{\gamma}(E)$ 
is analytic on the strip $\{ z \in \bb C: |\Re(z)| < \alpha' \}$. 

Moreover, if $(G, \mu)$ satisfies the contraction property \eqref{contraction-property-on-B} on $B$,
then, for every $t \in \bb R$, 
the operator $P_{\mu, it}$ has essential spectral radius strictly less than $1$ on $\bf \Gamma^{\gamma}(E)$. 
\end{corollary}

\begin{proof}
The boundedness of the operator $P_{\mu, z}$ on the Banach space $\bf \Gamma^{\gamma}(E)$ for small enough $\gamma$,
along with its analytic dependence on the complex parameter $z$,
can be obtained by expanding the arguments used in the proof of Theorem \ref{Thm-spectral-gap}. 
This adaptation is analogous to the standard treatment of group actions on function spaces; for a detailed reference, see \cite[Lemma 11.16]{BQ16b}.

 To establish the spectral gap property, we introduce a $z$-dependent deformation of the original bundle action. 
For each $z \in \bb C$, we define a new action of $G$ on the total space $E$ by setting,
for every $g \in G$ and $w \in E$, 
\begin{align}\label{def-new-action-001}
g \underset{z}{\cdot} w = e^{-z \sigma_0 (g, \pi w)} gw,
\end{align}
where $\pi: E \to B$ is the projection of the bundle. 
We now check that this indeed defines a group action. 
Since $\sigma_0: G \times B \to \bb R$ is a cocycle,
using \eqref{property-group-action} we have the following associativity property: 
for any $g_1, g_2 \in G$ and $w \in E$, 
\begin{align*}
g_1 \underset{z}{\cdot} (g_2  \underset{z}{\cdot} w) 
& =  g_1 \underset{z}{\cdot} \left( e^{-z \sigma_0 (g_2, \pi w)} g_2 w \right)  \notag\\
& =  e^{-z \sigma_0 (g_1, \pi (e^{-z \sigma_0 (g_2, \pi w)} g_2 w) )}  g_1 \left( e^{-z \sigma_0 (g_2, \pi w)} g_2 w \right) \notag\\
& =  e^{-z \sigma_0 (g_1, g_2 \pi (w) )}  e^{-z \sigma_0 (g_2, \pi w)}  g_1 g_2 w  \notag\\
& =  e^{-z \sigma_0 (g_1g_2, \pi w )}  g_1 g_2 w \notag\\
& = (g_1 g_2)  \underset{z}{\cdot} w,
\end{align*}
so that we actually have defined an action. 

This action is related to the operator $P_{\mu, z}$ as follows. 
Since $\pi (s(b')) = b'$ for every $b' \in B$, we have the following identity: for any section $s \in \bf \Gamma^{0}(E)$ and any $b \in B$,  
\begin{align*}
\int_{G}   g^{-1} \underset{z}{\cdot} s (g b) \mu(dg)
& =  \int_{G}  e^{- z \sigma_0(g^{-1}, \pi(s(gb)))} g^{-1} s (g b) \mu(dg)  \notag\\
& =  \int_{G}  e^{- z \sigma_0(g^{-1}, gb)} g^{-1} s (g b) \mu(dg)  \notag\\
& = \int_{G}  e^{z \sigma_0(g, b)} g^{-1} s (g b) \mu(dg)  \notag\\
& = P_{\mu, z} s (b). 
\end{align*}

Therefore, to conclude the proof of the corollary, it suffices to check that, for any $z \in i \bb R$ (purely imaginary), 
this new action satisfies all the assumptions of Theorem \ref{Thm-spectral-gap}. 
Indeed, we claim that, for any $z \in i \bb R$, this new action constitutes again an isometric Lipschitz bundle action on $E$. 
To verify this, we fix a finite local trivialization $(U_i, \phi_i, \theta_{i, j})_{i, j \in I}$ 
and as usual, for each $i \in I$, 
choose an open subset $V_i \subset U_i$ such that $\overline{V_i} \subset U_i$ and 
we still have $B = \bigcup_{i \in I} V_i$. 
For indices $i, j \in I$, we let $\sigma_{i, j}: \{ (g, b) \in G \times U_i: gb \in U_j \} \to U(d)$ be the transition maps satisfying \eqref{def-sigma-ij}. 
Using \eqref{def-new-action-001} and \eqref{def-sigma-ij},
we get, for any $z \in i \bb R$, $g \in G$, $v \in \bb C^d$ and $b \in U_i \cap  (g^{-1} U_j)$,   
\begin{align*}
g \underset{z}{\cdot} \left( \pi \times \phi_i \right)^{-1} (b, v) 
& = e^{-z \sigma_0(g, b)} g \left( \pi \times \phi_i \right)^{-1} (b, v)  \notag\\
& = e^{-z \sigma_0(g, b)} \left( \pi \times \phi_j \right)^{-1} (gb, \sigma_{i, j} (g, b) v) \notag\\
& = \left( \pi \times \phi_j \right)^{-1} \left(gb, e^{-z \sigma_0(g, b)} \sigma_{i, j} (g, b) v \right), 
\end{align*}
where the last equality follows from \eqref{fiber-vector-space-001}. 
This calculation shows that the new action is an isometric bundle action on $E$.  
Note that the associated map $\sigma_{i, j}^z: \{ (g, b) \in G \times U_i: gb \in U_j \} \to U(d)$ 
is defined by the product formula $\sigma_{i, j}^z(g, b) = e^{-z \sigma_0(g, b)} \sigma_{i, j} (g, b)$. 
Since $\sigma_{i, j}$ is locally Lipschitz in $b$, so is $\sigma_{i, j}^z$ and consequently, 
the new action $\underset{z}{\cdot}$ is also Lipschitz on $E$. 
Finally, it remains to verify that the moment assumption holds for this new action. 
For $g \in G$, 
we need to compute the new Lipschitz constant ${\rm Lip}_E^z(g)$ associated to the new action. 
By the definition given in \eqref{def-Lip-norm-g}, we have
\begin{align*}
 {\rm Lip}_E^z(g) 
&  = \sup_{\substack{b, b' \in B \\ b \neq b'}} \frac{\bf d (gb, gb')}{\bf d (b, b')}   \notag\\
& \quad + \max_{i, j \in I} 
  \sup_{\substack{b, b' \in  V_i  \cap  (g^{-1} V_j)    \\ b \neq b'}}   
  \frac{\| e^{-z \sigma_0(g, b)} \sigma_{i, j}(g, b) - e^{-z \sigma_0(g, b')} \sigma_{i, j}(g, b') \|}{\bf d (b, b')}. 
\end{align*}
For any $i, j \in I$ and $b, b' \in V_i  \cap  (g^{-1} V_j)$, 
by \eqref{sigma-Lip-g} and the fact that all the linear maps $\sigma_{i, j}(g, \cdot)$ are isometries of $\bb C^d$, 
we derive that 
\begin{align*}
& \| e^{-z \sigma_0(g, b)} \sigma_{i, j}(g, b) - e^{-z \sigma_0(g, b')} \sigma_{i, j}(g, b') \| \notag\\
& \leq \| e^{-z \sigma_0(g, b)} (\sigma_{i, j}(g, b) -  \sigma_{i, j}(g, b') ) \| + \|(e^{-z \sigma_0(g, b)} - e^{-z \sigma_0(g, b')}) \sigma_{i, j}(g, b') \| 
 \notag\\
& = \| \sigma_{i, j}(g, b) -  \sigma_{i, j}(g, b')  \| + \|e^{-z \sigma_0(g, b)} - e^{-z \sigma_0(g, b')} \|  \notag\\
& \leq  \| \sigma_{i, j}(g, b) -  \sigma_{i, j}(g, b')  \| +  \sigma_{\rm Lip}(g) |z| \bf d (b, b'),
\end{align*}
which, together with \eqref{def-Lip-norm-g}, yields 
\begin{align*}
{\rm Lip}_E^z(g)  \leq  {\rm Lip}_E(g) +  \sigma_{\rm Lip}(g) |z|. 
\end{align*}
Consequently, by the integrability assumption \eqref{moment-assumption-02-a}, we obtain 
\begin{align*}
\bb E ({\rm Lip}_E^z(g_1)^{\alpha})  \leq  c_{\alpha} \left( \bb E ({\rm Lip}_E(g_1)^{\alpha})  + \bb E  (\sigma_{\rm Lip}(g_1)^{\alpha}) |z|^{\alpha} \right) < \infty,
\end{align*}
which verifies the required moment condition. 
Hence, all the assumptions of Theorem \ref{Thm-spectral-gap} are satisfied for $z \in i \bb R$, completing the proof of the corollary. 
\end{proof}


\section{Principal bundles and representation theory}\label{Section principal bundles}

In this section, 
we give a general construction of vector bundles which will be used later in the present article. 
This construction requires the use of the formalism of principal bundles; 
for a comprehensive introduction to this topic, see for example \cite{Die08, Hus94}.  
Here we shall briefly recall the essential aspects of this theory, focusing specifically on the properties 
which we will need in the context of Lipschitz regularity of certain associated maps. 

\subsection{Free actions of compact groups}\label{subsec-free-action}
Let $\wt B$ be a compact metric space, equipped with a distance $\wt{\bf d}$ 
and endowed with a free right action by isometries of a compact metrizable group $M$. 
Recall that the action being free means that for any $m \in M$ and $\wt b \in \wt B$, the equality $\wt b m = \wt b$ implies $m = e$, 
where $e$ denotes the neutral element of $M$.

As the action of $M$ on $\wt B$ is free, there is a well-defined map from the graph of the action $\{(\wt b, \wt b m): \wt b \in \wt B, m \in M\}$
towards the group $M$, which maps the pair $(\wt b, \wt b m)$ to the group element $m \in M$. 
This map is called the {\it translation map} of the bundle
and it is easily seen to be continuous. 
For our purposes, however, we will actually need to use a stronger Lipschitz continuity property of this map, which holds when $M$ is a compact Lie group.

Indeed, we can define a distance $D$ on $M$ as follows. 
Since $M$ is a compact Lie group, it admits a faithful linear representation (see for example \cite{Sim96}). 
We choose such a faithful linear representation $\iota: M \to {\rm GL}_d(\bb C)$ for some integer $d \geq 1$, 
and equip $\bb C^d$ with an $M$-invariant norm $\| \cdot \|$. 
For any $m, m' \in M$, we then set 
\begin{align*}
D(m, m') = \| \iota(m) - \iota(m') \|. 
\end{align*}
Although this distance depends on the chosen representation and the norm, 
its Lipschitz equivalence class is independent of these choices, as follows from the next result. 
For notational simplicity, we shall denote by $\bf d$ the distance on $\wt B \times \wt B$ induced by the distance on $\wt B$,
whenever no confusion can arise. 

\begin{lemma}\label{Lem-structure}
Let $M$ be a compact Lie group acting freely on the right by isometries on a compact metric space $\wt B$.
Assume that the graph of the action is equipped with the restriction of the product distance $\bf d$, 
and that $M$ is equipped with the distance $D$. 
Then, the translation map of the action is Lipschitz continuous, which is to say that 
there exists a constant $c>0$ such that, for all $\wt b, \wt b' \in \wt B$ and $m, m' \in M$, 
\begin{align*}
D (m, m') \leq c \,  \bf d \Big( (\wt b, \wt b m), (\wt b',  \wt b' m') \Big). 
\end{align*}
\end{lemma}

\begin{proof}
The proof relies on classical linearization arguments, see Gleason \cite{Gle50} and Palais \cite{Pal61}.  
We will actually prove that the translation map is locally Lipschitz continuous. The conclusion will then follow by compactness.

We begin by constructing an example of a free right $M$-action with locally Lipschitz continuous translation map. 
We assume that $M$ is a closed subgroup of ${\rm GL}_d(\bb R)$ for some integer $d \geq 1$. 
Let $\mathfrak m \subset \mathfrak{gl}_d(\bb R)$ be the Lie algebra of $M$, 
where we view $\mathfrak{gl}_d(\bb R)$ as the space of square matrices with size $d$, 
 and choose a linear subspace $\mathfrak v \subset \mathfrak{gl}_d(\bb R)$
such that $\mathfrak{gl}_d(\bb R) = \mathfrak v \oplus \mathfrak m$. 
By the inverse function theorem, 
we may find an open neighborhood $W_0$ of the identity matrix $e = 1$ in $M$, and an open neighborhood $V_0$ of $0$ in $\mathfrak{v}$ 
such that the map 
\begin{align*}
(X, m) \longmapsto (1 + X)m
\end{align*}
is a smooth diffeomorphism from $V_0 \times W_0$ onto an open subset of $\mathfrak{gl}_d(\bb R)$,
which we can assume only to contain invertible matrices. 
Let $V_1 \subset V_0$ be an open neighborhood  of $0$ in $\mathfrak v$
such that 
\begin{align*}
\big( (1 + V_1)^{-1} (1 + V_1) \big) \cap M \subset W_0. 
\end{align*}

We claim that, for any $X \in (1 + V_1) M$, the intersection $(XM) \cap (1 + V_1)$ contains exactly one element. 
Indeed, suppose that there are two elements $Y_1, Y_2 \in V_1$ such that both $1 + Y_1$ and $1 + Y_2$ lie in $XM$.  
Then, as $1 + Y_1$ and $1 + Y_2$ are in the right $M$-orbit of $X$, we may find $m \in M$ satisfying $1 + Y_2 =  (1 + Y_1) m$. 
Hence we get $m = (1 + Y_1)^{-1} (1 + Y_2) \in W_0$. 
Since the map $(X, m) \mapsto (1 + X)m$ is a bijection from $V_0 \times W_0$ onto $(1 + V_0) W_0$, 
we must have $m = e$ and $Y_1 = Y_2$, proving the claim. 

Consequently, the map 
\begin{align*}
V_1 \times M  &  \longrightarrow  (1 + V_1)M  \notag\\
(Y, m) & \longmapsto  (1 + Y)m
\end{align*}
 is a bijection. 
Moreover, this map is a smooth diffeomorphism in the neighborhood of $V_1 \times \lbrace1 \rbrace$, 
and by $M$-equivariance, it is also a (global) diffeomorphism $V_1 \times M \to (1 + V_1)M$.  
This implies that the translation map associated with the right action of $M$ on $(1 + V_1)M$ is smooth and therefore, locally Lipschitz continuous. 

Fix an arbitrary point $\wt b_0 \in \wt B$.
We will use the previous construction to prove that the translation map of $\wt B$ is Lipschitz continuous on a neighborhood of $(\wt b_0, \wt b_0 M)$. 
Let $f: \wt B \to \bb R_+$ be a nonnegative Lipschitz continuous function 
with support contained in a neighborhood of $\wt b_0$ in $\wt B$ such that  
\begin{align*}
\int_{M} f(\wt b_0 m) dm = 1, 
\end{align*}
where $dm$ denotes the normalized Haar measure on $M$. 
For each $\wt b \in \wt B$, we define 
$$F(\wt b) = \int_{M} f(\wt b m^{-1}) m dm,$$ 
which is an element of $\mathfrak{gl}_d(\bb R)$. 
Up to shrinking the support of the function $f$, we may assume that $F(\wt b_0) \in (1 + V_1)M$. 
Notice that $F$ is a Lipschitz continuous map from $\wt B$ to $\mathfrak{gl}_d(\bb R)$ and 
it is right $M$-equivariant: for any $\wt b \in \wt B$
and $m \in M$, we have $F(\wt b m) = F(\wt b)m$. 
We finally choose an open neighborhood $U_0$ of $\wt b_0$ in $\wt B$ such that $F(U_0) \subset  (1+V_1)M$. 
By $M$-equivariance, the translation map of $U_0 M$ may be written as the composition of the translation map of $(1+ V_1)M$ with the restriction of $F \times F$ 
to the graph of the $M$-action on $U_0 M$. 
Therefore, the translation map of $U_0 M$ is locally Lipschitz continuous, as claimed. 
\end{proof}


\subsection{Principal bundles}\label{Section-Quotient vector bundles}
Let $\wt B$ still be a compact metric space with metric $\wt{\bf d}$, and 
let $M$ be a metrizable compact topological group acting freely and isometrically on $\wt B$. 
We define the quotient space 
\begin{align*}
B =  \wt B / M = \left\{  \wt bM:  \wt b \in \wt B  \right\}, 
\end{align*}
that is, the space of all right $M$-orbits in $\wt B$. 
We equip $B$ with the quotient topology, 
then $B$ is a compact Hausdorff space and its topology 
can be induced by a metric $\bf d$ given as follows: 
 for $b =  \wt b M \in B$ and $b' =  \wt b' M \in B$,
\begin{align}\label{def distance on B}
\bf d (b, b') = \inf_{m, m' \in M} \wt{\bf d} (\wt b m, \wt b' m') = \inf_{m \in M} \wt{\bf d} (\wt b m, \wt b'). 
\end{align}

Define by $\eta: \wt B \to B$ the quotient map $\wt b \mapsto  \wt b M$.  
We say that $\eta: \wt B \to B$ is a {\it Lipschitz $M$-principal bundle}
if there exist local Lipschitz continuous sections, meaning that,  
for every point $b \in B$, there exists an open neighborhood $U$ of $b$ in $B$ and a Lipschitz continuous map $\sigma: U \to \wt B$
such that, for every $b' \in U$, we have $\eta(\sigma(b')) = b'$.

It is important to emphasize that, in general, not every free action of a compact group $M$ gives rise to an $M$-principal bundle; 
counterexamples can be found in the literature (see \cite{Kol37}). 
Nevertheless, this is the case for smooth actions of Lie groups. 

\begin{lemma}\label{Lem-principal-bundle}
Let $M$ be a compact Lie group and let $\wt B$ be a compact smooth manifold equipped with a smooth free right action of $M$.
Then, the quotient map $\eta: \wt B \to \wt B / M$ is a Lipschitz $M$-principal bundle. 
\end{lemma}

\begin{proof}
The orbit space $B = \wt B / M$ admits a unique smooth manifold structure 
for which the quotient map $\eta: \wt B \to \wt B / M$ is a submersion, see \cite[Theorem 2.9.10]{Var84}. 
By the implicit function theorem, this map admits local smooth sections. 
\end{proof}

We now fix a linear representation of the group $M$, that is, 
a continuous homomorphism $\rho: M \to {\rm GL}_d(\bb C)$ for some integer $d \geq 1$. 
The space $\wt B \times \bb C^d$ is equipped with the product right $M$-action:
for any $\wt b \in \wt B$, $v \in \bb C^d$ and $m \in M$, we have 
\begin{align}\label{def-action-M}
(\wt b, v) m = \left( \wt b m, \rho (m)^{-1} v \right). 
\end{align}
Let $\wt \pi: \wt B \times \bb C^d \to \wt B$ be the projection onto the first component,
which, by definition, commutes with the action of $M$. 
We now consider the quotient space 
\begin{align*}
B \times_{\rho} \bb C^d :=  (\wt B \times \bb C^d) / M, 
\end{align*}
which will be equipped with the quotient topology. 
Since the map $\wt \pi$ commutes with the action of $M$, there exists a unique continuous map 
$\pi: B \times_{\rho} \bb C^d \to B$ such that for every $\wt b \in \wt B$ and $v \in \bb C^d$, 
we have
\begin{align}\label{def-pi-M}
\pi ((\wt b, v) M)  = \wt \pi  (\wt b, v) M = \wt b M. 
\end{align}
More explicitly, the situation is described in the following commutative diagram: 
\begin{equation}\label{commutative diagram}
\begin{CD}
 \wt B \times \bb C^d  @>\wt{\eta}>>   B \times_{\rho} \bb C^d =  (\wt B \times \bb C^d) / M  \\
@VV\wt{\pi}V    @VV\pi V \\
\wt B    @>\eta>>  B = \wt B / M
\end{CD}
\end{equation}
where we have denoted by $\wt{\eta}$ and $\eta$ 
the quotient maps $\wt B \times \bb C^d \to B \times_{\rho} \bb C^d$ and $\wt B \to B$, respectively.

The following lemma allows us to define a complex vector space structure on each fiber of the projection $\pi$. 

\begin{lemma}
Let $M$ be a compact metrizable group acting freely on the compact space $\wt B$. 
Let $\rho: M \to {\rm GL}_d(\bb C)$ be a continuous linear representation for some integer $d \geq 1$. 
Then, for every $b \in B = \wt B / M$, there exists a unique complex vector space structure
on the fiber $\pi^{-1}(b)$ such that, for any $\wt b \in \wt B$ with $\wt b M = b$,
the map 
\begin{align*}
\psi_{\wt b}:  \bb C^d  &    \longrightarrow  \pi^{-1}(b)  \notag\\
v & \longmapsto  (\wt b, v) M
\end{align*}
is a linear isomorphism.
\end{lemma}

\begin{proof}
Fix $b \in B$ and choose $\wt b \in \wt B$ such that $\wt b M = b$. 
We first claim that the map $\psi_{\wt b}: \bb C^d  \to \pi^{-1}(b)$ is a bijection. 
Indeed, it is surjective by the definition of the map $\pi$. 
To prove injectivity, suppose that there are two vectors $v, v' \in \bb C^d$ satisfying $\psi_{\wt b}(v) = \psi_{\wt b}(v')$,
then, by \eqref{def-action-M}, we may find $m \in M$ such that $(\wt b, v') = (\wt b, v) m = (\wt b m, \rho(m)^{-1} v)$. 
Comparing the first components gives $\wt b = \wt b m$. 
Since the action of $M$ on $\wt B$ is free, this implies that $m$ equals the identity element of $M$. Hence $v = v'$
and $\psi_{\wt b}$ is injective. 
Therefore, the map $\psi_{\wt b}$ is a bijection from $\bb C^d$ onto $\pi^{-1}(b)$.  

Besides, for any $\wt b' = \wt b m$ with $m \in M$ and $v \in \bb C^d$, in view of \eqref{def-action-M}, we have
\begin{align*}
\psi_{\wt b'}(v) = (\wt b', v) M = (\wt b m, v) M = (\wt b, \rho(m) v) M = (\psi_{\wt b} \circ \rho(m))(v), 
\end{align*}
so that $\psi_{\wt b'} = \psi_{\wt b} \circ \rho(m).$
Therefore, the vector space structure on $\pi^{-1}(b)$ which is obtained by transporting the one on $\bb C^d$
through the bijection $\psi_{\wt b}$ does not depend on the choice of $\wt b$,
which is the meaning of the lemma. 
\end{proof}

When $\eta: \wt B \to B$ is a Lipschitz $M$-principal bundle,  
we will refine the above result by constructing a Lipschitz vector bundle structure on the map $\pi: B \times_{\rho} \bb C^d \to B$.

Now we define our vector bundle structure. 
We let $I$ be the set of all local Lipschitz continuous sections of the map $\eta: \wt B \to B$.
More precisely, for each $i \in I$, we are given an open subset $U_i$ of $B$ and a Lipschitz continuous map 
$\sigma_i: U_i \to \wt B$ such that $\eta \circ \sigma_i$ is the identity map of $U_i$, that is, for every $b \in U_i$, 
\begin{align}\label{def-sigma-i}
\eta(\sigma_i(b)) = b.
\end{align}
By definition,  the family $(U_i)_{i \in I}$ forms an open covering of $B$. 
Our goal is to construct continuous maps $\phi_i: \pi^{-1} (U_i) \to \bb C^d$.
To this end, observe that the preimage $\pi^{-1} (U_i) \subset  B \times_{\rho} \bb C^d = (\wt B \times \bb C^d) / M$ 
is naturally the quotient of the space $\eta^{-1} (U_i) \times \bb C^d$ by the product $M$-action.  
This is summarized by the following commutative diagram: 
$$
\begin{CD}
\eta^{-1} (U_i) \times \bb C^d   @>\wt{\eta}>>   \pi^{-1} (U_i)  \\
@VV\wt{\pi}V    @VV\pi V \\
\eta^{-1} (U_i)    @>\eta>>  U_i
\end{CD}
$$
Consequently, defining a continuous map $\phi_i: \pi^{-1} (U_i) \to \bb C^d$ amounts to defining an $M$-invariant continuous map 
$\wt{\phi_i}: \eta^{-1} (U_i) \times \bb C^d \to \bb C^d$:  
$$
\begin{CD}
\bb C^d  @<\wt{\phi_i}<<  \eta^{-1} (U_i) \times \bb C^d   @>\wt{\eta}>>   \pi^{-1} (U_i)  @>\phi_i>>  \bb C^d  \\
@.   @VV\wt{\pi}V    @VV\pi V \\
@.   \eta^{-1} (U_i)    @>\eta>>  U_i
\end{CD}
$$
We now proceed to define the required invariant map $\wt{\phi_i}$. 
For any $\wt b \in \eta^{-1}(U_i)$, 
the freeness of the $M$-action on $\wt B$ guarantees a unique element $\mu_i(\wt b) \in M$ such that 
\begin{align}\label{existence-mu-i}
\wt b =  \sigma_i (\eta(\wt b)) \mu_i (\wt b),
\end{align}
where $\sigma_i: U_i \to \wt B$ is given by \eqref{def-sigma-i}. 
Geometrically, this means that $\mu_i (\wt b)$ is the unique group element 
that moves the reference section point $\sigma_i (\eta(\wt b))$ to the actual point $\wt b$. 
The situation is illustrated in Figure \ref{the-map-mu-i}. 
Since the translation map of the action of $M$ on $\wt B$ is continuous (see Subsection \ref{subsec-free-action}), 
the map $\mu_i: \eta^{-1}(U_i) \to M$ is continuous as well. 
Furthermore, for any $\wt b \in \eta^{-1} (U_i)$ and $m \in M$, since $\eta(\wt b m) = \eta(\wt b)$, 
by the uniqueness property in \eqref{existence-mu-i}, it follows that 
\begin{align}\label{property-mu-m-invariant}
\mu_i (\wt b m) = \mu_i (\wt b) m. 
\end{align}
With this in hand, 
we define the continuous map $\wt{\phi_i}: \eta^{-1} (U_i) \times \bb C^d \to \bb C^d$ by setting, 
for each $\wt b \in \eta^{-1}(U_i)$ and $v \in \bb C^d$,  
\begin{align}\label{def-wt-phi-i}
\wt{\phi_i}(\wt b, v) = \rho(\mu_i(\wt b)) v. 
\end{align} 
Then it remains to verify that $\wt{\phi_i}$ is $M$-invariant. 
Indeed, for any $m \in M$, using \eqref{property-mu-m-invariant}, \eqref{def-wt-phi-i} and the fact that 
$\rho$ is a group homomorphism, we compute  
\begin{align*}
\wt{\phi_i}(\wt b, \rho(m)^{-1} v) 
&  = \rho(\mu_i(\wt b m)) \rho(m)^{-1} v  \notag\\
& =  \rho(\mu_i( \wt b) m)  \rho(m)^{-1} v  
 = \rho(\mu_i(\wt b)) v = \wt{\phi_i}(\wt b, v). 
\end{align*}
Therefore, the continuous map $\wt{\phi_i}$ is $M$-invariant, and consequently it gives rise to a continuous map $\phi_i: \pi^{-1} (U_i) \to \bb C^d$.

\begin{figure}[H]
\begin{center}
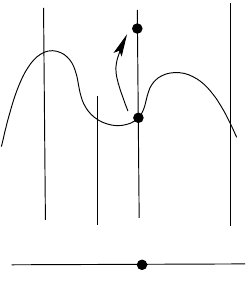
\caption{The map $\mu_i: \eta^{-1}(U_i) \to M$}
\label{the-map-mu-i}
\end{center}
\end{figure}

With the help of the above constructions, the next result says that $\pi: B \times_{\rho} \bb C^d \to B$ is a Lipschitz vector bundle. 

\begin{proposition}\label{Propo-Lipschitz-bundle-structure}
Let $M$ be a compact metrizable group
and let $\eta: \wt B \to B$ be a Lipschitz $M$-principal bundle. 
Let $\rho: M \to {\rm GL}_d(\bb C)$ be a continuous linear representation for some integer $d \geq 1$. 
Then, the maps $(U_i, \phi_i)_{i \in I}$ constructed above define a Lipschitz vector bundle structure on 
$\pi: B \times_{\rho} \bb C^d \to B$. 
\end{proposition}

\begin{proof}
We first check condition (1) in the definition of Lipschitz vector bundles given by Subsection \ref{Subsec-vector-bundles}. 
For each $i \in I$, consider the map 
\begin{align*}
\wt \Phi_i: \eta^{-1}(U_i) \times \bb C^d  & \longrightarrow  \eta^{-1}(U_i) \times \bb C^d  \notag\\
(\wt b, v) & \longmapsto \left(\wt{\pi} (\wt b, v), \wt{\phi_i} (\wt b, v) \right)
= \left( \wt b, \rho(\mu_i(\wt b)) v \right), 
\end{align*}
where $\wt \pi:  \eta^{-1}(U_i) \times \bb C^d  \to \eta^{-1}(U_i)$ is the projection onto the first component. 
Then it is immediate to see that $\wt \Phi_i$ is a homeomorphism of $\eta^{-1}(U_i) \times \bb C^d$ onto itself; 
its inverse is explicitly given by $(\wt b, v) \mapsto (\wt b, \rho(\mu_i(\wt b))^{-1} v)$, which is continuous as well.  
Besides, for any $m \in M$ and $(\wt b, v) \in \eta^{-1}(U_i) \times \bb C^d$, using \eqref{property-mu-m-invariant} and \eqref{def-wt-phi-i}, we compute 
\begin{align*}
\wt \Phi_i \left(\wt b m, \rho(m)^{-1} v \right) 
& = \left(\wt b m,   \rho(\mu_i(\wt b m))  \rho(m)^{-1} v \right)  \notag\\
& = \left(\wt b m, \rho(\mu_i(\wt b)) v \right)  \notag\\
& = \left(\wt b m, \wt{\phi_i} (\wt b, v) \right). 
\end{align*}  
Therefore, the map $\wt \Phi_i$ factors as a homeomorphism $\Phi_i$ from the quotient $\pi^{-1}(U_i)$ of $\eta^{-1}(U_i) \times \bb C^d$
by the diagonal action of $M$ onto $U_i \times \bb C^d$.
By construction, for any $w \in \pi^{-1}(U_i)$, we have 
\begin{align*}
\Phi_i (w) = (\pi(w), \phi_i(w)). 
\end{align*}
Hence the maps $\Phi_i: \pi^{-1}(U_i) \to U_i \times \bb C^d$, $i \in I$, 
provide the required local trivializations for the bundle $\pi : B \times_\rho \mathbb{C}^d \to B$, so condition (1) is satisfied.

Now we check condition (2) in the definition of Lipschitz vector bundles given by Subsection \ref{Subsec-vector-bundles}. 
Fix two indices $i, j \in I$. 
We consider the map $\wt \theta_{i, j}: \eta^{-1} (U_i \cap U_j) \to M$ by setting, 
for every $\wt b \in \eta^{-1} (U_i \cap U_j)$, 
\begin{align}\label{define-map-theta-ij}
\wt \theta_{i, j}(\wt b) = \mu_i (\wt b) \mu_j (\wt b)^{-1}. 
\end{align}
By \eqref{property-mu-m-invariant}, 
this map is $M$-invariant and factors as a continuous map $\theta_{i, j}: U_i \cap U_j \to M$. 
By construction, for any $(\wt b, v) \in \eta^{-1} (U_i \cap U_j) \times \bb C^d$, 
from \eqref{def-wt-phi-i} and \eqref{define-map-theta-ij} we have 
\begin{align*}
\wt \phi_i (\wt b, v) 
& = \rho(\mu_i(\wt b)) v  \notag\\
& = \rho(\mu_i(\wt b)) \rho(\mu_j(\wt b))^{-1} \rho(\mu_j(\wt b)) v  \notag\\
& = \rho(\wt \theta_{i, j}(\wt b))  \wt \phi_j (\wt b, v). 
\end{align*}
Since all these maps are $M$-invariant, passing to the quotient gives that, for every $w \in \pi^{-1} (U_i \cap U_j)$, 
\begin{align}\label{relation-phi-i-phi-j}
\phi_i (w) = \rho(\theta_{i, j} (\pi w)) \phi_j(w). 
\end{align}
To complete the proof of the proposition, 
it remains to show that the map 
$$\rho \circ \theta_{i, j}: U_i \cap U_j  \longrightarrow  {\rm GL}_d(\bb C)$$ is locally Lipschitz continuous. 
To this end, we first notice that, up to replacing $M$ by $\rho(M) = M / \ker \rho$
and $\wt B$ by $\wt B / \ker \rho$, we can assume $M$ to be a Lie group. 
Now observe that the map 
\begin{align*}
\eta^{-1}(U_i \cap U_j) &  \longrightarrow  (U_i \cap U_j) \times M  \notag\\
\wt b & \longmapsto (\eta(\wt b), \mu_i(\wt b))
\end{align*}
is a homeomorphism; by \eqref{existence-mu-i}, its inverse is given by $(b, m) \mapsto \sigma_i(b) m$. 
Therefore, 
applying Lemma \ref{Lem-structure} (which guarantees the Lipschitz continuity of the translation map for Lie group actions), we conclude that both
 maps $\mu_i$ and $\mu_j$ are locally Lipschitz continuous
on $\eta^{-1}(U_i \cap U_j)$.  
Consequently, their composition in \eqref{define-map-theta-ij} shows that the map $\wt \theta_{i, j}$ is locally Lipschitz continuous on $\eta^{-1}(U_i \cap U_j)$, 
and hence the factorized map $\theta_{i, j}$ is locally Lipschitz continuous on $U_i \cap U_j$. 
Therefore, the map $\rho \circ \theta_{i, j}: U_i \cap U_j \to {\rm GL}_d(\bb C)$ is also locally Lipschitz continuous, as required. 
\end{proof}


\subsection{Isometric actions of groups and principal bundles}

In the sequel we will use isometric actions of groups on vector bundles of the form $B \times_{\rho} \bb C^d$,
which are obtained as follows. Assume that the space $\wt B$ is equipped with a left action 
of a locally compact second countable topological group $G$
which commutes with the action of $M$, that is, for any $g \in G$, $\wt b \in \wt B$  and $m \in M$, we have $g(\wt b m) = (g \wt b) m$. 
Then we let $G$ act on $\wt B \times \bb C^d$ by setting, for $g \in G$, $\wt b \in \wt B$ and $v \in \bb C^d$, 
\begin{align}\label{action-G-on-B-Cd}
g(\wt b, v) = (g \wt b, v).  
\end{align}
Since $G$ commutes with the action of $M$,
by \eqref{def-action-M}, this action commutes with the action of $M$ on $\wt B \times \bb C^d$
and therefore induces an action of $G$ on $B \times_{\rho} \bb C^d$.

\begin{corollary}\label{Corollary-isometric Lipschitz bundle action}
Let $M$ be a compact metrizable group 
and let $\eta: \wt B \to B$ be a Lipschitz $M$-principal bundle. 
Let $\rho: M \to {\rm GL}_d(\bb C)$ be a continuous linear representation for some integer $d \geq 1$.  
Let $G$ be a locally compact second countable topological group acting continuously on $\wt B$ on the left by Lipschitz continuous homeomorphisms, 
and assume that the actions of $M$ and $G$ commute with each other. 
Then, the induced action of $G$ on $B \times_{\rho} \bb C^d$ is an isometric Lipschitz bundle action,
and there exists a constant $c >0$ such that, for any $g \in G$,  we have 
\begin{align}\label{Lipschitz-inequality001}
{\rm Lip}_{B \times_{\rho} \bb C^d}(g) \leq c \, {\rm Lip}(g). 
\end{align}
\end{corollary}

In \eqref{Lipschitz-inequality001} we have used the notation \eqref{def-Lip-norm-g} for ${\rm Lip}_{B \times_{\rho} \bb C^d}(g)$, 
and for $g \in G$, we have denoted by ${\rm Lip}(g)$ the Lipschitz constant of $g$: 
\begin{align}\label{def Lip g}
{\rm Lip}(g) 
 = \sup_{\substack{\wt b, \wt b' \in \wt B \\ \wt b \neq \wt b'}} \frac{\wt{\bf d} (g \wt b, g \wt b')}{\wt{\bf d} (\wt b, \wt b')}, 
\end{align}
where $\wt{\bf d}$ is the metric on $\wt B$. 

\begin{proof}[Proof of Corollary \ref{Corollary-isometric Lipschitz bundle action}]
Let us first prove that the induced action of $G$ commutes with the bundle projection map $\pi: B \times_{\rho} \bb C^d \to B$.
Indeed, for any $\wt b \in \wt B$ and $v \in \bb C^d$, by \eqref{def-pi-M} we have $\pi ( (\wt b, v) M) = \wt b M$ 
and hence, using \eqref{action-G-on-B-Cd} and the fact that the actions of $M$ and $G$ commute, we get that, for any $g \in G$, 
\begin{align*}
\pi (g  (\wt b, v) M) 
 = \pi ( (g\wt b, v) M)  = g\wt b M  = g \pi ( (\wt b, v) M). 
\end{align*}
Besides, we claim that $G$ acts on $B$ by Lipschitz homeomorphisms. Indeed, by \eqref{def distance on B} and \eqref{def Lip g},
we have, for any $g \in G$, $b = \wt b M \in B$ and $b' = \wt b' M \in B$, 
\begin{align}\label{G acting on B Lipschitz}
\bf d (g b, g b') 
 = \inf_{m \in M} \wt{\bf d} (g \wt b m, g \wt b') 
 \leq  {\rm Lip}(g)  \inf_{m \in M} \wt{\bf d} (\wt b m, \wt b')
= {\rm Lip}(g)  \bf d (b, b'). 
\end{align}

Now recall that $I$ stands for the index set of all local Lipschitz continuous sections of the map $\eta: \wt B \to B$.
We define an action of $G$ on this index set $I$ as follows. 
For any $g \in G$, and any $i \in I$ with associated local section $(U_i, \sigma_i)$
where $\sigma_i: U_i \to \wt B$ is given by \eqref{def-sigma-i}, 
we set 
\begin{align}\label{def-sigma-gi-b}
U_{gi} = g U_i 
\quad \mbox{and} \quad 
\sigma_{gi}(b) = g \sigma_i(g^{-1} b) 
\quad  \mbox{for}  \  b \in U_{gi}. 
\end{align}
The map $\sigma_{gi}$ is indeed a section since the actions of $M$ and $G$ commute with each other and consequently, 
$$\eta(\sigma_{gi}(b)) = \eta(g \sigma_i(g^{-1} b)) = g \eta(\sigma_i(g^{-1} b)) = g g^{-1} b = b.$$ 
Furthermore, the section $\sigma_{gi}$ is locally Lipschitz continuous as $G$ acts on both $\wt B$ and $B$ by Lipschitz continuous homeomorphisms. 
We claim that, for any $g \in G$ and $\wt b \in \eta^{-1} (gU_i)$, it holds that 
\begin{align}\label{relation-mu-wtb}
\mu_{gi}(\wt b) = \mu_i(g^{-1} \wt b), 
\end{align}
where $\mu_i: \eta^{-1}(U_i) \to M$ is given by \eqref{existence-mu-i}. 
Indeed, by \eqref{def-sigma-gi-b}, \eqref{existence-mu-i} 
and again the fact that the actions of $M$ and $G$ commute with each other, we have, for any $g \in G$ and $\wt b \in \eta^{-1} (gU_i)$, 
\begin{align*}
 \sigma_{gi}(\eta (\wt b))  \mu_i(g^{-1} \wt b)
 =   g \sigma_i(g^{-1} \eta (\wt b))  \mu_i(g^{-1} \wt b)  
= g g^{-1} \wt b = \wt b,
\end{align*}
so that \eqref{relation-mu-wtb} holds by the uniqueness of the map $\mu_{gi}: \eta^{-1}(gU_i) \to M$. 
From this construction, by \eqref{def-wt-phi-i}, \eqref{action-G-on-B-Cd} and \eqref{relation-mu-wtb}, 
we get that, for any $g \in G$, $\wt b \in \eta^{-1} (gU_i)$ and $v \in \bb C^d$, 
\begin{align*}
\wt \phi_{gi}(\wt b, v) 
 = \rho \big( \mu_{gi}(\wt b) \big) v   = \rho \big( \mu_i(g^{-1} \wt b) \big) v  
 = \wt \phi_{i}(g^{-1} \wt b, v) =  \wt \phi_{i}(g^{-1}(\wt b, v)). 
\end{align*}
Consequently, in view of \eqref{relation-phi-i-phi-j}, by passing to the quotient, this yields that, for any $g \in G$ and $w \in \pi^{-1}(gU_i)$, 
\begin{align}\label{identity-phi-i}
\phi_{gi}(w) = \phi_i(g^{-1} w). 
\end{align}

Keeping the notation from the proof of Proposition \ref{Propo-Lipschitz-bundle-structure}, 
for $i, j \in I$, $g \in G$ and $b \in U_i \cap (g^{-1} U_j)$, we define 
\begin{align}\label{identity-sigma-ij-theta-ij}
\sigma_{i, j}(g, b) = \theta_{i, g^{-1} j} (b)^{-1}. 
\end{align}
Up to replacing $M$ by $M/\ker \rho$ and $\wt B$ by $\wt B / \ker \rho$,
we may assume that $M$ is a Lie group. 
Then, by Lemma \ref{Lem-structure},
the map $b \mapsto \rho(\sigma_{i, j}(g, b))$ is Lipschitz continuous. 
Now, using \eqref{relation-phi-i-phi-j}, \eqref{identity-phi-i} and \eqref{identity-sigma-ij-theta-ij}, 
for $w \in \pi^{-1}(U_i \cap (g^{-1} U_j) )$, we obtain 
\begin{align*}
\phi_i (w) 
& =  \rho(\theta_{i, g^{-1} j} (\pi w)) \phi_{g^{-1} j}(w)  \notag\\
& = \rho( \sigma_{i, j}(g, \pi w)^{-1} ) \phi_{g^{-1} j}(w)  \notag\\
& =  \rho( \sigma_{i, j}(g, \pi w)^{-1} ) \phi_j(g w)
\end{align*}
and hence
\begin{align}\label{relation-phi-j-phi-i}
\phi_j(gw)  = \rho( \sigma_{i, j}(g, \pi w) ) \phi_i (w). 
\end{align}
Fix $(b, v) \in U_i \times \bb C^d$ and set $w = (\pi \times \phi_i)^{-1}(b, v)$. 
We have $\pi w = b$ and $\phi_i(w) = v$. 
By \eqref{property-group-action}, we get $\pi(gw) = g \pi(w) = g b$ and therefore, 
from \eqref{relation-phi-j-phi-i} we derive that, for any $g \in G$, $v \in \bb C^d$ and $b \in U_i \cap  (g^{-1} U_j)$, 
\begin{align*}
g (\pi \times \phi_i)^{-1} (b, v) 
 = gw 
& = \left( \pi \times \phi_j \right)^{-1} \Big( \pi(gw), \rho \big( \sigma_{i, j} (g, \pi w) \big) \phi_i (w) \Big)  \notag\\
&= \left( \pi \times \phi_j \right)^{-1} \Big( gb, \rho(\sigma_{i, j} (g, b)) v \Big). 
\end{align*}
That is, the identity \eqref{def-sigma-ij} holds. 

To establish \eqref{Lipschitz-inequality001}, 
we proceed as in Subsection \ref{Subsection-Group actions-spectral gap} by fixing a finite subset $J \subset I$ such that 
the corresponding open subsets $(U_i)_{i \in J}$ form a finite open cover of $B$. 
From \eqref{identity-sigma-ij-theta-ij}, \eqref{define-map-theta-ij}, \eqref{relation-mu-wtb} and \eqref{existence-mu-i}, 
we deduce that,
for any $i, j \in J$, $g \in G$ and $\wt b \in \eta^{-1}(U_i)$ such that $g \wt b \in \eta^{-1}(U_j)$, 
\begin{align*}
\sigma_{i, j}(g, \eta(\wt b)) 
& = \theta_{i, g^{-1} j} (\eta(\wt b))^{-1}  \notag\\
& =  \mu_{g^{-1} j} (\wt b)  \mu_i (\wt b)^{-1}  \notag\\
& =  \mu_{j} (g \wt b)  \mu_i (\wt b)^{-1}   \notag\\
& = T \left( \sigma_j(\eta(g \wt b)), g\wt b \right) T \left( \sigma_i(\eta(\wt b)), \wt b \right)^{-1}, 
\end{align*}
where $T$ denotes the translation map for the action of $M$ on $\wt B$ (cf.\ Subsection \ref{subsec-free-action}). 
Thus, taking $b \in U_i$ and applying the above identity to $\wt b = \sigma_i (b)$, we obtain 
\begin{align}\label{identity sigma ij sigma j}
\sigma_{i, j}(g, b) =  T \left( \sigma_j(g b), g \sigma_i (b) \right). 
\end{align}
For each $i \in J$, we write $c_i$ for the Lipschitz constant of the section $\sigma_i: U_i \to \wt B$. 
Taking into account \eqref{identity sigma ij sigma j} and applying Lemma \ref{Lem-structure} to the $\rho(M)$-principal bundle $\wt B / \ker \rho$, 
we obtain that, for any $g \in G$, and any $b, b' \in U_i \cap  (g^{-1} U_j)$, 
\begin{align*}
 D(\rho(\sigma_{i, j}(g, b)), \rho(\sigma_{i, j}(g, b')))  
& \leq c  \Big( \wt{\bf d} \left( \sigma_j(gb),  \sigma_j(gb')  \right) +   \wt{\bf d} (g \sigma_i (b), g \sigma_i (b'))  \Big) \notag\\
& \leq c  {\rm Lip}(g) (c_i + c_j)   \bf d (b, b') \notag\\
& \leq  c'  {\rm Lip}(g)  \bf d (b, b'), 
\end{align*}
where we have used the finiteness of $J$. 
In view of \eqref{G acting on B Lipschitz} and \eqref{def-Lip-norm-g},  
this proves the required Lipschitz estimate \eqref{Lipschitz-inequality001}, and therefore completes the proof 
of the corollary. 
\end{proof}

\subsection{Equivariant maps and sections}\label{Sec-Equivariant maps}
We retain the previous setup: a compact metrizable group $M$ acts on a compact metrizable topological space $\wt B$ on the right,
and we have a continuous linear representation $\rho: M \to {\rm GL}_d(\bb C)$ for some integer $d \geq 1$. 
We now interpret $M$-equivariant maps $\wt B \to \bb C^d$ as sections of the associated vector bundle $B \times_{\rho} \bb C^d \to B$. 
Denote by $C^0_{\rho}(\wt B, \bb C^d)$ the space of all continuous $M$-equivariant maps $\wt B \to \bb C^d$. 
This means that an element $\varphi \in C^0_{\rho}(\wt B, \bb C^d)$ is a continuous map $\varphi: \wt B \to \bb C^d$ such that,
for any $\wt b \in \wt B$ and $m \in M$, it holds 
\begin{align}\label{func-space-C0-wtB}
\varphi(\wt b m) = \rho(m)^{-1} \varphi(\wt b). 
\end{align}
We equip the space $C^0_{\rho}(\wt B, \bb C^d)$ with the topology of uniform convergence. 
In the same way, for $\gamma \in (0, 1]$, let $C^{\gamma}_{\rho}(\wt B, \bb C^d)$ be the space of all $M$-equivariant $\gamma$-H\"older continuous maps $\wt B \to \bb C^d$. 

Given a continuous map $\varphi \in C^0_{\rho}(\wt B, \bb C^d)$, 
we shall define a continuous section $s_{\varphi}$ of the bundle $B \times_{\rho} \bb C^d \to B$ as follows. 
For $\wt b \in \wt B$, 
we set 
\begin{align}\label{def wt s varphi}
\wt s_{\varphi} (\wt b) = (\wt b, \varphi(\wt b)) \in \wt B \times \bb C^d. 
\end{align}
By the equivariance property \eqref{func-space-C0-wtB}, the map $\wt s_{\varphi}: \wt B \to \wt B \times \bb C^d$ is continuous 
and commutes with the action of $M$ 
(i.e., $\wt s_{\varphi} (\wt b m) = (\wt b m, \rho(m)^{-1} \varphi(\wt b)) = (\wt b, \varphi(\wt b)) m = \wt s_{\varphi} (\wt b) m$, by using \eqref{def-action-M}), 
so that it factors as a continuous map $s_{\varphi}: B \to B \times_{\rho} \bb C^d$. 
Recall from Subsection \ref{Subsec-vector-bundles} that, if $E \to B$ is a Lipschitz vector bundle, for $\gamma \in [0, 1]$, 
we write $\bf \Gamma^{\gamma}(E)$ for the space of $\gamma$-H\"older continuous sections of $E$. 

\begin{lemma}\label{Lem-closed injection-001}
Let $M$ be a compact metrizable group 
and let $\eta: \wt B \to B$ be a Lipschitz $M$-principal bundle. 
Let $\rho: M \to {\rm GL}_d(\bb C)$ be a continuous linear representation for some integer $d \geq 1$. 
Then, for any continuous map $\varphi \in C^0_{\rho}(\wt B, \bb C^d)$, 
the map $s_{\varphi}$ is a continuous section of the bundle $B \times_{\rho} \bb C^d \to B$. 
Moreover, for any $\gamma \in [0, 1]$, the map 
\begin{align*}
 C^{\gamma}_{\rho}(\wt B, \bb C^d) &  \longrightarrow  \bf \Gamma^{\gamma}(B \times_{\rho} \bb C^d)  \notag\\
\varphi  &  \longmapsto s_{\varphi}
\end{align*}
is a Banach space isomorphism. 
\end{lemma}

\begin{proof}
Let $b  \in B$ and pick $\wt b \in \wt B$ such that $b = \eta(\wt b) = \wt b M$. Since, by using \eqref{commutative diagram}, 
\begin{align*}
\pi(s_{\varphi}(b)) 
= \pi \left( \wt \eta(\wt s_{\varphi}(\wt b)) \right) 
= \eta \left( \wt \pi(\wt s_{\varphi}(\wt b)) \right) 
= \eta \left( \wt \pi(\wt b, \varphi(\wt b)) \right)
= \eta(\wt b) 
= b, 
\end{align*}
the map $s_{\varphi}: B \to B \times_{\rho} \bb C^d$ is indeed a section. 
Conversely, let us show that each section is of that form. 
As the action of $M$ on $\wt B$ is free, 
if $s \in \bf \Gamma^0(B \times_{\rho} \bb C^d)$ is a continuous section of $B \times_{\rho} \bb C^d$, 
then, for any $\wt b \in \wt B$, there exists a unique vector $\varphi(\wt b) \in \bb C^d$ such that 
\begin{align*}
s(\eta(\wt b)) = (\wt b, \varphi(\wt b)) M. 
\end{align*}
We claim that the map $\varphi: \wt B \to \bb C^d$ is continuous. 
Indeed, if $\wt b_n \xrightarrow[n\to\infty]{} \wt b$ in $\wt B$, then 
\begin{align*}
(\wt b_n, \varphi(\wt b_n)) M = s(\eta(\wt b_n)) \xrightarrow[n\to\infty]{} s(\eta(\wt b)), 
\end{align*}
and therefore, the sequence $(\varphi(\wt b_n))_{n \geq 1}$ is bounded in $\bb C^d$.
Since any limit point $v$ of $(\varphi(\wt b_n))_{n \geq 1}$ in $\bb C^d$ has to satisfy $s(\eta(\wt b)) = (\wt b, v) M$,
by uniqueness, we get that $\varphi(\wt b_n)$ converges to $\varphi(\wt b)$ as $n \to \infty$. 
Therefore, $\varphi$ is continuous and the uniqueness property also shows that $\varphi$ is $\rho$-equivariant. 
By construction, we get $s = s_{\varphi}$, so that each section is of that form.

Now let $i \in I$ and let $(U_i, \sigma_i)$ be the associated  local section of the map $\eta$.
Then, by \eqref{def-wt-phi-i}, for any $b = \wt b M \in U_i$, we have 
\begin{align}\label{identity phi varphi}
\phi_i(s_{\varphi}(b)) = \wt \phi_i(\wt s_{\varphi}(\wt b)) = \wt \phi_i(\wt b, \varphi(\wt b)) =  \rho(\mu_i(\wt b)) \varphi(\wt b). 
\end{align}
Therefore, for an open subset $V_i \subset U_i$ as in \eqref{def-norm-s-infty}, we have 
$$\sup_{b \in V_i} \| \phi_i(s_{\varphi}(b)) \| =  \sup_{\wt b \in \eta^{-1}(V_i)} \| \varphi(\wt b) \|,$$
which says that the map $\varphi \mapsto s_{\varphi}$ 
is continuous and is a Banach space isomorphism from 
$C^0_{\rho}(\wt B, \bb C^d)$ to $\bf \Gamma^0(B \times_{\rho} \bb C^d).$ 

Now let $\gamma \in (0, 1]$ and take $\varphi \in C^{\gamma}_{\rho}(\wt B, \bb C^d)$. 
For any $i \in I$ and $b, b' \in U_i$, by \eqref{def-sigma-i}, we get  
\begin{align}\label{Lipschitz phi varphi}
\| \phi_i(s_{\varphi}(b)) - \phi_i(s_{\varphi}(b')) \| 
& = \|  \varphi(\sigma_i (b)) -  \varphi(\sigma_i (b')) \|  \notag\\
& \leq \|\varphi\|_{\gamma}  \wt{\bf d}  (\sigma_i (b), \sigma_i (b'))^{\gamma}  \notag\\
& \leq  c_i^{\gamma}  \|\varphi\|_{\gamma}  \bf d (b, b')^{\gamma}, 
\end{align}
where $c_i$ is the Lipschitz constant of the section $\sigma_i: U_i \to \eta^{-1}(U_i)$. 
This tells us that the map $\varphi \mapsto s_{\varphi}$ sends $C^{\gamma}_{\rho}(\wt B, \bb C^d)$ 
continuously to $\bf \Gamma^{\gamma}(B \times_{\rho} \bb C^d)$. 
Conversely, if $s_{\varphi}$ is a $\gamma$-H\"older continuous section, then, by \eqref{identity phi varphi},
we have, for any $i \in I$ and $\wt b \in \eta^{-1}(U_i)$, 
\begin{align}\label{identity varphi b rho s-varphi}
\varphi(\wt b) = \rho(\mu_i(\wt b))^{-1} \phi_i(s_{\varphi}(\eta(\wt b))). 
\end{align}
By Lemma \ref{Lem-structure}, the map $\rho \circ \mu_i$ is Lipschitz continuous on $\eta^{-1}(U_i)$. 
Hence $\varphi$ is $\gamma$-H\"older continuous on $\eta^{-1}(U_i)$. 
As this is true for every $i \in I$, the map $\varphi$ is $\gamma$-H\"older continuous on $\wt B$
and this shows that the map $\varphi \mapsto s_{\varphi}$ sends $C^{\gamma}_{\rho}(\wt B, \bb C^d)$ 
 to a closed subspace of $\bf \Gamma^{\gamma}(B \times_{\rho} \bb C^d)$. 
 Indeed, let $c>0$ be such that, for any $i \in I$, the function $\rho \circ \mu_i$ is $c$-Lipschitz continuous on $\eta^{-1}(U_i)$. 
 For any $i \in I$, we choose an open subset $V_i$ of $U_i$ such that the closure of $V_i$ in $B$ is contained in $U_i$ in such a way that
we have $B = \bigcup_{i \in I} V_i$. 
Then, using \eqref{identity varphi b rho s-varphi}, \eqref{Lipschitz phi varphi} and \eqref{def-norm-s-Holder}, 
and the fact that $\rho \circ \mu_i$ is Lipschitz continuous on $\eta^{-1}(V_i)$, 
we get that for any $\wt b, \wt b' \in \eta^{-1}(V_i)$,  
\begin{align*}
\| \varphi(\wt b) - \varphi(\wt b') \| 
& = \|  \rho(\mu_i(\wt b))^{-1} \phi_i(s_{\varphi}(\eta(\wt b))) - \rho(\mu_i(\wt b'))^{-1} \phi_i(s_{\varphi}(\eta(\wt b'))) \| \notag\\
& \leq  \Big\|  \rho(\mu_i(\wt b))^{-1}  \Big[  \phi_i(s_{\varphi}(\eta(\wt b))) -  \phi_i(s_{\varphi}(\eta(\wt b'))) \Big]   \Big\|  \notag\\
& \quad  +  \Big\| \Big[ \rho(\mu_i(\wt b))^{-1} - \rho(\mu_i(\wt b'))^{-1}  \Big]  \phi_i(s_{\varphi}(\eta(\wt b')))   \Big\|  \notag\\
& \leq \|  \phi_i(s_{\varphi}(\eta(\wt b))) -  \phi_i(s_{\varphi}(\eta(\wt b'))) \| \notag\\
& \quad  + \|  \rho(\mu_i(\wt b))^{-1}  - \rho(\mu_i(\wt b'))^{-1}  \| \|\phi_i(s_{\varphi}(\eta(\wt b')))\|  \notag\\
& =  \|  \phi_i(s_{\varphi}(\eta(\wt b))) -  \phi_i(s_{\varphi}(\eta(\wt b'))) \| \notag\\
& \quad  + \|  \rho(\mu_i(\wt b'))  - \rho(\mu_i(\wt b))  \| \|\phi_i(s_{\varphi}(\eta(\wt b')))\|  \notag\\
& \leq  \|s_{\varphi}\|_{\gamma} \bf d(\eta(\wt b), \eta(\wt b'))^{\gamma} + c \|s_{\varphi}\|_{\gamma}  \bf d(\wt b', \wt b)  \notag\\
& \leq  \left( 1 + c \, {\rm diam}(\wt B)^{1 - \gamma} \right)\|s_{\varphi}\|_{\gamma} \bf d(\wt b, \wt b')^{\gamma}. 
\end{align*}
Since $\wt B$ is compact, it follows that there exists a constant $c>0$ such that $\|\varphi\|_{\gamma} \leq c \|s_{\varphi}\|_{\gamma}$. 
This shows that the map $\varphi \mapsto s_{\varphi}$ is a Banach space isomorphism from $C^{\gamma}_{\rho}(\wt B, \bb C^d)$ 
 to $\bf \Gamma^{\gamma}(B \times_{\rho} \bb C^d)$. 
\end{proof}

Let $G$ be a locally compact second countable group acting on $\wt B$ on the left
in such a way that this action commutes with the action of $M$.
Then $G$ acts naturally on both the space of continuous $M$-equivariant maps $C^0_{\rho}(\wt B, \bb C^d)$ 
and the space of continuous sections $\bf \Gamma^0(B \times_{\rho} \bb C^d)$.
Specifically, for any $g \in G$, $\varphi \in C^0_{\rho}(\wt B, \bb C^d)$ and $\wt b \in \wt B$, 
we define 
\begin{align}\label{g action on function}
(g \varphi)(\wt b) = \varphi(g^{-1} \wt b). 
\end{align}
Likewise, for any $s \in \bf \Gamma^0(B \times_{\rho} \bb C^d)$ and $b \in B$, we set 
\begin{align}\label{g action on section}
(g s)(b) = g s(g^{-1} b). 
\end{align}
A direct computation gives the following equivariance property.   

\begin{corollary}\label{Cor-closed injection-001}
Let $M$ be a compact metrizable group 
and let $\eta: \wt B \to B$ be a Lipschitz $M$-principal bundle. 
Let $\rho: M \to {\rm GL}_d(\bb C)$ be a continuous linear representation for some integer $d \geq 1$.  
Let $G$ be a locally compact second countable topological group acting continuously on $\wt B$ by Lipschitz continuous homeomorphisms, 
and assume that the actions of $M$ and $G$ commute with each other. 
Then, for any $g \in G$ and $\varphi \in C^0_{\rho}(\wt B, \bb C^d)$, we have 
\begin{align*}
s_{g \varphi} = g s_{\varphi}. 
\end{align*}
\end{corollary}

\begin{proof}
This is a consequence of \eqref{action-G-on-B-Cd}, \eqref{def wt s varphi}, \eqref{g action on function} and \eqref{g action on section}. 
\end{proof}



\section{Existence of limit measures}\label{Sec-Existence of limit measures}

\subsection{Equidistribution properties}\label{section equidistribution 001}
We will apply the results developed in the preceding sections to study Markov chains 
arising from actions of a group $G$ that commute with the action of another compact group $M$. 

Let $G$ be a second countable locally compact group
acting on the left by Lipschitz continuous homeomorphisms on the compact metric space $\wt B$.  
As before, for $g \in G$, we denote by ${\rm Lip}(g)$ the Lipschitz constant of $g$:  
\begin{align*}
{\rm Lip}(g) 
 = \sup_{\substack{\wt b, \wt b' \in \wt B \\ \wt b \neq \wt b'}} \frac{\wt{\bf d} (g \wt b, g \wt b')}{\wt{\bf d} (\wt b, \wt b')}, 
\end{align*}
where $\wt{\bf d}$ is the metric on $\wt B$. 
We assume that $\wt B$ is also equipped with a free right action by isometries of a compact metrizable group $M$ 
such that the quotient map $\eta: \wt B \to B = \wt B / M$ is a Lipschitz $M$-principal bundle
(see Subsection \ref{Section-Quotient vector bundles} for the definition). 
We further suppose that the actions of $G$ and $M$ commute with each other.  
The quotient space $B$ is equipped with the quotient action of $G$. 
Let $\mu$ be a Borel probability measure on $G$, 
and let $(g_n)_{n \geq 1}$ be a sequence of independent random variables with common law $\mu$. 
We always assume that the action of $G$ on $B$ satisfies the contraction property \eqref{contraction-property-on-B}. 
In addition, we always assume that there exists a constant $\alpha>0$ such that 
\begin{align}\label{moment-assumption-01}
\int_{G} {\rm Lip}(g)^{\alpha} \mu(dg) < \infty. 
\end{align}
Denote by $P_{\mu}$ the convolution operator associated with $\mu$, that is, for any continuous function $\varphi \in C^0(\wt B)$ and $\wt b \in \wt B$,
we set 
\begin{align*}
P_{\mu} \varphi(\wt b) = \int_{G} \varphi(g \wt b) \mu(dg) = \bb E \varphi(g_1 \wt b). 
\end{align*}
Under these assumptions, we will prove the following equidistribution result, 
which describes the asymptotic behavior of the random walk on the space $\wt B$. 

\begin{theorem}\label{Thm-equdistribution}
Let $\mu$ be a Borel probability measure on $G$. 
Assume that the action of $G$ on $\wt B$ satisfies the moment condition \eqref{moment-assumption-01} and 
that the action of $G$ on $B$ satisfies the contraction property \eqref{contraction-property-on-B}. 
Then, for every $\wt b \in \wt B$, there exists a Borel probability measure $\wt \nu_{\wt b}$ on $\wt B$ such that, for any continuous 
function $\varphi$ on $\wt B$, we have 
\begin{align*}
 \frac{1}{n} \sum_{k=0}^{n-1}  P_{\mu}^k \varphi(\wt b) 
=  \frac{1}{n} \sum_{k=0}^{n-1}  \bb E \varphi(g_k \cdots g_1 \wt b) 
\xrightarrow[n\to\infty]{}  \int_{\wt B} \varphi(\wt b') \wt \nu_{\wt b}(d \wt b'). 
\end{align*}
Moreover, this convergence holds uniformly over $\wt b \in \wt B$, and we have the following harmonicity property: for any $\wt b \in \wt B$, 
\begin{align*}
\wt \nu_{\wt b} = \int_G \wt \nu_{g \wt b} \mu(dg). 
\end{align*}
\end{theorem}

Under some additional assumption, which will be satisfied in all of our concrete examples, 
we can get rid of the Birkhoff average in the previous theorem. 
We will prove that, 
since $(G, \mu)$ has the contraction property on $B$,  
there exists a unique $\mu$-stationary probability measure $\nu$ on $B$, see Lemma \ref{Lemma-unqiue-sta-measure} below. 
Then, we define $\wt \nu$ to be the unique $M$-invariant probability measure $\wt \nu$ on $\wt B$ whose image (push-forward) under 
the quotient map $\eta: \wt B \to B = \wt B / M$ equals $\nu$, i.e., 
$\eta_* \wt \nu = \nu$. 

\begin{corollary}\label{Cor-equdistribution}
Let $\mu$ be a Borel probability measure on $G$. 
Assume that the action of $G$ on $\wt B$ satisfies the moment condition \eqref{moment-assumption-01} and 
that the action of $G$ on $B$ satisfies the contraction property \eqref{contraction-property-on-B}.  
Suppose that, for every $\wt b \in \wt B$ with $\eta(\wt b) \in \supp \nu$,   
we have $\wt b M \subset \overline{\Gamma_{\mu} \wt b}$. 
Then, $\wt \nu$ is the unique invariant probability measure of the operator $P_{\mu}$. 
Furthermore, there exist  \\
(1) a closed coabelian normal subgroup $M_0^{\mu}$ of $M$, \\  
(2) an element $m_0^{\mu} \in M/M_0^{\mu}$ which spans a dense subgroup of $M/M_0^{\mu}$,  \\
(3) a continuous map $\wt b \mapsto \theta_{\wt b}$ from $\wt B$ to the set of Borel probability measures on $M/M_0^{\mu}$,  \\
(4) a non-negative linear map $\varphi \mapsto Q \varphi: C^0(\wt B) \to C^0(M/M_0^{\mu})$ with $Q 1 = 1$, \\
 such that, for any $\varphi \in C^0(\wt B)$, we have, uniformly in $\wt b \in \wt B$, 
\begin{align}
P_{\mu}^n \varphi(\wt b) - \int_{M/M_0^{\mu}} Q \varphi \big( (m_0^{\mu})^n m \big)  \,  \theta_{\wt b}(dm) 
 \xrightarrow[n\to\infty]{}  0. 
  \label{convergence-P-mu-meas} 
\end{align}
In addition, we have the following identities: 
\begin{align}
& \int_{\wt B} \varphi(\wt b)  \,  \wt \nu(d\wt b) = \int_{M/M_0^{\mu}} Q \varphi(m) dm,  \label{nu-M-M0-meas}\\
&  Q (P_{\mu} \varphi)(m) = Q \varphi(m_0^{\mu} m),  \label{property-F-meas}\\
&  \theta_{\wt b} = \int_{G} ((m_0^{\mu})^{-1})_*\theta_{g \wt b} \, \mu(dg).  \label{harmonicity-meas}
\end{align}
\end{corollary}

Recall that a normal subgroup $\Delta$ of a group $\Gamma$ is said to be {\it coabelian} if the quotient group $\Gamma/\Delta$ is Abelian. 

\begin{remark}\label{Remark-unique-invariant-measure}
The convergence \eqref{convergence-P-mu-meas} together with the integral identity \eqref{nu-M-M0-meas} 
implies the uniqueness of the invariant measure of the operator $P_{\mu}$. 
Indeed, for any $\wt b \in \wt B$ and $\varphi \in C^0(\wt B)$,  
we have  
\begin{align*}
\lim_{n \to \infty}  \frac{1}{n} \sum_{k=0}^{n-1}  P_{\mu}^k \varphi(\wt b) 
& =  \lim_{n \to \infty}  \int_{M/M_0^{\mu}}  \frac{1}{n} \sum_{k=0}^{n-1}  Q \varphi \big( (m_0^{\mu})^k m \big)  \, \theta_{\wt b}(dm)  \notag\\
& =  \int_{M/M_0^{\mu}}  Q \varphi(m) dm  \notag\\
& =  \int_{\wt B} \varphi(\wt b)  \,  \wt \nu(d\wt b), 
\end{align*}
where we have used the fact that, since $m_0^{\mu}$ spans a dense subgroup of $M/M_0^{\mu}$, the translation map $m \mapsto m_0^{\mu} m$ is uniquely ergodic on $M/M_0^{\mu}$. 
Note that this also implies Theorem \ref{Thm-equdistribution} in this particular case. 
\end{remark}

\begin{remark}\label{Rem-operator-Q-q}
Thanks to the Riesz representation theorem, 
the identities \eqref{convergence-P-mu-meas}, \eqref{nu-M-M0-meas} and \eqref{property-F-meas} can be equivalently reformulated 
in terms of a family of probability measures on $\wt B$. 
Specifically, there exists a continuous map $m \mapsto q_m$ from $M/M_0^{\mu}$ to the set of Borel probability measures on $\wt B$ 
 such that, for any $\wt b \in \wt B$ and $\varphi \in C^0(\wt B)$, 
we have 
\begin{align*}
&   P_{\mu}^n \varphi(\wt b) - 
  \int_{M/M_0^{\mu}}  \left( \int_{\wt B} \varphi(\wt b) \, q_{(m_0^{\mu})^n m}(d\wt b)  \right)   \theta_{\wt b}(dm) 
   \xrightarrow[n\to\infty]{} 0, \\
& \int_{\wt B} \varphi(\wt b)  \,  \wt \nu(d\wt b)  = \int_{M/M_0^{\mu}}  \left( \int_{\wt B} \varphi(\wt b) \, q_{m}(d\wt b) \right)  dm, \\ 
&  \int_{\wt B}  P_{\mu} \varphi(\wt b) \, q_{m}(d\wt b)  =  \int_{\wt B} \varphi(\wt b) \, q_{m_0^{\mu} m}(d\wt b). 
\end{align*}

\end{remark}

\begin{example}\label{Example-G-compact}
Consider the case where $G = M$ is a compact group 
equipped with a Borel probability measure $\mu$ whose support spans a dense subgroup of $M$. 
Then we can define $\wt B$ as the space $M$ itself, endowed with the actions on the left and on the right by translations
which commute with each other. 
The proof of Corollary \ref{Cor-equdistribution} then allows to retrieve the classical fact that $M$ admits 
a largest topologically cyclic quotient $M/M_0^{\mu}$ such that the image of $\mu$ in $M/M_0^{\mu}$ is a Dirac mass $\delta_{m_0^{\mu}}$
for some element $m_0^{\mu} \in M/M_0^{\mu}$. 
In this setting, the conclusion of the corollary states that for any $\varphi \in C^0(M)$, uniformly in $m \in M$, 
we have 
\begin{align*}
 \bb E \varphi(g_n \cdots g_1 m) - \int_{M_0^{\mu}} \varphi \big( (m_0^{\mu})^n m h \big) dh 
\xrightarrow[n\to\infty]{}  0. 
\end{align*}
\end{example}

\subsection{Harmonic analysis on compact groups}

When the group $M$ is trivial, that is, $\wt B = B$, the equidistribution property stated in Theorem \ref{Thm-equdistribution}
follows directly  from the contraction property  \eqref{contraction-property-on-B}. 

\begin{lemma}\label{Lemma-unqiue-sta-measure}
Let $\mu$ be a Borel probability measure on $G$. 
Assume that the action of $G$ on $B$ satisfies the contraction property \eqref{contraction-property-on-B}.
Then, the operator $P_{\mu}$ admits a unique invariant Borel probability measure $\nu$ on $B$, and 
for any continuous function $\varphi$ on $B$, we have, uniformly in $b \in B$, 
\begin{align*}
 P_{\mu}^n \varphi(b) 
=    \bb E \varphi(g_n \cdots g_1 b) 
\xrightarrow[n\to\infty]{}  \int_{B} \varphi(b')  \nu(d b'). 
\end{align*}
\end{lemma}

\begin{proof}
Let $\nu$ be a $P_{\mu}$-invariant Borel probability measure on $B$, which exists since $B$ is compact. 
For any $\varphi \in C^0(B)$, the contraction property \eqref{contraction-property-on-B} 
together with the uniform continuity of $\varphi$ yields, uniformly in $b \in B$, 
\begin{align*}
  \int_{B} | P_{\mu}^n \varphi(b)  -  P_{\mu}^n \varphi(b')| \nu(d b')  \xrightarrow[n\to\infty]{}  0.
\end{align*}
The conclusion follows as 
\begin{align*}
\int_{B}  P_{\mu}^n \varphi(b') \nu(d b') = \int_{B}  \varphi(b') \nu(d b') 
\end{align*}
for any $n \geq 1$. 
\end{proof}

For proving Theorem \ref{Thm-equdistribution} in the general case where $M$ is nontrivial, 
we need to lift this equidistribution property in $B$ to an equidistribution property in $\wt B$. 
This will be achieved by studying the action of the operator $P_{\mu}$ on a particular class of functions on $\wt B$. 
To define these functions, we employ the representation theory of compact groups, for which we refer to \cite{Ser77, Sim96}. 
Let $\widehat M$ be the set of isomorphism classes of complex irreducible representations of $M$.
Thus, for each $x \in \wh M$, we have an associated integer $d_x \geq 1$ 
and a continuous irreducible representation $\rho_x: M \to {\rm GL}_{d_x}(\bb C)$. Besides, for any $d \geq 1$ 
and any continuous irreducible representation $\rho: M \to {\rm GL}_{d}(\bb C)$, 
there exists a unique $x \in \wh M$ such that $d_x = d$ and we may find $A \in {\rm GL}_{d}(\bb C)$ with $\rho(m) = A \rho_x(m) A^{-1}$
for any $m \in M$.  
Note that since we have assumed the compact group $M$ to be metrizable, the set $\wh M$ is countable. 

For $x \in \wh M$ and $m \in M$, we set as usual $\chi_x(m) = {\rm tr}(\rho_x(m))$ to be the trace of $\rho_x(m)$. 
Define the spectral projection $\pi_x$ on the $x$-isotypical component of $C^0(\wt B)$ as follows: 
for $\varphi \in C^0(\wt B)$ and $\wt b \in \wt B$, 
\begin{align}\label{def operator pi-x}
\pi_x(\varphi)(\wt b) = d_x \int_{M} \overline{\chi_x(m)} \varphi(\wt b m) dm. 
\end{align}
where $dm$ denotes the normalized Haar measure on $M$. 
We denote by $W_x^0$ the space of all continuous $M$-equivariant maps from $\wt B$ to $V_x: = \bb C^{d_x}$, that is,
the space of all continuous maps $\psi: \wt B \to V_x$ satisfying, for any $m \in M$ and $\wt b \in \wt B$, 
\begin{align}\label{equivariant maps psi 001}
\psi(\wt b m) = \rho_x(m)^{-1} \psi(\wt b). 
\end{align}
The following lemma gathers several classical results from the representation theory of compact groups.

\begin{lemma}\label{Lem-representation-compact-group}
Let $M$ be a metrizable compact group acting continuously on the right on the compact metrizable space $\wt B$. 
Then, the following statements hold. 
\begin{enumerate}
\item
For each $x \in \wh M$, 
the operator $\pi_x$ is a continuous projection of $C^0(\wt B)$ whose range is the closed subspace $\bigcap_{\substack{x' \in \wh M \\ x' \neq x}} \ker \pi_{x'}$.

\item
The linear subspace $\bigoplus_{x \in \wh M} \pi_x (C^0(\wt B))$ is dense in $C^0(\wt B)$. 

\item
The linear map $f \otimes \psi \mapsto \langle f, \psi \rangle: V_x^* \otimes W_x^0 \to C^0(\wt B)$ is a continuous surjection onto $\pi_x (C^0(\wt B))$. 
\end{enumerate}
\end{lemma}

In the above statement, 
we write $V_x^*$ for the dual vector space of $V_x = \bb C^{d_x}$, and
for $f \in V_x^*$ and $\psi \in W_x^0$, we have denoted by $\langle f, \psi \rangle$ the continuous function 
$\wt b \mapsto \langle f, \psi(\wt b) \rangle$ on $\wt B$. 

\begin{proof}[Proof of Lemma \ref{Lem-representation-compact-group}]
See \cite[Part I, Chapter 2]{Ser77} or \cite[Chapter III, Section 7 and Chapter VII, Section 9]{Sim96}. 
\end{proof}

Now we show that the operator $P_{\mu}$ is block-diagonal in the direct sum $\bigoplus_{x \in \wh M} \pi_x (C^0(\wt B))$. 
For $x \in \wh M$, $\psi \in W_x^0$ and $\wt b \in \wt B$, define the operator $P_{\mu, x}$ as follows: 
\begin{align}\label{def operator P mu x}
P_{\mu, x} \psi(\wt b) = \int_{G} \psi(g \wt b) \mu(dg). 
\end{align}
Since the actions of $G$ and $M$ on $\wt B$ commute with each other, we get the following lemma. 

\begin{lemma}\label{Lem-two-operators}
Let $\mu$ be a Borel probability measure on $G$. 
Then, for any $x \in \wh M$, 
the operator $P_{\mu, x}$ is bounded on $W_x^0$, and for any $f \in V_x^*$ and $\psi \in W_x^0$, we have 
\begin{align*}
P_{\mu} \langle f, \psi \rangle = \langle f, P_{\mu, x} \psi \rangle. 
\end{align*}
\end{lemma}

\begin{proof}
We claim that, for any $\psi \in W_x^0$, 
the function $P_{\mu, x} \psi$ is still an element of $W_x^0$, that is, $P_{\mu, x} \psi$ is still an $M$-equivariant function
from $\wt B$ to $V_x$. 
Indeed, by \eqref{equivariant maps psi 001} and the fact that the actions of $G$ and $M$ on $\wt B$ commute, 
we have that, for any $\wt b \in \wt B$ and $m \in M$, 
\begin{align*}
P_{\mu, x} \psi(\wt b m)  
& = \int_{G} \psi(g (\wt b m)) \mu(dg)   \notag\\
& = \int_{G} \psi((g \wt b) m) \mu(dg)  \notag\\
& = \rho_x(m)^{-1} \int_{G}  \psi(g \wt b) \mu(dg)  \notag\\
& = \rho_x(m)^{-1} P_{\mu, x} \psi(\wt b). 
\end{align*}
Thus, the operator $P_{\mu, x}$ preserves the space $W_x^0$. 
Now $W_x^0$ is a closed subspace of $C^0(\wt B, V_x)$ and $P_{\mu, x}$ may be seen as the restriction to $W_x^0$ of a bounded operator
of $C^0(\wt B, V_x)$. Therefore, $P_{\mu, x}$ is bounded on $W_x^0$. 
The last formula is obvious. 
\end{proof}

To prove Theorem \ref{Thm-equdistribution}, it suffices to study the powers of the operator $P_{\mu, x}$ in $W_x^0$
for a fixed element $x \in \wh M$.

\begin{proof}[Proof of Theorem \ref{Thm-equdistribution}]
Recall that in Subsection \ref{Section-Quotient vector bundles} 
we have associated with $x \in \wh M$ the vector bundle $E_x = B \times_{\rho_x} \bb C^{d_x}$ over $B$.
By Corollary \ref{Corollary-isometric Lipschitz bundle action} and the moment assumption \eqref{moment-assumption-01},
the moment condition (i.e., $\bb E ({\rm Lip}_{E_x}(g_1)^{\alpha} ) < \infty$ for some constant $\alpha >0$)
 required in Corollary \ref{Corollary-001} is satisfied for the vector bundle $E_x \to B$. 
By Corollary \ref{Corollary-001}, we know that the Birkhoff averages of the natural operator associated with $\mu$ on $\bf \Gamma^0(E_x)$ 
converge strongly in the space of bounded operators on $\bf \Gamma^0(E_x)$. 
Thanks to Lemma \ref{Lem-closed injection-001} and Corollary \ref{Cor-closed injection-001},
this tells us that the Birkhoff averages of the operator $P_{\mu, x}$ on $W_x^0$ 
converge strongly in the space of bounded operators on $W_x^0$. 
By part (3) of Lemma \ref{Lem-representation-compact-group} and Lemma \ref{Lem-two-operators}, 
this implies that the Birkhoff averages of the operator $P_{\mu}$ on $\pi_x (C^0(\wt B))$ 
converge strongly in the space of bounded operators on $\pi_x (C^0(\wt B))$. 
By part (2) of Lemma \ref{Lem-representation-compact-group}, 
as all powers of the operator $P_{\mu}$ in $C^0(\wt B)$ have norm at most $1$, 
we obtain that, for any $\varphi \in C^0(\wt B)$, the sequence $\frac{1}{n} \sum_{k=0}^{n-1}  P_{\mu}^k \varphi$ is a uniform Cauchy sequence 
and the conclusion follows. 
\end{proof}


\subsection{Eigenvalues of modulus one}
In order to get the more precise statement in Corollary \ref{Cor-equdistribution}, we need to describe the eigenvalues of modulus one
of the operators $P_{\mu, x}$ for any representation $x \in \wh M$. 
This is a refinement of the arguments in \cite[Section 15.2]{BQ16b}.

\begin{proposition}\label{Propo-eigenvalue-one}
Let $\mu$ be a Borel probability measure on $G$. 
Assume that the action of $G$ on $\wt B$ satisfies the moment condition \eqref{moment-assumption-01} and 
that the action of $G$ on $B$ satisfies the contraction property \eqref{contraction-property-on-B}.  
Suppose that, for every $\wt b \in \wt B$ with $\eta(\wt b) \in \supp \nu$,   we have $\wt b M \subset \overline{\Gamma_{\mu} \wt b}$. 
Then, there exists a unique closed coabelian normal subgroup $M_0^{\mu}$ of $M$ 
such that, 
for any $x \in \wh M$, the operator $P_{\mu, x}$ admits an eigenvalue with modulus one in $W^0_x$ if and only if 
$\rho_x$ is the one-dimensional representation associated with a character $\chi_x$ of $M$ which is trivial on $M_0^{\mu}$. 
\end{proposition}

In the course of the proof of this proposition, we shall need the following lemma. 
For any $x \in \wh M$ and $\gamma \in (0, 1]$, 
define $W_{x}^{\gamma}$ (respectively $\overline{W}_{x}^{\gamma}$)
as the space of all $\gamma$-H\"older continuous functions $\varphi: \wt B \to \bb C^{d_x}$ (respectively $\varphi: \eta^{-1}(\supp \nu) \to \bb C^{d_x}$)
such that, for any $\wt b \in \wt B$ (respectively $\wt b \in \eta^{-1}(\supp \nu)$) and $m \in M$, 
we have $\varphi(\wt b m) = \rho_x(m)^{-1} \varphi(\wt b)$. 

\begin{lemma}\label{Lemma-surjectivity}
For any $x \in \wh M$ and $\gamma \in (0, 1]$, 
 the restriction map $W_{x}^{\gamma} \to \overline{W}_{x}^{\gamma}$ is surjective and hence open.
\end{lemma}

\begin{proof}
Take a function $\varphi \in \overline{W}_{x}^{\gamma}$. 
By general facts on function spaces (see for example \cite[Lemma 11.8]{BQ16b}), 
there exists a $\gamma$-H\"older continuous function $\psi: \wt B \to \bb C^{d_x}$  
whose restriction to $\eta^{-1}(\supp \nu)$ is the function $\varphi$.  
For $\wt b \in \wt B$, we now define 
\begin{align*}
\psi_1(\wt b) =  \int_{M} \rho_x(m) \psi(\wt b m) dm, 
\end{align*}
where $dm$ is the normalized Haar measure of $M$. 
Since $M$ preserves the distance on $\wt B$ and the norm on $\bb C^{d_x}$, 
the function $\psi_1$ is also $\gamma$-H\"older continuous on $\wt B$. 
Moreover, we have, for any $m_0 \in M$, 
\begin{align*}
\psi_1(\wt b m_0) 
& = \int_{M} \rho_x(m) \psi(\wt b m_0 m) dm  \notag\\
& = \rho_x(m_0)^{-1} \int_{M} \rho_x(m_0 m) \psi(\wt b m_0 m) dm  \notag\\
& = \rho_x(m_0)^{-1}  \psi_1(\wt b), 
\end{align*}
so that $\psi_1 \in W_{x}^{\gamma}$. 
Besides, it holds that, for any $\wt b \in \eta^{-1}(\supp \nu)$, 
\begin{align*}
\psi_1(\wt b) = \int_{M} \rho_x(m) \psi(\wt b m) dm = \int_{M} \rho_x(m) \varphi(\wt b m) dm = \varphi(\wt b). 
\end{align*}
This shows that the restriction map $W_{x}^{\gamma} \to \overline{W}_{x}^{\gamma}$ is surjective. 
By the open mapping theorem, it follows that this restriction map is open.  
\end{proof}

\begin{proof}[Proof of Proposition \ref{Propo-eigenvalue-one}]
Let $x \in \wh M$ be such that $P_{\mu, x}$ admits an eigenvalue with modulus one in $W^0_x$. 
The main step of the proof is to show that $d_x = 1$. 
Indeed, let $\varphi \in W_x^0$ be a non-zero element satisfying that there exists $\lambda \in \bb C$ with $|\lambda| = 1$
such that $P_{\mu, x} \varphi = \lambda \varphi$. 
Up to rescaling, we can assume that $\sup_{\wt b \in \wt B} \|\varphi(\wt b)\| = 1$. 
Then, set 
\begin{align*}
\wt X = \left\{ \wt b \in \wt B: \|\varphi(\wt b)\| = 1 \right\}.
\end{align*}
Note that, for any $\wt b \in \wt X$ and $m \in M$, we have 
\begin{align*}
\|\varphi(\wt b m)\| = \| \rho_x(m)^{-1} \varphi(\wt b) \| = 1 
\end{align*}
and hence $\wt b m \in \wt X$. 
Thus, the set $\wt X$ is $M$-invariant. 
Moreover, from \eqref{def operator P mu x} and the equation $P_{\mu, x} \varphi = \lambda \varphi$ we derive that, 
for any $g \in \supp \mu$, we have $g \wt X \subseteq \wt X$ and for any $\wt b \in \wt X$, it holds that $\varphi(g \wt b) = \lambda \varphi(\wt b)$.
Therefore, the map $\bar \varphi: \wt b \mapsto \bb C \varphi(\wt b)$ from $\wt X$ to $\bb P(\bb C^{d_x})$  is $\Gamma_{\mu}$-invariant.
As $\wt X$ is closed, both $M$-invariant and $\Gamma_{\mu}$-invariant, 
we get that $\wt X$ contains the set $\eta^{-1}(\supp \nu)$, by Lemma \ref{Lemma-unqiue-sta-measure}. 
Let $\wt b$ be in $\eta^{-1}(\supp \nu)$, then, by assumption, we have $\wt b M \subset \overline{\Gamma_{\mu} \wt b}$
and therefore, the map $\bar \varphi$ is constant on $\wt b M$. 
Consequently, since $\bar \varphi$ is $M$-equivariant, this means that $\bb C \varphi(\wt b)$ is a fixed point of $M$ in $\bb P(\bb C^{d_x})$. 
As the representation $\rho_x: M \to {\rm GL}_{d_x}(\bb C)$ is irreducible, 
we get that $\bb C^{d_x} = \bb C \varphi(\wt b)$ and hence $d_x = 1$, as required.

Thus, the set $\Lambda$ of all the representations $x \in \wh M$ such that 
$P_{\mu, x}$ has an eigenvalue of modulus one is  a subset of the group of characters of $M$,
where we identify each character with the associated one-dimensional representation. 
To conclude, by the Pontryagin duality theorem (see  \cite[Proposition 38]{Mor77}), 
it remains to show that this set $\Lambda$ is a subgroup of the group of characters, which we will proceed to prove below. 

Denote by $\bb U$ the group of complex numbers with modulus one. 
We fix a character $\chi$ of the group $M$, that is, $\chi$ is a continuous homomorphism from $M$ to $\bb U$. 
We claim that 
$P_{\mu, \chi}$ has an eigenvalue of modulus one in $W_{\chi}^0$ if and only if 
there exists a continuous function $\varphi: \eta^{-1}(\supp \nu) \to \bb U$
and $\lambda \in \bb U$
such that, for every $g \in \supp \mu$, $m \in M$ and $\wt b \in \eta^{-1}(\supp \nu)$, we have 
\begin{align}\label{two equations varphi}
\varphi(\wt b m) = \chi(m)^{-1} \varphi(\wt b) 
\quad  \mbox{and}  \quad 
\varphi(g \wt b) = \lambda \varphi(\wt b). 
\end{align}
Indeed, if $P_{\mu, \chi}$ has an eigenvalue of modulus one in $W_{\chi}^0$, then, by the above discussions, such a function exists. 
Conversely, assume that $P_{\mu, \chi}$ has no eigenvalue of modulus one in $W_{\chi}^0$, 
then it has no eigenvalue of modulus one in $W_{\chi}^{\gamma} \subset W_{\chi}^0$. 
By Corollary \ref{Corollary-isometric Lipschitz bundle action} and the moment assumption \eqref{moment-assumption-01},
the moment condition (i.e., $\bb E ({\rm Lip}_{E_{\chi}}(g_1)^{\alpha} ) < \infty$ for some constant $\alpha >0$) 
in Theorem \ref{Thm-spectral-gap} is satisfied by the complex line bundle $E_{\chi} := B  \times_{\chi} \bb C \to B$. 
By Theorem \ref{Thm-spectral-gap}, 
Lemma \ref{Lem-closed injection-001} and Corollary \ref{Cor-closed injection-001}, 
the operator $P_{\mu, \chi}$ has essential spectral radius strictly smaller than one in $W_{\chi}^{\gamma}$. 
Since $P_{\mu, \chi}$ has no eigenvalue of modulus one in $W_{\chi}^{\gamma}$,
 it follows that $P_{\mu, \chi}$ has spectral radius strictly smaller than one in $W_{\chi}^{\gamma}$. 
Recall that $\overline{W}_{\chi}^{\gamma}$ is the space of all $\gamma$-H\"older continuous functions $\varphi: \eta^{-1}(\supp \nu) \to \bb C$
such that, for any $\wt b \in \eta^{-1}(\supp \nu)$ and $m \in M$, 
$\varphi(\wt b m) = \chi(m)^{-1} \varphi(\wt b)$. 
Then, by Lemma \ref{Lemma-surjectivity}, 
the restriction map $W_{\chi}^{\gamma} \to \overline{W}_{\chi}^{\gamma}$ is surjective and open.
Therefore, the operator $P_{\mu, \chi}$ has spectral radius strictly smaller than one in $\overline{W}_{\chi}^{\gamma}$. 
In particular, this implies that, for any $\varphi \in \overline{W}_{\chi}^{\gamma}$, we have $P_{\mu, \chi}^n \varphi \xrightarrow[n\to\infty]{} 0$. 
By \eqref{P-mu-n-001} in the proof of Corollary \ref{Corollary-001}, 
we obtain that this convergence also holds for any $\varphi \in \overline{W}_{\chi}^{0}$, 
so that there cannot exist $0 \neq \varphi \in \overline{W}_{\chi}^{0} $ and $\lambda \in \bb U$ 
such that $\varphi(g \wt b) = \lambda \varphi(\wt b)$ for any $g \in \supp \mu$ and $\wt b \in \eta^{-1}(\supp \nu)$. 
This completes the proof of the claim. 

Now, let $\chi$ and $\chi'$ be the characters of $M$
such that $P_{\mu, \chi}$ and $P_{\mu, \chi'}$ admit eigenvalues of modulus one in $W_{\chi}^0$
and $W_{\chi'}^0$, respectively. 
Then, using the above claim, we take continuous functions $\varphi, \varphi': \eta^{-1}(\supp \nu) \to \bb U$ 
and $\lambda, \lambda' \in \bb U$ as in \eqref{two equations varphi}. 
Set 
\begin{align*}
\chi'' = \chi^{-1} \chi',  
\quad 
\varphi'' = \varphi^{-1} \varphi'
\quad \mbox{and}  \quad
\lambda'' = \lambda^{-1} \lambda'. 
\end{align*}
Then, $\varphi'': \eta^{-1}(\supp \nu) \to \bb U$ is also a continuous function 
and, for every $g \in \supp \mu$, $m \in M$ and $\wt b \in \eta^{-1}(\supp \nu)$, 
we have 
\begin{align*}
\varphi''(\wt b m) = \chi''(m)^{-1} \varphi''(\wt b)
\quad \mbox{and}  \quad 
\varphi''(g \wt b) = \lambda'' \varphi''(\wt b). 
\end{align*}
Therefore, again by the above claim, the operator $P_{\mu, \chi''}$ admits an eigenvalue of modulus one in $W_{\chi''}^0$. 
This shows that $\Lambda$ is indeed a subgroup of the group of characters of $M$, as required. 

Finally, by  \cite[Proposition 38]{Mor77}, 
it suffices to take $M_0^{\mu}$ as the annihilator of $\Lambda$ in $M$, 
that is, 
\begin{align}\label{def annihilator M0}
M_0^{\mu} = \{ m \in M:  \chi(m) = 1,  \,  \forall \, \chi \in \Lambda \}.
\end{align}
Then $M_0^{\mu}$ is a closed normal subgroup of $M$ 
and the quotient $M/M_0^{\mu}$ is Abelian. 
The conclusion of the proposition follows. 
\end{proof}

Inspecting the proof of Proposition \ref{Propo-eigenvalue-one}, we obtain the following additional properties. 

\begin{corollary}\label{Cor-eigenvalue-modulus-one}
Let $\mu$ be a Borel probability measure on $G$. 
Assume that the action of $G$ on $\wt B$ satisfies the moment condition \eqref{moment-assumption-01} and 
that the action of $G$ on $B$ satisfies the contraction property \eqref{contraction-property-on-B}.  
Suppose that, for every $\wt b \in \wt B$ with $\eta(\wt b) \in \supp \nu$,   we have $\wt b M \subset \overline{\Gamma_{\mu} \wt b}$. 
Fix $\wt b_0 \in \wt B$ with $\eta(\wt b_0) \in \supp \nu$. 
Then, for every character $\chi$ of the group $M/M_0^{\mu}$, the operator $P_{\mu, \chi}$ 
has a unique eigenvalue $\lambda_{\chi}$ of modulus one in the space $W_{\chi}^{0}$. 

Moreover, there exists a unique associated eigenvector $\varphi_{\chi} \in W_{\chi}^{0}$
such that $\varphi_{\chi}(\wt b_0) = 1$ and $\varphi_{\chi}$ is H\"older continuous on $\wt B$.  
Also, the generalized eigenspace associated with the eigenvalue $\lambda_{\chi}$ for the operator $P_{\mu, \chi}$ acting on the space $W_{\chi}^0$ 
is equal to $\bb C \varphi_{\chi}$, that is, this eigenvalue is multiplicity free. 

Finally, the restriction $\varphi_{\chi}|_{\eta^{-1}(\supp \nu)}$ is the unique continuous function on $\eta^{-1}(\supp \nu)$
such that $\varphi_{\chi}(\wt b_0) = 1$ and for every $m \in M$,  $g \in \supp \mu$ and $\wt b \in \eta^{-1}(\supp \nu)$, we have 
\begin{align}\label{property varphi phi chi m}
\varphi_{\chi}(\wt b m) = \chi(m)^{-1} \varphi_{\chi}(\wt b)
\quad \mbox{and} \quad
\varphi_{\chi}(g \wt b) = \lambda_{\chi} \varphi_{\chi}(\wt b). 
\end{align}
\end{corollary}

\begin{proof}
Let $\lambda$ and $\lambda'$ be eigenvalues of modulus one of the operator $P_{\mu, \chi}$ in $W_{\chi}^{0}$,
and let $\varphi$ and $\varphi'$ be associated non-zero eigenvectors in $W_{\chi}^{0}$. 
It follows from the proof of Proposition \ref{Propo-eigenvalue-one} that both $\varphi$ and $\varphi'$ have constant non-zero modulus 
on $\eta^{-1}(\supp \nu)$ and that for every $g \in \supp \mu$ and $\wt b \in \eta^{-1}(\supp \nu)$, we have 
\begin{align*}
\varphi(g \wt b) = \lambda \varphi(\wt b)
\quad \mbox{and} \quad
\varphi'(g \wt b) = \lambda' \varphi'(\wt b). 
\end{align*}
For $\wt b \in \eta^{-1}(\supp \nu)$, we set $\varphi''(\wt b) = \frac{\varphi(\wt b)}{\varphi'(\wt b)}$. 
Then, for every $g \in \supp \mu$, $m \in M$ and $\wt b \in \eta^{-1}(\supp \nu)$, we get 
\begin{align}\label{equality equivalence varphi001}
\varphi''(g \wt b) = \frac{\lambda}{\lambda'} \varphi''(\wt b). 
\end{align}
Besides, as $\varphi(\wt b m) = \chi(m)^{-1} \varphi(\wt b)$ and $\varphi'(\wt b m) = \chi(m)^{-1} \varphi'(\wt b)$,
we have $\varphi''(\wt b m) = \varphi''(\wt b).$ 
This equality implies that we may write $\varphi'' = \overline{\varphi''} \circ \eta$, 
where $\overline{\varphi''}$ is a continuous function on $\supp \nu$.  
The equality \eqref{equality equivalence varphi001} may be rewritten as follows: 
 for every $g \in \supp \mu$ and $b \in \supp \nu$,
 we have $\overline{\varphi''}(g b) = \frac{\lambda}{\lambda'} \overline{\varphi''}(b)$. 
 Since the contraction property \eqref{contraction-property-on-B} holds, 
 we get that $\lambda = \lambda'$ and $\overline{\varphi''}$ is constant on $\supp \nu$, 
 which means that the eigenvalue is unique and that the associated eigenspace is one-dimensional. 
 As the powers of the operator $P_{\mu}$ acting on the space $C^0(\wt B)$ have uniformly bounded norm,
 so have the powers of the operator $P_{\mu, \chi}$ acting on the space $W_{\chi}^0$. 
 Therefore, the generalized eigenspace associated with the eigenvalue $\lambda_{\chi}$ coincides with the eigenspace. 
 The remainder of the corollary is a consequence of the proof of Proposition \ref{Propo-eigenvalue-one}. 
\end{proof}

We now proceed to the proof of Corollary \ref{Cor-equdistribution}. 
We shall need the following intermediate results. 

\begin{lemma}\label{Lem-dense-Con-Holder}
For any $x \in \wh M$ and $\gamma \in (0, 1]$, the space $W_x^{\gamma}$ is dense in $W_x^{0}$. 
\end{lemma}

\begin{proof}
For any $\ee >0$ and $\varphi \in W_x^{0}$, 
there exists a $\gamma$-H\"older continuous function $\psi: \wt B \to \bb C^{d_x}$  
such that $\|\psi(\wt b) - \varphi(\wt b)\| < \ee$ for all $\wt b \in \wt B$. 
For $\wt b \in \wt B$, we now define 
\begin{align*}
\psi_1(\wt b) =  \int_{M} \rho_x(m) \psi(\wt b m) dm, 
\end{align*}
where $dm$ is the normalized Haar measure of $M$.  
As in the proof of Lemma \ref{Lemma-surjectivity}, 
the function $\psi_1$ is also $\gamma$-H\"older continuous on $\wt B$, 
and $\psi_1(\wt b m_0) = \rho_x(m_0)^{-1} \psi_1(\wt b)$
 for any $m_0 \in M$, 
so that $\psi_1 \in W_{x}^{\gamma}$. 
Besides, as $\varphi$ is $\rho_x$-equivariant
(i.e., $\varphi(\wt b m) = \rho_x(m)^{-1} \varphi(\wt b)$ for any $\wt b \in \wt B$ and $m \in M$), 
we get that, for any $\wt b \in \wt B$, 
\begin{align*}
\|\psi_1(\wt b) - \varphi(\wt b)\| 
& =  \left\|  \int_{M} \rho_x(m) \psi(\wt b m) dm -  \varphi(\wt b) \right\|  \notag\\
& = \left\|  \int_{M} \rho_x(m) (\psi(\wt b m) -  \varphi(\wt b m)) dm  \right\|  \notag\\
& < \ee   \int_{M} \| \rho_x(m) \| dm  = \ee. 
\end{align*}
Since $\ee >0$ is arbitrary, the conclusion follows. 
\end{proof}

\begin{lemma}\label{Lem-Riesz-repre}
Let $X, Y$ be compact metrizable spaces, and let $W \subset C^0(X)$ be a dense subspace 
that contains the constant functions and is stable under complex conjugation. 
Assume that we are given a linear operator $Q: W \to C^0(Y)$ such that $Q 1 = 1$ and $Q \varphi \geq 0$ 
for every $\varphi \geq 0$ in the space $W$. 
Then $Q$ can be extended to a continuous operator $Q: C^0(X) \to C^0(Y)$. 

Moreover, there exists a unique continuous map $y \mapsto q_y$ from $Y$ to the set of Borel probability measures on $X$, where the latter is equipped with the weak-$*$ topology
such that, for every $\varphi \in C^0(X)$ and $y \in Y$, 
\begin{align*}
Q \varphi(y) = \int_X \varphi(x) \,  q_y(dx). 
\end{align*}
\end{lemma}

\begin{proof}
The positivity ($Q \varphi \geq 0$ for every non-negative $\varphi \in W$) 
and $Q 1 = 1$ imply that, for any $\varphi \in W$, one has $\|Q \varphi \|_{\infty} \leq \| \varphi \|_{\infty} Q 1 = \| \varphi \|_{\infty}$. 
Thus $Q$ is a contraction on $W$ and can be continuously extended as a contraction from $C^0(X)$ to $C^0(Y)$, 
giving the first claim.

For each $y \in Y$, the map $\varphi \mapsto Q\varphi(y)$ is a positive linear functional on $C^0(X)$ with norm $1$. 
By the Riesz representation theorem (see for example \cite[Theorems 2.14 and 2.18]{Rud87}), 
there exists a unique Borel probability measure $q_y$ on $X$ representing this functional. 
The continuity of the map $y \mapsto q_y$ in the weak-$*$ topology 
follows from the fact that for each $\varphi \in C^0(X)$, the function $y \mapsto Q\varphi(y)$ is continuous. 
\end{proof}

\begin{proof}[Proof of Corollary \ref{Cor-equdistribution}]
We begin by constructing the element $m_0^{\mu}$. 
We keep the notation of Corollary \ref{Cor-eigenvalue-modulus-one},
where $M_0^{\mu}$ is defined by \eqref{def annihilator M0}. 
The uniqueness properties established in Corollary \ref{Cor-eigenvalue-modulus-one}
imply that, for any characters $\chi, \chi'$ of $M/M_0^{\mu}$, 
we have $\lambda_{\chi \chi'} = \lambda_{\chi} \lambda_{\chi'}$, 
where $\lambda_{\chi}$ is the unique eigenvalue of modulus one for 
the operator $P_{\mu, \chi}$ acting on the space $W_{\chi}^{0}$. 
Thus the map $\chi \mapsto \lambda_{\chi}$
is a character of the dual group of $M/M_0^{\mu}$.
By the Pontryagin--van Kampen duality theorem \cite[Chapter 4]{Mor77}, 
there exists a unique element $m_0^{\mu} \in M/M_0^{\mu}$ such that, for any character $\chi$ of $M/M_0^{\mu}$,
we get 
\begin{align}\label{def m0 mu}
\chi(m_0^{\mu}) = \lambda_{\chi}. 
\end{align}
In the same way, for any characters $\chi, \chi'$ of $M/M_0^{\mu}$, 
the uniqueness properties also give that 
on the set $\eta^{-1}(\supp \nu)$, 
the eigenvectors satisfy $\varphi_{\chi \chi'} = \varphi_{\chi} \varphi_{\chi'}$. 
Consequently, there exists a unique H\"older continuous map $\theta: \eta^{-1}(\supp \nu) \to M/M_0^{\mu}$ such that, 
for every character $\chi$ of $M/M_0^{\mu}$ and every $\wt b \in \eta^{-1}(\supp \nu)$, we have 
\begin{align}\label{def-theta-wt-b}
\chi(\theta(\wt b)) = \varphi_{\chi}(\wt b).
\end{align}
By \eqref{property varphi phi chi m}, for any character $\chi$ of $M/M_0^{\mu}$, 
any $m \in M$, any $g \in \supp \mu$ and any $\wt b \in \eta^{-1}(\supp \nu)$, 
we have $\varphi_{\chi}(\wt b m) = \chi(m)^{-1} \varphi_{\chi}(\wt b)$ and $\varphi_{\chi}(g \wt b) = \lambda_{\chi} \varphi_{\chi}(\wt b)$. 
Hence, from \eqref{def m0 mu} and \eqref{def-theta-wt-b}, it follows that 
\begin{align*}
\theta(\wt b m) = m^{-1} \theta(\wt b)
\quad \mbox{and}  \quad
\theta(g \wt b) = m_0^{\mu} \theta(\wt b). 
\end{align*}
In particular, define $M_1^{\mu}$ as the closure in $M/M_0^{\mu}$ of the subgroup spanned by $m_0^{\mu}$. 
Fix $\wt b \in \eta^{-1}(\supp \nu)$ with $\theta(\wt b) = e$, then the above relations imply that the set $\theta^{-1}(M_1^{\mu})$ 
contains the orbit closure $\overline{\Gamma_{\mu} \wt b}$. 
Thus, using the assumption $\wt b M \subset \overline{\Gamma_{\mu} \wt b}$, 
we get that $\theta^{-1}(M_1^{\mu})$ contains $\wt b M$. 
Therefore, as $\theta(\wt b M) = M/M_0^{\mu}$, we obtain $M_1^{\mu} = M/M_0^{\mu}$,
which is to say that $m_0^{\mu}$ spans a dense subgroup of $M/M_0^{\mu}$.

Thanks to Corollary \ref{Corollary-isometric Lipschitz bundle action} and the moment assumption \eqref{moment-assumption-01},
we may find $\gamma>0$ such that the conclusion of  Theorem \ref{Thm-spectral-gap} 
applies for all the vector bundles $B  \times_{\rho_x} \bb C^{d_x} \to B$,
where $x \in \wh M$. 
Next we will construct, for any representation $x \in \wh M$, 
a linear operator $Q_x: \pi_x(C^{\gamma}(\wt B)) \to C^0(M/M_0^{\mu})$ as follows,
where $\pi_x$ is the operator defined by \eqref{def operator pi-x}. 
If $P_{\mu, x}$ does not admit an eigenvalue of modulus one in $W_x^0$, then we set $Q_x = 0$. 
Note that, as the space $\pi_x(C^{\gamma}(\wt B))$ is $P_{\mu}$-invariant, on this space we have $Q_x P_{\mu} = 0$. 
Otherwise, if $P_{\mu, x}$ admits an eigenvalue of modulus one in $W_x^0$, 
then, by Proposition \ref{Propo-eigenvalue-one}, the representation $\rho_x$ is the one-dimensional representation
associated to a character $\chi$ of $M/M_0^{\mu}$. 
By Theorem \ref{Thm-spectral-gap}, Lemma \ref{Lem-closed injection-001} and Corollary \ref{Cor-closed injection-001}, 
the essential spectral radius of $P_{\mu, x} = P_{\mu, \chi}$ in $W_x^{\gamma}$ is strictly smaller than one. 
By Corollary \ref{Cor-eigenvalue-modulus-one}, $P_{\mu, \chi}$ admits $\lambda_{\chi}$ as its unique eigenvalue of modulus one
and the associated generalized eigenspace is the line $\bb C \varphi_{\chi}$. Therefore, there exists a unique bounded linear functional $\zeta_{\chi}$
of $W_{\chi}^{\gamma}$ such that $\zeta_{\chi}(\varphi_{\chi}) = 1$
and $P_{\mu, \chi} \ker(\zeta_{\chi}) \subset \ker(\zeta_{\chi})$. 
In particular, if $\chi$ is the trivial character of $M/M_0^{\mu}$, then $W_{\chi}^{\gamma}$ may be identified with the space of $\gamma$-H\"older continuous functions on $B$,
and thanks to Lemma \ref{Lemma-unqiue-sta-measure}, the associated linear functional $\zeta_0$ is integration with respect to the stationary measure $\nu$. 
In general, for any character $\chi$ of $M/M_0^{\mu}$ and any function $\varphi \in W_{\chi}^{\gamma}$, we set 
\begin{align}\label{formula Q chi varphi}
Q_{\chi} \varphi = \zeta_{\chi}(\varphi) \chi \in C^0(M/M_0^{\mu}). 
\end{align}
Note that, by \eqref{def m0 mu} and \eqref{formula Q chi varphi}, we get that, for any $m \in M$,  
\begin{align}\label{equivariance-chi-001}
Q_{\chi} P_{\mu, \chi} \varphi(m) 
& = \zeta_{\chi}(P_{\mu, \chi}\varphi) \chi(m) 
= \lambda_{\chi} \zeta_{\chi}(\varphi) \chi(m)  \notag\\
& = \chi(m_0^{\mu}) \zeta_{\chi}(\varphi) \chi(m)
= Q_{\chi} \varphi(m_0^{\mu} m). 
\end{align}

We claim that for any $x \in \wh M$, any $\varphi \in \pi_x(C^{\gamma}(\wt B))$ and any $\wt b \in \eta^{-1}(\supp \nu)$,
we have 
\begin{align*}
 P_{\mu, x}^n \varphi(\wt b) -  Q_x \varphi \big( (m_0^{\mu})^n \theta(\wt b) \big) 
\xrightarrow[n\to\infty]{}  0. 
\end{align*}
Indeed, if $P_{\mu, x}$ does not admit an eigenvalue of modulus one in $W_x^0$, then $P_{\mu, x}$
has spectral radius strictly smaller than one in $W_x^{\gamma}$, 
and hence, for any $\wt b \in \wt B$,
\begin{align}\label{convergence-P-mu-easy}
P_{\mu, x}^n \varphi(\wt b)  \xrightarrow[n\to\infty]{}   0. 
\end{align}
Else, the representation $\rho_x$ is the one-dimensional representation
associated to a character $\chi$ of $M/M_0^{\mu}$. Then, $P_{\mu, \chi}$
has spectral radius strictly smaller than one in $\ker (\zeta_{\chi})$. 
As $\varphi - \zeta_{\chi}(\varphi) \varphi_{\chi} \in \ker (\zeta_{\chi})$, we get, uniformly in $\wt b \in \wt B$,
\begin{align}\label{Convergence-P-mu}
P_{\mu, \chi}^n \varphi(\wt b) -  \zeta_{\chi}(\varphi) \lambda_{\chi}^n \varphi_{\chi}(\wt b) 
\xrightarrow[n\to\infty]{} 0. 
\end{align}
The conclusion follows, since by the definition of $m_0^{\mu}$ and $\theta$ (cf.\  \eqref{def m0 mu} and \eqref{def-theta-wt-b}), 
for any $\wt b \in \eta^{-1}(\supp \nu)$, 
\begin{align}\label{property Q chi m0 mu}
Q_{\chi} \varphi ((m_0^{\mu})^n \theta(\wt b)) =  \zeta_{\chi}(\varphi) \chi((m_0^{\mu})^n \theta(\wt b)) = \zeta_{\chi}(\varphi) \lambda_{\chi}^n \varphi_{\chi}(\wt b). 
\end{align}

To summarize, now we have built a linear operator 
\begin{align}\label{def operator Q}
Q = \bigoplus_{x \in \wh M} Q_x: \bigoplus_{x \in \wh M} \pi_x (C^{\gamma}(\wt B)) \to C^0(M/M_0^{\mu}),
\end{align}
where, by direct sum, we mean the algebraic direct sum. 
Then, by \eqref{convergence-P-mu-easy}, \eqref{Convergence-P-mu} and \eqref{property Q chi m0 mu},
 we have, for any $\varphi \in \bigoplus_{x \in \wh M} \pi_x (C^{\gamma}(\wt B))$, uniformly in $\wt b \in \eta^{-1}(\supp \nu)$,
\begin{align*}  
P_{\mu}^n \varphi(\wt b) - Q\varphi \big( (m_0^{\mu})^n \theta(\wt b) \big) 
\xrightarrow[n\to\infty]{} 0. 
\end{align*}
Note that, in particular, if $\varphi$ is real and nonnegative, so is $Q \varphi$. 
Indeed, let us fix $\wt b \in \eta^{-1}(\supp \nu)$. As $m_0^{\mu}$ spans a dense subgroup of $M/M_0^{\mu}$, 
for any $m \in M$, we may find a sequence $(n_k)_{k \geq 1}$ of integers such that $(m_0^{\mu})^{n_k} \theta(\wt b) \to m$ as $k \to \infty$. 
Then we get 
$$Q \varphi(m) = \lim_{k \to \infty} Q\varphi \big( (m_0^{\mu})^{n_k} \theta(\wt b) \big) = \lim_{k \to \infty} P_{\mu}^{n_k} \varphi(\wt b) \geq 0.$$
By Lemmas \ref{Lem-representation-compact-group} and \ref{Lem-dense-Con-Holder}, 
the space $\bigoplus_{x \in \wh M} \pi_x (C^{\gamma}(\wt B))$ is dense in $\bigoplus_{x \in \wh M} \pi_x (C^{0}(\wt B))$.
Therefore, using Lemma \ref{Lem-Riesz-repre}, we can extend $Q$ as a bounded operator  from the space $\bigoplus_{x \in \wh M} \pi_x (C^{0}(\wt B))$ to $C^0(M/M_0^{\mu})$.
Moreover, since $\|P_{\mu}\| \leq 1$ in $C^0(\wt B)$, we still get that, for any $\varphi \in C^0(\wt B)$, uniformly in $\wt b \in \eta^{-1}(\supp \nu)$, 
\begin{align*} 
P_{\mu}^n \varphi(\wt b) - Q\varphi \big( (m_0^{\mu})^n \theta(\wt b) \big) 
\xrightarrow[n\to\infty]{} 0. 
\end{align*}
Note that, in particular, for any $\psi  = \sum_{x \in \wh M} \psi_x \in \bigoplus_{x \in \wh M} \pi_x (C^{\gamma}(\wt B))$, 
by \eqref{formula Q chi varphi}, we derive that 
\begin{multline}\label{integral Q psi 001}
\int_{M/M_0^{\mu}}  Q \psi(m) dm 
 = \sum_{\chi \in \wh{M/M_0^{\mu}} } \int_{M/M_0^{\mu}} Q_{\chi} \psi_{\chi}(m) dm   \\ 
 =  \sum_{\chi \in \wh{M/M_0^{\mu}} } \zeta_{\chi}(\psi_{\chi}) \int_{M/M_0^{\mu}} \chi(m) dm  
 =  \zeta_{0}(\psi_{0}) 
 = \int_B  \psi_0 d\nu  
 = \int_{\wt B} \psi d\wt\nu. 
\end{multline}
The last equality follows from the definition of the $\mu$-stationary probability measure $\wt \nu$ 
introduced in Subsection \ref{section equidistribution 001}.
Indeed, by Lemma \ref{Lem-representation-compact-group}, for $\wt b \in \wt B$, we have $\psi_0(\wt b) = \int_{M} \psi(\wt b m) dm$, 
Hence, if we identify this $M$-invariant function on $\wt B$ with a function on $B$, we obtain $\int_B  \psi_0 d\nu = \int_{\wt B} \psi d\wt\nu$. 
By density, the formula \eqref{integral Q psi 001} holds for any function  $\psi \in C^0(\wt B)$. 
Thanks to the interpretation of the operator $Q$ as a family of probability measures (see Lemma \ref{Lem-Riesz-repre}), this proves equation \eqref{nu-M-M0-meas}. 
In the same way, thanks to equation \eqref{equivariance-chi-001}, for any $\varphi \in C^0(\wt B)$ and $m \in M$, we obtain 
$Q (P_{\mu} \varphi)(m) = Q \varphi(m_0^{\mu} m)$, which shows \eqref{property-F-meas}. 

To conclude, we need to extend the previous convergence to the case of any $\wt b \in \wt B$. 
To this end, for any $\wt b \in \eta^{-1}(\supp \nu)$,  we set $\theta_{\wt b} = \delta_{\theta(\wt b)}$ to be the Dirac measure concentrated on $\theta(\wt b)$,
and we will extend the map $\wt b \mapsto \theta_{\wt b}$ as a continuous map from all of $\wt B$ to the space of probability measures on $M/M_0^{\mu}$.

Recall that, in Corollary \ref{Cor-eigenvalue-modulus-one}, 
for any character  $\chi$ of $M/M_0^{\mu}$, we have defined a H\"older continuous function $\varphi_{\chi}$ on $\wt B$.
We claim that, for any $\wt b \in \wt B$, the function $\chi \mapsto \varphi_{\chi}(\wt b)$ is the Fourier transform of a Borel probability measure on $M/M_0^{\mu}$. 
This amounts to showing that, for any integer $r \geq 1$, any characters $\chi_1, \ldots, \chi_r$, any scalars $z_1, \ldots, z_r \in \bb C$ and any $\wt b \in \wt B$, we have
\begin{align*}
\sum_{1 \leq i, j \leq r} z_i \overline{z_j} \varphi_{\chi_i \overline{\chi_j}}(\wt b) \geq 0. 
\end{align*}
We choose a sequence of integers $(n_k)_{k \geq 1}$ such that $n_k \to \infty$ and $(m_0^{\mu})^{n_k} \to e$ as $k \to \infty$. 
Then we get 
\begin{align*}
\sum_{1 \leq i, j \leq r} z_i \overline{z_j} \varphi_{\chi_i \overline{\chi_j}}(\wt b) 
& = \lim_{k \to \infty} \sum_{1 \leq i, j \leq r} z_i \overline{z_j}  \chi_i((m_0^{\mu})^{n_k}) \overline{ \chi_j( (m_0^{\mu})^{n_k}) } \varphi_{\chi_i \overline{\chi_j}}(\wt b)  \notag\\
& = \lim_{k \to \infty} \sum_{1 \leq i, j \leq r} z_i \overline{z_j}  \lambda_{\chi_i \overline{\chi_j}}^{n_k}  \varphi_{\chi_i \overline{\chi_j}}(\wt b)  \notag\\
& = \lim_{k \to \infty} \sum_{1 \leq i, j \leq r} z_i \overline{z_j}  P_{\mu}^{n_k}  \varphi_{\chi_i \overline{\chi_j}}(\wt b)  \notag\\
& = \lim_{k \to \infty} \sum_{1 \leq i, j \leq r} z_i \overline{z_j}  \bb E  \varphi_{\chi_i \overline{\chi_j}}(g_{n_k} \cdots g_1 \wt b). 
\end{align*}
It follows from the contraction property \eqref{contraction-property-on-B} that, for any $\wt b \in \wt B$, we have,  as $n \to \infty$, 
$\wt{\bf d}(g_n \cdots g_1 \wt b, \eta^{-1}(\supp \nu)) \to 0$ in probability. 
Since, on the set $\eta^{-1}(\supp \nu)$, for any $1 \leq i, j \leq r$, we have $\varphi_{\chi_i \overline{\chi_j}} = \varphi_{\chi_i} \overline{\varphi_{\chi_j}}$, 
we get 
\begin{align*}
\sum_{1 \leq i, j \leq r} z_i \overline{z_j} \varphi_{\chi_i \overline{\chi_j}}(\wt b)
& = \lim_{k \to \infty} \sum_{1 \leq i, j \leq r} z_i \overline{z_j}  \bb E \left[ \varphi_{\chi_i} (g_{n_k} \cdots g_1 \wt b) \overline{\varphi_{\chi_j}} (g_{n_k} \cdots g_1 \wt b) \right] \notag\\
& = \lim_{k \to \infty}   \bb E \bigg[ \bigg|  \sum_{i = 1}^{r} z_i   \varphi_{\chi_i} (g_{n_k} \cdots g_1 \wt b)  \bigg|^2 \bigg]  
 \geq 0. 
\end{align*}
This proves the claim that, for any $\wt b \in \wt B$, the function $\chi \mapsto \varphi_{\chi}(\wt b)$ is the Fourier transform of a Borel probability measure on $M/M_0^{\mu}$,
which we denote by $\theta_{\wt b}$. 
It follows that, for any character $\chi$ of $M/M_0^{\mu}$, by \eqref{def m0 mu}, we have 
\begin{align*}
\int_{M/M_0^{\mu}} \chi(m) \theta_{\wt b}(dm)
& = \varphi_{\chi}(\wt b) 
 = \lambda_{\chi}^{-1} P_{\mu, \chi}\varphi_{\chi}(\wt b)   \notag\\
& = \chi(m_0^{\mu})^{-1}  \int_G  \varphi_{\chi}(g \wt b) \mu(dg) \notag\\
& =  \chi(m_0^{\mu})^{-1}  \int_G \bigg(  \int_{M/M_0^{\mu}} \chi(m)  \theta_{g \wt b}(dm) \bigg) \mu(dg)  \notag\\
& =    \int_G  \bigg(  \int_{M/M_0^{\mu}} \chi((m_0^{\mu})^{-1} m)  \theta_{g \wt b}(dm) \bigg) \mu(dg). 
\end{align*}
This proves \eqref{harmonicity-meas}. 
 
Thanks to \eqref{convergence-P-mu-easy} and \eqref{Convergence-P-mu}, we know that, for any $\varphi \in \bigoplus_{x \in \wh M} \pi_x (C^{\gamma}(\wt B))$, we have,
uniformly in $\wt b \in \wt B$, 
\begin{align*}
P_{\mu}^n \varphi(\wt b) -  \int_{M/M_0^{\mu}} Q\varphi \big( (m_0^{\mu})^n m \big) \theta_{\wt b}(dm) 
\xrightarrow[n\to\infty]{} 0. 
\end{align*}
Since the operator $Q: C^{0}(\wt B) \to C^0(M/M_0^{\mu})$ is bounded and the space $\bigoplus_{x \in \wh M} \pi_x (C^{\gamma}(\wt B))$ is dense in $C^{0}(\wt B)$,
the above convergence holds for any $\varphi \in C^{0}(\wt B)$. 
Thanks to the interpretation of the operator $Q$ as a family of probability measures (see Lemma \ref{Lem-Riesz-repre}), this proves equation \eqref{convergence-P-mu-meas}. 
As noticed in Remark \ref{Remark-unique-invariant-measure}, this, together with \eqref{nu-M-M0-meas}, implies the uniqueness of the invariant measure $\wt \nu$. 
\end{proof}


\section{Local limit theorems for cocycles}\label{Section-LLT}

We shall state a general local limit theorem for certain cocycles over group actions.
As before, we assume that $\wt B$ is a compact metric space, $M$ is a compact metrizable group acting on the right on $\wt B$, 
and $\eta: \wt B \to B$ is a Lipschitz $M$-principal bundle. 
Let $G$ be a locally compact group acting on the left on $\wt B$ by Lipschitz continuous transformations,
and suppose that this action commutes with the action of $M$. 
We also suppose that $\mu$ is a Borel probability measure on $G$ 
such that the action of $G$ on the base space $B$ satisfies the contraction property \eqref{contraction-property-on-B}.

\subsection{Cocycles and eigenvalues of modulus one}\label{Subsec-Cocycles and eigenvalues of modulus one}
We let $\sigma_0: G \times B \to \bb R$ be a continuous cocycle such that, for any $g \in G$, 
the function $\sigma_0(g, \cdot)$ is Lipschitz continuous on $B$.
For each $g \in G$, we set 
\begin{align*}
\sigma_{\sup}(g) = \sup_{b \in B} |\sigma_0(g, b)|
\quad  \mbox{and}   \quad
\sigma_{\rm Lip}(g) 
 = \sup_{\substack{b, b' \in B \\ b \neq b'}} \frac{|\sigma_0(g, b) - \sigma_0(g, b')|}{\bf d (b, b')}. 
\end{align*}
We always assume that there exists a constant $\alpha>0$ such that 
\begin{align}\label{moment-assumption-02}
\int_{G} e^{\alpha \sigma_{\sup}(g)} \mu(dg) < \infty
\quad  \mbox{and}   \quad
\int_{G} \sigma_{\rm Lip}(g)^{\alpha} \mu(dg) < \infty. 
\end{align}

For $\gamma \in [0,1]$ and $x \in \wh M$, we still write $W_x^{\gamma}$ 
for the space of $\gamma$-H\"older continuous maps $\varphi: \wt B \to \bb C^{d_x}$ that are $\rho_x$-equivariant.
For a complex number $z \in \bb C$ whose real part is sufficiently close to $0$, 
we define a bounded linear operator of $W_x^{0}$ by setting, for $\varphi \in W_x^{0}$ and $\wt b \in \wt B$, 
\begin{align}\label{def-P-mu-x-z}
P_{\mu, x, z} \varphi(\wt b) = \int_{G} \varphi(g \wt b) e^{z \sigma_0(g, \eta(\wt b))} \mu(dg). 
\end{align}
This operator is a perturbation of the operator $P_{\mu, x}$ studied in Section \ref{Sec-Existence of limit measures}. 
We can extend Proposition \ref{Propo-eigenvalue-one} to this family of operators with two parameters. 

\begin{proposition}\label{Propo-eigenvalue-one-002}
Let $\mu$ be a Borel probability measure on $G$. 
Assume that the action of $G$ on $\wt B$ satisfies the moment condition \eqref{moment-assumption-01},
the action of $G$ on $B$ satisfies the contraction property \eqref{contraction-property-on-B}, 
and the continuous cocycle $\sigma_0: G \times B \to \bb R$ satisfies the integrability condition \eqref{moment-assumption-02}. 
Suppose moreover that, for every $\wt b \in \wt B$ with $\eta(\wt b) \in \supp \nu$,   we have $\wt b M \subset \overline{\Gamma_{\mu} \wt b}$. 
Then, there exists a unique closed subgroup $M_1^{\mu}$ of $M \times \bb R$ with $[M, M] \times \{0\} \subset M_1^{\mu}$ such that,  
for any $x \in \wh M$ and $t \in \bb R$, the operator $P_{\mu, x, it}$ admits an eigenvalue with modulus one in $W^0_x$ if and only if  
$\rho_x$ is the one-dimensional representation associated with a character $\chi_x$ of $M$
such that, for any $(m, s) \in M_1^{\mu}$, one has $\chi_x(m) e^{i st} = 1$. 
\end{proposition}

\begin{remark}\label{Rem-group-M0-M1}
The group $M_0^{\mu}$ appearing in Proposition \ref{Propo-eigenvalue-one} is the closure of the projection
of the group $M_1^{\mu}$ onto $M$ in Proposition \ref{Propo-eigenvalue-one-002}. 
\end{remark}

\begin{proof}[Proof of Proposition \ref{Propo-eigenvalue-one-002} and Remark \ref{Rem-group-M0-M1}]
The argument follows the same lines as the proof of Proposition \ref{Propo-eigenvalue-one}.
By using the same technique, one first shows that, for any $x \in \wh M$ and $t \in \bb R$, 
 if the operator $P_{\mu, x, it}$ admits an eigenvalue of modulus one,
then necessarily $d_x = 1$. 
Next, for a character $\chi$ of $M$ and $t \in \bb R$, by using Corollary \ref{Corollary-Spectral-gap-002}, 
one shows that $P_{\mu, \chi, it}$ admits an eigenvalue $\lambda$ of modulus one
if and only if there exists a $\chi$-equivariant continuous function $\varphi$ on $\eta^{-1}(\supp \nu)$ with values of modulus one
such that, for any $\wt b \in \eta^{-1}(\supp \nu)$ and $g \in \supp \mu$, 
\begin{align*}
\varphi( g \wt b) = e^{- it \sigma_0(g, \eta(\wt b))} \lambda \varphi(\wt b). 
\end{align*}
This characterization implies that the set of pairs $(\chi, t)$ for which $P_{\mu, \chi, it}$ admits an eigenvalue of modulus one
is a closed subgroup of $\wh{M^{ab}} \times \bb R$, where $M^{ab} = M/[M, M]$ is the abelianization of $M$. 
We define $M_1^{\mu} \subset M \times \bb R$ as being the annihilator of this subgroup, that is, 
\begin{align*}
M_1^{\mu} 
& = \bigg\{ (m, s) \in M \times \bb R:   \forall (\chi, t) \in \wh{M^{ab}} \times \bb R,  \notag\\
& \qquad\quad \mbox{if $P_{\mu, \chi, it}$ admits an eigenvalue of modulus one, then $\chi(m) e^{ist} = 1$}  \bigg\}. 
\end{align*}
Note that, in particular, in $\wh{M^{ab}} \times \bb R$, 
we have 
\begin{align*}
(M_0^{\mu} \times \bb R)^{\perp} 
& = (M_0^{\mu})^{\perp} \times \{0\}  \notag\\
& = (M_1^{\mu})^{\perp} \cap (\wh{M^{ab}} \times \{0\}) = (M_1^{\mu})^{\perp} \cap (\{0\} \times \bb R)^{\perp}. 
\end{align*}
By Pontryagin's duality, we obtain that $M_0^{\mu} \times \bb R$ is the closure of $M_1^{\mu}(\{0\} \times \bb R)$ in $M \times \bb R$, 
which is to say that Remark \ref{Rem-group-M0-M1} holds. 
\end{proof}

We can also state an extension of Corollary \ref{Cor-eigenvalue-modulus-one}, 
whose proof follows the same lines. 

\begin{corollary}\label{Cor-eigenvalue-modulus-one-002}
Let $\mu$ be a Borel probability measure on $G$. 
Assume that the action of $G$ on $\wt B$ satisfies the moment condition \eqref{moment-assumption-01},
the action of $G$ on $B$ satisfies the contraction property \eqref{contraction-property-on-B}, 
and the continuous cocycle $\sigma_0: G \times B \to \bb R$ satisfies the integrability condition \eqref{moment-assumption-02}. 
Suppose that, for every $\wt b \in \wt B$ with $\eta(\wt b) \in \supp \nu$,  we have $\wt b M \subset \overline{\Gamma_{\mu} \wt b}$. 
Fix $\wt b_0 \in \wt B$ with $\eta(\wt b_0) \in \supp \nu$. 
Then, for every character $(\chi, t)$ of the group $(M \times \bb R)/ M_1^{\mu}$, the operator $P_{\mu, \chi, it}$ 
has a unique eigenvalue $\lambda_{\chi, t}$ of modulus one in the space $W_{\chi}^{0}$. 

Moreover, there exists a unique associated eigenvector $\varphi_{\chi, t} \in W_{\chi}^{0}$
such that $\varphi_{\chi, t}(\wt b_0) = 1$ and $\varphi_{\chi, t}$ is H\"older continuous on $\wt B$.  
Also, the generalized eigenspace associated with the eigenvalue $\lambda_{\chi, t}$ for the operator $P_{\mu, \chi, it}$ acting on the space $W_{\chi}^0$ 
is equal to $\bb C \varphi_{\chi, t}$, that is, this eigenvalue is multiplicity free. 

Finally, the restriction $\varphi_{\chi, t}|_{\eta^{-1}(\supp \nu)}$ is the unique continuous function on $\eta^{-1}(\supp \nu)$
such that $\varphi_{\chi, t}(\wt b_0) = 1$ and for every $m \in M$,  $g \in \supp \mu$ and $\wt b \in \eta^{-1}(\supp \nu)$, we have 
\begin{align*}
\varphi_{\chi, t}(\wt b m) = \chi(m)^{-1} \varphi_{\chi, t}(\wt b)
\quad \mbox{and} \quad
\varphi_{\chi, t}(g \wt b) = e^{-it\sigma_0(g, \eta(\wt b))} \lambda_{\chi, t} \varphi_{\chi, t}(\wt b). 
\end{align*} 
\end{corollary}

\subsection{Statement of the local limit theorem}
We keep the assumptions used in Subsection \ref{Subsec-Cocycles and eigenvalues of modulus one}. 
Under the integrability condition \eqref{moment-assumption-02}, the real numbers 
\begin{align}\label{def-mean-variance}
\lambda_{\mu} = \lim_{n \to \infty} \frac{1}{n} \bb E \sigma_0(g_n \cdots g_1, b) 
\quad  \mbox{and}  \quad
\upsilon_{\mu}^2 = \lim_{n \to \infty} \frac{1}{n} \bb E \left[ \big( \sigma_0(g_n \cdots g_1, b) - n \lambda_{\mu} \big)^2 \right]
\end{align}
are well defined and do not depend on the choice of $b \in B$,
see \cite[Chapters 11.5 and 11.6]{BQ16b}. 
In particular, recall  from \cite[Lemma 11.18]{BQ16b} that $\upsilon_{\mu} = 0$ 
if and only if there exists a continuous function $\varphi: \supp\nu \to \bb R$ such that for every $g \in \supp \mu$ and $b \in \supp \nu$, 
we have $\sigma_0(g, b) = \lambda_{\mu} + \varphi(gb) - \varphi(b)$, 
where $\nu$ is the unique $\mu$-stationary Borel probability measure on $B$. 
Thus, in particular, if $\upsilon_{\mu} = 0$, then one has $M_1^{\mu} \subset M \times \{0\}$. 
Therefore, the assumption that $M_1^{\mu} = M_0^{\mu} \times \bb R$ in Theorem \ref{Local-limit-theorem-001} below 
implies $\upsilon_{\mu} > 0$.

Finally, we let $\sigma: G \times \wt B \to \bb R$ be a continuous cocycle which is cohomologous to $\sigma_0$, meaning that 
there exists a continuous function $\psi: \wt B \to \bb R$ such that, for any $g \in G$ and $\wt b \in \wt B$, 
\begin{align}\label{cohomologous-sigma}
\sigma(g, \wt b) = \sigma_0(g, \eta(\wt b)) + \psi(g \wt b) - \psi(\wt b). 
\end{align}
In this section we shall prove that, under the assumptions \eqref{moment-assumption-01} and \eqref{moment-assumption-02}, 
for $\wt b \in \wt B$, 
the sequence of random variables $(\sigma(g_n \cdots g_1, \wt b))_{n \geq 1}$
satisfies the following local limit theorem.

\begin{theorem}\label{Local-limit-theorem-001}
Let $\mu$ be a Borel probability measure on $G$. 
Assume that the action of $G$ on $\wt B$ satisfies the moment condition \eqref{moment-assumption-01},
the action of $G$ on $B$ satisfies the contraction property \eqref{contraction-property-on-B}, 
and the continuous cocycle $\sigma_0: G \times B \to \bb R$ 
satisfies the integrability condition \eqref{moment-assumption-02} and $M_1^{\mu} = M_0^{\mu} \times \bb R$. 
Suppose that, for every $\wt b \in \wt B$ with $\eta(\wt b) \in \supp \nu$, 
we have $\wt b M \subset \overline{\Gamma_{\mu} \wt b}$. 
Let $\sigma: G \times \wt B \to \bb R$ be a continuous cocycle cohomologous to $\sigma_0$. 
Then, for every continuous compactly supported function $h: \bb R \to \bb C$, we have, 
uniformly in $\wt b \in \wt B$ and in $w_n \in \bb R$ with $|w_n| = o(\sqrt{n})$, 
\begin{align*}
 \upsilon_{\mu} \sqrt{2 \pi n} \,  \bb E  \left[ h \left(\sigma(g_n \cdots g_1, \wt b) - n \lambda_{\mu}  - w_n  \right) \right]
\xrightarrow[n\to\infty]{}  \int_{\bb R}   h(u) du. 
\end{align*}
\end{theorem}

In this statement, we have used the notation of Corollary \ref{Cor-equdistribution}. 
By uniformly in $w_n \in \bb R$ with $|w_n| = o(\sqrt{n})$, we mean that for any sequence of positive numbers $(\beta_n)_{n \geq 1}$ satisfying $\lim_{n \to \infty} \beta_n = 0$, the result holds uniformly in $|w_n| \leq \beta_n \sqrt{n}$. 
The perturbation sequence $(w_n)_{n \geq 1}$ will be useful to deduce Theorem \ref{Main-Thm-LLT} from Theorem \ref{Local-limit-theorem-001},
as will be done in Section \ref{Section-Proof-Main-theorem}. 
At the first reading, the reader can assume that $w_n = 0$. 

\begin{remark}\label{Remark-variance}
Note that, in view of \eqref{def-mean-variance} and \eqref{cohomologous-sigma}, for any $\wt b \in \wt B$, we have 
\begin{align*}
\upsilon_{\mu}^2 = \lim_{n \to \infty} \frac{1}{n} \bb E \left[ \big( \sigma(g_n \cdots g_1, \wt b) - n \lambda_{\mu} \big)^2 \right]. 
\end{align*}
\end{remark}

\subsection{Proof for smooth observables}
We begin by proving a version of Theorem \ref{Local-limit-theorem-001} for regular target functions. 
For $\gamma \in (0, 1)$, we denote by $W_x^{\gamma}$ the space of all $\gamma$-H\"older continuous 
$M$-equivariant maps from $\wt B$ to $V_x: = \bb C^{d_x}$, that is,
all $\gamma$-H\"older continuous maps $\psi: \wt B \to V_x$ such that, for any $m \in M$ and $\wt b \in \wt B$, 
\begin{align*}
\psi(\wt b m) = \rho_x(m)^{-1} \psi(\wt b). 
\end{align*}

\begin{proposition}\label{Prop-Local-limit-theorem-001}
Let $\mu$ be a Borel probability measure on $G$. 
Assume that the action of $G$ on $\wt B$ satisfies the moment condition \eqref{moment-assumption-01},
the action of $G$ on $B$ satisfies the contraction property \eqref{contraction-property-on-B}, 
and the continuous cocycle $\sigma_0: G \times B \to \bb R$ satisfies the integrability condition \eqref{moment-assumption-02} 
and $M_1^{\mu} = M_0^{\mu} \times \bb R$. 
Suppose that, for every $\wt b \in \wt B$ with $\eta(\wt b) \in \supp \nu$,   
we have $\wt b M \subset \overline{\Gamma_{\mu} \wt b}$. 
Then, there exists $\gamma \in (0,1]$ with the following property.  
Let $h$ be a Schwartz function on $\bb R$ whose Fourier transform has compact support. 

(1) If $x \in \wh M$ is not the one-dimensional representation associated with a character vanishing on $M_0^{\mu}$, then,
for any $f \in V_x^*$ and $\psi \in W_x^{\gamma}$, 
we have, uniformly in $\wt b \in \wt B$ and in $w_n \in \bb R$ with $|w_n| = o(\sqrt{n})$,  
\begin{align*}
\upsilon_{\mu} \sqrt{2 \pi n} \,  \bb E  \left[ \langle f, \psi \rangle(g_n \cdots g_1 \wt b)  \,  h \left( \sigma_0(g_n \cdots g_1, \eta(\wt b)) - n \lambda_{\mu} - w_n \right)  \right]  
\xrightarrow[n\to\infty]{}  0. 
\end{align*}

(2) For any character $\chi$ of $M/M_0^{\mu}$, for any $\psi \in W_{\chi}^{\gamma}$, 
we have, uniformly in $\wt b \in \wt B$ and in $w_n \in \bb R$ with $|w_n| = o(\sqrt{n})$,  
\begin{align*}
&  \upsilon_{\mu} \sqrt{2 \pi n} \,  \bb E  \left[ \psi(g_n \cdots g_1 \wt b)  \,  h \left( \sigma_0(g_n \cdots g_1, \eta(\wt b)) - n \lambda_{\mu} - w_n \right)  \right]  \notag\\
& \qquad\qquad\qquad\qquad\quad
- \zeta_{\chi}(\psi)  \varphi_{\chi}(\wt b)  \chi(m_0^{\mu})^n  \int_{\bb R} h(u) du 
 \xrightarrow[n\to\infty]{}  0. 
\end{align*}
\end{proposition}

In the above statement, we have used the notation from the proof of Corollary \ref{Cor-equdistribution}. 

The proof relies on adapting to spaces of sections of bundles, the arguments which are usually used for proving local limit theorems
for Markov chains under spectral gap assumptions.

\begin{proof}[Proof of Proposition \ref{Prop-Local-limit-theorem-001}]
Thanks to Corollary \ref{Corollary-isometric Lipschitz bundle action}, we know that the assumptions of Corollary \ref{Corollary-Spectral-gap-002}
are satisfied. Thus, we may find $\gamma \in (0, 1]$ such that the conclusions of Corollary \ref{Corollary-Spectral-gap-002} hold
for each vector bundle $B \times_{\rho_x} \bb C^{d_x} \to B$ with $x \in \wh M$ and for each $t \in \bb R$. 
Thanks to Lemma \ref{Lem-closed injection-001} and Corollary \ref{Cor-closed injection-001}, 
we can apply these conclusions to the operator $P_{\mu, x, it}$ defined by \eqref{def-P-mu-x-z} acting on the space $W_{x}^{\gamma}$.  

Let $x \in \wh M$ and $\psi \in W_{x}^{\gamma}$. 
For any integrable function $h$ on $\bb R$, we write its Fourier transform as 
\begin{align*}
\wh h(t) = \int_{\bb R} e^{- i ut} h(u) du, \quad t \in \bb R. 
\end{align*} 
By the Fourier inversion formula, for any $\wt b \in \wt B$, $n \geq 1$ and $w_n \in \bb R$, we have 
\begin{align*}
&  \bb E  \left[ h \Big( \sigma_0(g_n \cdots g_1, \eta(\wt b)) - n \lambda_{\mu} - w_n \Big) \, \psi(g_n \cdots g_1 \wt b)    \right]  \notag\\
& = \frac{1}{2 \pi} \int_{\bb R}  e^{- it (n \lambda_{\mu} + w_n) } \wh h(t)  P_{\mu, x, it}^n \psi(\wt b) dt. 
\end{align*}

If $x$ is not the one-dimensional representation associated with a character that vanishes on $M_0^{\mu}$, 
then, by Proposition \ref{Propo-eigenvalue-one-002} and the assumption that $M_1^{\mu} = M_0^{\mu} \times \bb R$,
for any $t \in \bb R$,  the operator $P_{\mu, x, it}$ has spectral radius strictly smaller than  one in $W_x^{\gamma}$. 
Consequently, $P_{\mu, x, it}^n \psi$ converges to $0$ exponentially fast in $W_x^{\gamma}$, 
uniformly in $t$ belonging to the (compact) support of $\wh h$. 
This establishes assertion (1).

We now analyze the case where $x$ is the one-dimensional representation associated with the character $\chi$ of $M/M_0^{\mu}$.
Fix a small constant $\ee > 0$ (to be determined later). 
By the assumption that $M_1^{\mu} = M_0^{\mu} \times \bb R$, for any $t \in \bb R \setminus\{0\}$, 
the operator $P_{\mu, \chi, it}$ has spectral radius strictly smaller than  one in $W_{\chi}^{\gamma}$. 
Therefore, $P_{\mu, \chi, it}^n \psi$ converges to $0$ exponentially fast  in $W_{\chi}^{\gamma}$, 
uniformly in $t$ belonging to the support of $\wh h$ with $|t| \geq \ee$. 
Consequently, uniformly in $\wt b \in \wt B$ and $w_n \in \bb R$, 
\begin{align*}
&  \upsilon_{\mu} \sqrt{2 \pi n}   \bigg(  \bb E  \left[ \psi(g_n \cdots g_1 \wt b)  \,  h \left( \sigma_0(g_n \cdots g_1, \eta(\wt b)) - n \lambda_{\mu} - w_n \right)  \right]  \notag\\
& \qquad\qquad\qquad
-  \frac{1}{2 \pi} \int_{-\ee}^{\ee}  e^{- it (n \lambda_{\mu} + w_n)} \wh h(t)  P_{\mu, \chi, it}^n \psi(\wt b) dt  \bigg) 
\xrightarrow[n\to\infty]{}  0. 
\end{align*}
By the perturbation theory for linear operators (see for example \cite{BQ16b}), 
after possibly shrinking $\ee$, 
we may write the following decomposition for $z \in \bb C$ with $|z| \leq \ee$, 
\begin{align*}
P_{\mu, \chi, z} = \kappa_{\chi, z} \zeta_{\chi, z} \otimes \varphi_{\chi, z} + R_{\chi, z},
\end{align*}
where $\kappa_{\chi, z}$ is a complex number, $\zeta_{\chi, z}$ is a bounded linear functional on $W_{\chi}^{\gamma}$,
$\varphi_{\chi, z}$ is in $W_{\chi}^{\gamma}$, 
and $R_{\chi, z}$ is a bounded linear operator on $W_{\chi}^{\gamma}$. 
All these objects depend analytically on $z$ and satisfy 
$\zeta_{\chi, z}(\varphi_{\chi, z}) = \zeta_{\chi, z}(\varphi_{\chi}) = 1$, 
$R_{\chi, z} (\varphi_{\chi, z}) = 0$ and $R_{\chi, z}^* (\zeta_{\chi, z} ) =  0$,
and at $z = 0$, we have $\kappa_{\chi, 0} = \chi(m_0^{\mu}) = \lambda_{\chi}$, 
$\varphi_{\chi, 0} = \varphi_{\chi}$ and $\zeta_{\chi, 0} = \zeta_{\chi}$. 
Moreover, for every $z \in \bb C$ with $|z| \leq \ee$, the spectral radius of $R_{\chi, z}$ is smaller than $1 - \ee$. 
Thus, using this decomposition and the exponential decay of $R_{\chi, it}^n$, 
we obtain that uniformly in $\wt b \in \wt B$ and $w_n \in \bb R$, 
\begin{align}\label{LLT decom spectral001}
& \upsilon_{\mu} \sqrt{2 \pi n}   
\bigg(  \bb E  \left[ \psi(g_n \cdots g_1 \wt b)  \,  h \left( \sigma_0(g_n \cdots g_1, \eta(\wt b)) - n \lambda_{\mu} - w_n \right)  \right]  \notag\\
& \qquad\qquad
-  \frac{1}{2 \pi} \int_{-\ee}^{\ee}  e^{- it (n \lambda_{\mu} + w_n)} \wh h(t)  \kappa_{\chi, it}^n \zeta_{\chi, it}( \psi)  \varphi_{\chi, it}(\wt b) dt  \bigg) 
\xrightarrow[n\to\infty]{}  0. 
\end{align}

As in the proof of Proposition \ref{Propo-eigenvalue-one}, 
let $\overline{W}_{\chi}^{\gamma}$ be the space of 
all $\gamma$-H\"older continuous and $\chi$-equivariant maps $\varphi: \eta^{-1}(\supp \nu) \to \bb C^{d_x}$. 
From Corollary \ref{Cor-eigenvalue-modulus-one}, 
we know that, for every $g \in \supp \mu$ and $\wt b \in \eta^{-1}(\supp \nu)$, 
it holds that $\varphi_{\chi}(g \wt b) = \lambda_{\chi} \varphi_{\chi}(\wt b) = \chi(m_0^{\mu}) \varphi_{\chi}(\wt b)$. 
This tells us that, for any $\varphi \in \overline{W}_{\chi}^{\gamma}$ and $z \in \bb C$ with real part $\Re z$ sufficiently close to $0$, we have 
\begin{align}\label{relation-conjugacy-operator}
P_{\mu, \chi, z} \varphi = \chi(m_0^{\mu}) \varphi_{\chi} P_{\mu, 0, z}(\varphi_{\chi}^{-1} \varphi). 
\end{align}
It follows from Corollary \ref{Corollary-Spectral-gap-002},  Corollary \ref{Corollary-isometric Lipschitz bundle action},
Corollary \ref{Cor-closed injection-001} and Corollary \ref{Cor-eigenvalue-modulus-one-002}
that, for any $t \in \bb R$, the operators $P_{\mu, \chi, it}$ and $P_{\mu, 0, it}$ have spectral radius strictly smaller than one
 in the spaces 
$\{\varphi \in W_{\chi}^{\gamma}: \varphi|_{\eta^{-1}(\supp \nu)} = 0\}$
and $\{\varphi \in W_{0}^{\gamma}: \varphi|_{\eta^{-1}(\supp \nu)} = 0\}$, respectively. 
Thus, the conjugacy relation \eqref{relation-conjugacy-operator} between $\overline{W_{\chi}^{\gamma}}$ and $\overline{W_{0}^{\gamma}}$
forces, for $t$ small enough, 
\begin{align*}
\kappa_{\chi, it} = \chi(m_0^{\mu}) \kappa_{0, it}.
\end{align*}
Inserting this into \eqref{LLT decom spectral001}, 
we obtain that, uniformly in $\wt b \in \wt B$ and $w_n \in \bb R$, 
\begin{align*}
& \upsilon_{\mu} \sqrt{2 \pi n}   
 \bigg(  \bb E  \left[ \psi(g_n \cdots g_1 \wt b)  \,  h \left( \sigma_0(g_n \cdots g_1, \eta(\wt b)) - n \lambda_{\mu} - w_n \right)  \right]  \notag\\
& \qquad
-  \chi(m_0^{\mu})^n 
\frac{1}{2 \pi} \int_{-\ee}^{\ee}  e^{- it (n \lambda_{\mu} + w_n)} \wh h(t)   \kappa_{0, it}^n \zeta_{\chi, it}( \psi)  \varphi_{\chi, it}(\wt b) dt  \bigg) 
\xrightarrow[n\to\infty]{} 0. 
\end{align*}
Now the classical proof of the local limit theorem (cf.\ \cite[Lemma 16.11]{BQ16b}) shows that, 
uniformly in $\wt b \in \wt B$ and in $w_n \in \bb R$ with $|w_n| = o(\sqrt{n})$, 
\begin{align*}
\upsilon_{\mu} \sqrt{\frac{n}{2 \pi}}  \int_{-\ee}^{\ee}  e^{- it (n \lambda_{\mu} + w_n) } \wh h(t)   \kappa_{0, it}^n \zeta_{\chi, it}( \psi)  \varphi_{\chi, it}(\wt b) dt
\xrightarrow[n\to\infty]{}  \wh h(0) \zeta_{\chi, 0}(\psi)  \varphi_{\chi, 0}(\wt b). 
\end{align*}
Multiplying by $\chi(m_0^{\mu})^n$ 
and noting that $\wh h(0) = \int_{\bb R} h(u) du$, $\zeta_{\chi, 0} = \zeta_{\chi}$ and $\varphi_{\chi, 0} = \varphi_{\chi}$, 
the conclusion of assertion (2) follows. 
\end{proof}

In order to extend Proposition \ref{Prop-Local-limit-theorem-001} 
to the case of an arbitrary continuous function $\psi$ on $\wt B$, 
we need a coarse domination estimate for certain auxiliary functions $h$ on $\bb R$.

\begin{corollary}\label{Corollary-Local-limit-theorem-001}
Let $\mu$ be a Borel probability measure on $G$. 
Assume that the action of $G$ on $\wt B$ satisfies the moment condition \eqref{moment-assumption-01},
 the action of $G$ on $B$ satisfies the contraction property \eqref{contraction-property-on-B}, 
and the continuous cocycle $\sigma_0: G \times B \to \bb R$ 
satisfies the integrability condition \eqref{moment-assumption-02} and $M_1^{\mu} = M_0^{\mu} \times \bb R$. 
Suppose that, for every $\wt b \in \wt B$ with $\eta(\wt b) \in \supp \nu$,   we have $\wt b M \subset \overline{\Gamma_{\mu} \wt b}$. 
Let $h$ be a non-negative measurable function on $\bb R$ such that 
$\sup_{u \in \bb R} (1 + |u|^k) h(u) < \infty$ for any $k \geq 1$. 
Then, we have, uniformly in $\wt b \in \wt B$ and in $w_n \in \bb R$ with $|w_n| = o(\sqrt{n})$,  
\begin{align*}
\sup_{n \geq 1}  \sqrt{n}  \,  \bb E  \left[  h \left( \sigma_0(g_n \cdots g_1, \eta(\wt b)) - n \lambda_{\mu} - w_n \right)  \right]   
< \infty. 
\end{align*}
\end{corollary}

\begin{proof}
For $u \in \bb R$, set $h^+(u) = \sup_{|u - u'| \leq 1} h(u')$. 
The rapid decay of $h$ implies that $h^+$ is also rapidly decreasing.
We choose a non-negative Schwartz function $\alpha$ on $\bb R$ such that $\int_{-1}^1 \alpha(u) du = 1$
and the Fourier transform of $\alpha$ has compact support. 
Define $h_1 = \alpha * h^+$.
Then $h_1$ is a non-negative Schwartz function 
whose Fourier transform has compact support, and $h \leq h_1$ pointwise. 
Consequently, we get 
\begin{align*}
& \sup_{n \geq 1}  \sqrt{n}  \,  \bb E  \left[  h \left( \sigma_0(g_n \cdots g_1, \eta(\wt b)) - n \lambda_{\mu} - w_n \right)  \right]    \notag\\
& \leq  \sup_{n \geq 1}  \sqrt{n}  \,  \bb E  \left[  h_1 \left( \sigma_0(g_n \cdots g_1, \eta(\wt b)) - n \lambda_{\mu} -w_n \right)  \right].
\end{align*}
The latter is finite by Proposition \ref{Prop-Local-limit-theorem-001} with $\psi = 1$. 
\end{proof}

Using the domination estimate from Corollary \ref{Corollary-Local-limit-theorem-001}, 
we can now extend Proposition \ref{Prop-Local-limit-theorem-001} to 
the case of any continuous function $\psi$ on $\wt B$.

\begin{corollary}\label{Corollary-Local-limit-theorem-002}
Let $\mu$ be a Borel probability measure on $G$. 
Assume that the action of $G$ on $\wt B$ satisfies the moment condition \eqref{moment-assumption-01},
the action of $G$ on $B$ satisfies the contraction property \eqref{contraction-property-on-B}, 
and the continuous cocycle $\sigma_0: G \times B \to \bb R$ 
satisfies the integrability condition \eqref{moment-assumption-02} and $M_1^{\mu} = M_0^{\mu} \times \bb R$. 
Suppose that, for every $\wt b \in \wt B$ with $\eta(\wt b) \in \supp \nu$,   
we have $\wt b M \subset \overline{\Gamma_{\mu} \wt b}$. 
Let $h$ be a Schwartz function on $\bb R$ whose Fourier transform has compact support,
and let $\psi$ be a continuous function on $\wt B$. 
Then, uniformly in $\wt b \in \wt B$ and in $w_n \in \bb R$ with $|w_n| = o(\sqrt{n})$,   
\begin{align*}
 \upsilon_{\mu} \sqrt{2 \pi n} \,  
& \bb E \left[ \psi\left( g_n \cdots g_1 \wt b \right) h \left( \sigma_0(g_n \cdots g_1, \eta(\wt b)) - n \lambda_{\mu} - w_n \right) \right]  \notag\\
& \qquad\quad -  \int_{M/M_0^{\mu}}  Q \psi((m_0^{\mu})^n m)  \, \theta_{\wt b}(d m)  \int_{\bb R}  h(u) du 
\xrightarrow[n\to\infty]{} 0. 
\end{align*}
\end{corollary}

\begin{proof}
From the definition of the operator $Q$ in the proof of Corollary \ref{Cor-equdistribution} (cf.\ \eqref{def operator Q}), 
we see that Proposition \ref{Prop-Local-limit-theorem-001} says that the conclusion of Corollary \ref{Corollary-Local-limit-theorem-002}
holds when $\psi$ belongs to the direct sum $\bigoplus_{x \in \wh M} \pi_x (C^{\gamma}(\wt B))$. 
For a general continuous function $\psi \in C^0(\wt B)$ and any $\ee > 0$, 
we can find $\psi' \in \bigoplus_{x \in \wh M} \pi_x (C^{\gamma}(\wt B))$
such that $\| \psi - \psi' \| < \ee$. 
Then, by applying Proposition \ref{Prop-Local-limit-theorem-001} to the function $\psi'$, we get 
\begin{align*}
& \limsup_{n \to \infty}
\bigg| \upsilon_{\mu} \sqrt{2 \pi n} \,  
 \bb E \left[ \psi\left( g_n \cdots g_1 \wt b \right) h \left( \sigma_0(g_n \cdots g_1, \eta(\wt b)) - n \lambda_{\mu} - w_n \right) \right]  \notag\\
& \qquad\qquad\qquad\qquad\qquad\qquad -  \int_{M/M_0^{\mu}}  Q \psi((m_0^{\mu})^n m) \,  \theta_{\wt b}(d m) 
   \int_{\bb R}  h(u) du \bigg|   \notag\\
& \leq \ee \left(  \sup_{n \geq 1}  \upsilon_{\mu} \sqrt{2 \pi n} \,   
\bb E \left[ |h| \left( \sigma_0(g_n \cdots g_1, \eta(\wt b)) - n \lambda_{\mu} - w_n \right) \right] 
+ \int_{\bb R} |h(u)| du  \right). 
\end{align*}
Since $\ee>0$ is arbitrary, the conclusion follows by applying Corollary \ref{Corollary-Local-limit-theorem-001}. 
\end{proof}

\begin{proof}[Proof of Theorem \ref{Local-limit-theorem-001}]
We now deduce the local limit theorem for the cocycle $\sigma: G \times \wt B \to \bb R$ cohomologous to $\sigma_0$. 
By applying the same method as in the proof of \cite[Proposition 16.6]{BQ16b}, 
one shows from Corollary \ref{Corollary-Local-limit-theorem-002} that,
for every continuous compactly supported function $\varphi: \wt B \times \bb R \to \bb C$, we have, uniformly in $\wt b \in \wt B$
and in $w_n \in \bb R$ with $|w_n| = o(\sqrt{n})$,   
\begin{align}\label{LLT joint target function001}
&  \upsilon_{\mu} \sqrt{2 \pi n} \,  
\bb E  \left[ \varphi \left(g_n \cdots g_1 \wt b, \sigma_0(g_n \cdots g_1, \eta(\wt b)) - n \lambda_{\mu}  - w_n  \right)  \right]
  \notag\\
& \qquad\qquad\quad
- \int_{\bb R} \int_{M/M_0^{\mu}}  Q \varphi((m_0^{\mu})^n m, u)  \, \theta_{\wt b}(d m) \, du 
\xrightarrow[n\to\infty]{}  0. 
\end{align}
Now let $h$ be a continuous compactly supported function on $\bb R$
and define, for $\wt b \in \wt B$ and $u \in \bb R$, $\varphi(\wt b, u) = h (u + \psi(\wt b))$,
where $\psi: \wt B \to \bb R$ is the continuous function appearing in \eqref{cohomologous-sigma}. 
Applying \eqref{LLT joint target function001} to this target function $\varphi$ and using the equation \eqref{cohomologous-sigma}, 
we get that, uniformly in $\wt b \in \wt B$ and for $w_n \in \bb R$ with $|w_n| = o(\sqrt{n})$,   
\begin{align*}
& \upsilon_{\mu} \sqrt{2 \pi n} \,  \bb E  \left[ h \left(\sigma(g_n \cdots g_1, \wt b) - n \lambda_{\mu} - w_n  \right) \right]  \notag\\
& \qquad
 - \int_{\bb R} \int_{M/M_0^{\mu}}  Q \varphi((m_0^{\mu})^n m, u) \, \theta_{\wt b}(d m) \, du 
 \xrightarrow[n\to\infty]{}  0,  
\end{align*}
where we used the fact that $\psi(\wt b)$ is bounded and hence it can be absorbed by $w_n$. 
Notice that, for any $n \geq 1$, by Remark \ref{Rem-operator-Q-q} and Fubini's theorem, we have 
\begin{align*}
& \int_{\bb R} \int_{M/M_0^{\mu}}  Q \varphi \big( (m_0^{\mu})^n m, u \big) \,  \theta_{\wt b}(d m) \, du  \notag\\
& = \int_{\bb R} \int_{M/M_0^{\mu}}  \int_{\wt B}  \varphi(\wt b', u)  \,  q_{(m_0^{\mu})^n m}(d \wt b')  \, \theta_{\wt b}(d m) \,  du  \notag\\
& = \int_{\bb R} \int_{M/M_0^{\mu}}  \int_{\wt B}  h (u + \psi(\wt b'))  \,  q_{(m_0^{\mu})^n m}(d \wt b')  \,   \theta_{\wt b}(d m) \,  du    \notag\\
& =  \int_{M/M_0^{\mu}}  \int_{\wt B}  \int_{\bb R}  h (u + \psi(\wt b')) \,  du  \,  q_{(m_0^{\mu})^n m}(d \wt b')  \,  \theta_{\wt b}(d m)  \notag\\
& = \int_{\bb R}  h (u) du. 
\end{align*}
This concludes the proof of the theorem. 
\end{proof}

\section{Contraction property and strongly irreducible representations}\label{Sec-action-Strong-irre}

\subsection{Setting}
Now we come back to the assumptions used in the introduction. 
Let $\mu$ be a Borel probability measure on the linear group ${\rm GL}(V)$,
where $V$ is a finite-dimensional real vector space. 
Denote by $\Gamma_{\mu}$ the closed subsemigroup of ${\rm GL}(V)$ spanned by the support of the measure $\mu$, 
and assume that $\Gamma_{\mu}$ acts strongly irreducibly on $V$. 
We also assume that the moment condition \eqref{Exponential-moment} is satisfied. 
As an intermediate step towards proving Theorem \ref{Main-Thm-LLT}, we will apply Theorem \ref{Local-limit-theorem-001} to obtain
that, for a certain class of (non-invertible) endomorphisms $h \in {\rm End}(V)$, 
the sequence of random variables $\log \| g_n \cdots g_1 h \|$ satisfies a local limit theorem. 
In order to state this result, we need to introduce precisely the class of endomorphisms $h$ that we will be able to manage. 

Let $V_{\bb C} = \bb C \otimes V$ be the complexification of $V$. 
Let $\bf G \subset {\rm GL}(V_{\bb C})$ be the Zariski closure of $\Gamma_{\mu}$, 
and let $\bf G^{\circ} \subset \bf G$ be the Zariski connected component of the identity in $\bf G$. 
Thus $\bf G^{\circ}$ is a normal subgroup of finite index in $\bf G$. 
Since we have assumed that $\Gamma_{\mu}$ acts strongly irreducibly on $V$, 
the action of $G^{\circ}$ on $V$ remains strongly irreducible, 
where $G^{\circ} = \bf G^{\circ}(\bb R) = \bf G^{\circ} \cap {\rm GL}(V)$ is the group of real points of $\bf G^{\circ}$. 
In particular, the connected real algebraic group $\bf G^{\circ}$ is reductive.

We choose a maximal split torus $\bf A \subset \bf G^{\circ}$ and a minimal parabolic subgroup $\bf P \subset \bf G^{\circ}$ 
such that $\bf A \subset \bf P$, and let $\bf U$ be the unipotent radical of $\bf P$.
We denote by $V^+$ the space of $U$-invariant vectors in $V$, i.e.,
\begin{align*}
V^+ = \{ v \in V: \forall u \in U, uv = v \}, 
\end{align*}
where $U = \bf U(\bb R) = \bf U \cap {\rm GL}(V)$ is the group of real points of $\bf U$.
(Note that the space $V^+$ is the highest weight space for the action of $A = \bf A(\bb R)$ on $V$ 
with respect to the order associated with the choice of the group $\bf P$, see \cite{Tit71}.) 
The space $V^+$ admits a unique $A$-invariant complementary subspace, which we denote by $V^-$. 
Let $\varpi \in {\rm End}(V)$ be the projection onto $V^+$ with kernel $V^-$. 

Our intermediate step towards proving Theorem \ref{Main-Thm-LLT} is stated as follows.

\begin{theorem}\label{Thm-LLT-intermediate}
Assume that $\Gamma_{\mu}$ is strongly irreducible 
and that the image of $\Gamma_{\mu}$ in the projective linear group 
${\rm PGL}(V)$ is not contained in any compact subgroup of ${\rm PGL}(V)$. 
Assume that $\mu$ admits a finite exponential moment \eqref{Exponential-moment}. 
Then, there exists $\upsilon_{\mu} > 0$ such that, for any real numbers $a < b$,
 uniformly in $\xi \in G \varpi G$ and in $w_n \in \bb R$ with $|w_n| = o(\sqrt{n})$,   
\begin{align*}
\upsilon_{\mu} \sqrt{2 \pi n}  \,  \bb P  \left( \log \frac{\|g_n \cdots g_1 \xi \|}{\| \xi \|} - n \lambda_{\mu} - w_n \in [a, b] \right) 
\xrightarrow[n\to\infty]{}  b-a. 
\end{align*}
\end{theorem}

In this statement, we write $G$ for the group of real points of $\bf G$, that is, $G = \bf G(\bb R) = \bf G \cap {\rm GL}(V)$. 

\begin{remark}\label{Remark-variance-002}
Theorem \ref{Thm-LLT-intermediate} will be deduced from Theorem \ref{Local-limit-theorem-001}. 
Together with Remark \ref{Remark-variance}, 
this shows in particular that the asymptotic variance $\upsilon_{\mu}^2$ can be defined by the convergence 
\begin{align*}
\upsilon_{\mu}^2 = \lim_{n \to \infty} \frac{1}{n} \bb E \left[ \left( \log \frac{\|g_n \cdots g_1 \xi \|}{\| \xi \|} - n \lambda_{\mu} \right)^2 \right], 
\end{align*}
which is uniform in $\xi \in G \varpi G$. 
We will show later that the convergence \eqref{def-variance-aa} also holds, see Remark \ref{Remark-variance-003} below. 
\end{remark}

\begin{remark}
Since all minimal parabolic subgroups and maximal split tori of $\bf G$ are conjugated 
(see  \cite[Theorem 19.2 and Proposition 20.5]{Borel12}), 
the set $G \varpi G$ does not depend on the choices of $\bf P$ and $\bf A$. 
Indeed, if $\bf P' \supset  \bf A'$ is another such pair, we may find $g \in G^{\circ}$ with $\bf P' = g \bf P g^{-1}$ and $\bf A' = g \bf A g^{-1}$.
Then, if $\varpi'$ is the projection of $V$ associated with the choice of $\bf P'$ and $\bf A'$, we have $\varpi' = g \varpi g^{-1}$. 
Hence $G \varpi' G = G \varpi G$. 
\end{remark}

\begin{example}
Assume that $\Gamma_{\mu}$ is Zariski dense in ${\rm GL}(V)$. Then the projection $\varpi$ has rank $1$ and the set 
$G \varpi G = {\rm GL}(V) \varpi {\rm GL}(V)$ is the set of rank $1$ endomorphisms of $V$. 
More generally, if $\Gamma_{\mu}$ is proximal in $V$, then all the elements of the set $G \varpi G$ have rank $1$. 
\end{example}

\begin{example}
Assume that $V$ has even dimension and that $\Gamma_{\mu}$ preserves a complex structure on $V$, that is,
there exists a real endomorphism $J$ of $V$ such that $J^2 = -1$ and $\Gamma_{\mu}$ commutes with $J$. 
For simplicity, suppose that $\Gamma_{\mu}$ is Zariski dense in the group ${\rm GL}(V, J)$ 
of automorphisms of $V$ which preserve the complex structure, that is, the group automorphisms which commute with $J$. 
Then the set $G \varpi G = {\rm GL}(V, J) \varpi {\rm GL}(V, J)$ is the set of rank $2$ endomorphisms of $V$  
which preserve the complex structure. 
\end{example}

In the previous examples, the group $G$ is Zariski connected. Here is an example where it is not. 

\begin{example}
Assume again that $V$ has even dimension and that $J$ is a complex structure on $V$. 
Suppose that $G$ is the group ${\rm GL}(V, \pm J)$ defined by ${\rm GL}(V, \pm J) = \{ g \in {\rm GL}(V):  gJg^{-1} = \pm J \}$. 
In other words, the group ${\rm GL}(V, \pm J)$ is the group of $\bb R$-linear automorphisms of $V$ which are either linear or semi-linear with respect to 
the complex structure $J$ on $V$. 
Then the set $G \varpi G$ is the set of rank $2$ endomorphisms of $V$  
which are either linear or semi-linear. 
\end{example}

\subsection{Structure results}
The proof of Theorem \ref{Thm-LLT-intermediate} boils down to applying Theorem \ref{Local-limit-theorem-001} 
to the action of $\Gamma_{\mu}$ on the quotient space $\wt B = G / AU$. 
Indeed, let us now relate this compact space to the objects that appear in the statement of Theorem \ref{Local-limit-theorem-001}.

We set $M = N_G(P) / AU$ to be the quotient of the normalizer of $P$ in $G$ by its normal subgroup $AU$.  
The group $M$ is compact. 
Indeed, $P/AU$ is compact since $P$ is a minimal parabolic subgroup of $G^{\circ}$,  
and as $N_G(P) \cap G^{\circ} = P$ (see for example \cite[Theorem 11.16]{Borel12}),
the group $N_G(P) / P$ is isomorphic to the finite group $G / G^{\circ}$.

Denote by $X$ the space $G (\bb R \varpi) G \subset \bb P ({\rm End}(V))$. 
The set $X$ is equipped with a transitive action of the group $G \times G$: for $(g_1, g_2) \in G \times G$ and $x \in X$,  
we set $(g_1, g_2) x = g_1 x g_2^{-1}$. 
Denote by $\bf P^-$ the opposite minimal parabolic subgroup of $\bf P$ with respect to $\bf A$, 
that is, the unique minimal parabolic subgroup of $\bf G$ such that $\bf P \cap \bf P^-$ is the centralizer of $\bf A$ in $\bf G^\circ$, 
and write $U^-$ for the group of real points of the unipotent radical of $\bf P^-$, see \cite[Section 21]{Borel12}. 

We claim that the orbit map $G \times G \to X: (g_1, g_2) \mapsto g_1 (\bb R \varpi) g_2^{-1}$
induces a surjective map $(G \times G) / (AU \times AU^-) \to X$. 
Indeed, for every $u \in U$, since the range of $\varpi$ is the space of $U$-invariant vectors, we have $u \varpi = \varpi$. 
In the same way, 
the adjoint of $\varpi$ is the $A$-invariant projection onto the space of $U^-$-invariant vectors (see \cite{Tit71}), 
in the dual vector space $V^*$ of $V$, so that $\varpi u = \varpi$ for every $u \in U^-$. 
Moreover, the space $V^+$ is the weight space associated to some character $\chi$ of the group $A$, 
hence, for $a \in A$, we have $a \varpi = \varpi a = \chi(a) \varpi$, which is to say that the line $\bb R \varpi$ is 
a fixed point for the actions of $A$ on the projective space $\bb P ({\rm End}(V))$ both on the left and on the right. 
Therefore, $AU  \times AU^-$ is contained in the stabilizer of $x = \bb R \varpi \in X$ 
under the action of the group $G \times G$. 

Consequently, for proving Theorem \ref{Thm-LLT-intermediate}, 
we will study random trajectories associated to the probability measure $\mu$ on $G$ and the action of $G$ on the first component 
of the product $G/AU \times G/AU^-$, which amounts to studying trajectories of the action of $G$ on $G/AU$. 
Notice that, by Lemma \ref{Lem-principal-bundle} and the definition of $M$, 
the space $G/AU$ is an $M$-principal bundle over $G/N_G(P) = G^{\circ} / P$. 

The next statement says that the action of $\mu$ on $G/AU$ satisfies all the assumptions of Corollary \ref{Cor-equdistribution}. 
Since it does not make use of the action of $G$ on the vector space $V$, 
we forget about this action and just phrase it as a consequence of  
 the fact that $\mu$ is a Borel probability measure on $G$ and $\Gamma_{\mu}$ is Zariski dense in $G$.

\begin{proposition}\label{Prop-check-assum-Corollary}
Let $\bf G$ be a real reductive group and let $\mu$ be a Borel probability measure on $G$. 
Suppose that $\Gamma_{\mu}$ is Zariski dense in $G$ and that $\mu$ admits a finite exponential moment. 
Then the following assertions hold. 

(1) The action of $G$ on $B = G/N_G(P)$ satisfies the contraction property \eqref{contraction-property-on-B}
with respect to any Riemannian distance on $B$.  

(2) The action of $G$ on $\wt B = G/AU$ satisfies the moment condition \eqref{moment-assumption-01} 
with respect to any Riemannian distance on $\wt B$. 

(3) For every $\wt b \in \wt B$ with $\eta(\wt b) \in \supp \nu$,   we have $\wt b M \subset \overline{\Gamma_{\mu} \wt b}$,
where $\nu$ is the unique $\mu$-stationary probability measure on $B$. 
\end{proposition}

\begin{remark}
Recall from \cite{BQ16b} that saying that $\mu$ admits a finite exponential moment is equivalent to saying that 
\eqref{Exponential-moment} holds in a certain faithful rational representation of $\bf G$. 
Then, \eqref{Exponential-moment} actually holds in any rational representation of $\bf G$. 
\end{remark}

\begin{remark}\label{Rem-Riemannian metric}
In order to apply the results of Sections \ref{Sec-Existence of limit measures} and \ref{Section-LLT} 
to the $M$-principal bundle $\wt B = G/AU \to B = G/N_G(P)$, 
we need to equip the space $\wt B$ with an $M$-invariant distance.
As $\wt B$ is a compact smooth manifold and the right $M$-action on $\wt B$ is smooth, 
we can choose an $M$-invariant Riemannian metric on $\wt B$. 
The associated distance only depends on the choice of the metric up to Lipschitz equivalence.
This distance, or rather this Lipschitz class, will work for all our constructions. 
In particular, the moment assumption \eqref{Exponential-moment} for the norm 
implies the moment assumption \eqref{moment-assumption-01} for the Lipschitz constant. 

Indeed, given a Euclidean norm on a vector space $W$, 
we can equip the space $\bb P(W)$ with the distance defined by, 
for $x = \bb R v$ and $y = \bb R w$ in $\bb P(W)$, 
\begin{align*}
\bf d(x, y) = \frac{\| v \wedge w \|}{\|v\| \|w\|}, 
\end{align*}
where the exterior square $\wedge^2 W$ is equipped with the Euclidean norm such that whenever $v$ and $w$ are orthogonal vectors in $W$,
one has $\| v \wedge w \| = \|v\| \|w\|$.
The above distance $\bf d$ on $\bb P(W)$ is not Riemannian. 
Indeed, the distance between the lines $x = \bb R v$ and $y = \bb R w$ is the absolute value of the sine of the angle between these two lines. 
Therefore, the Riemannian distance for the natural Riemannian metric on $\bb P(W)$ associated with the Euclidean norm is defined by the formula:
\begin{align*}
\bf d_R(x, y) = \mbox{arcsin}(\bf d(x, y)). 
\end{align*}
Elementary analysis shows that 
\begin{align*}
\frac{2}{\pi} \bf d_R(x, y) \leq \bf d(x, y) \leq \bf d_R(x, y), 
\end{align*}
which means that these two distances are equivalent. 
Note that in \cite[Section 13]{BQ16b} all inequalities are shown for the distance $\bf d$, but from the above discussion, 
they also hold for the Riemannian distance $\bf d_R$.

By Chevalley's theorem (see for example \cite{Borel12}), we may find a rational representation $\rho: G \to {\rm GL}(W)$ and a line $x_0 \in \bb P(W)$
such that the stabilizer of $x_0$ in $G$ is $AU$. 
Then, the compact quotient space $G/AU$ may be seen as a submanifold of $\bb P(W)$ 
and the distance on $G/AU$ is Lipschitz equivalent to the restrictions of the distances $\bf d$ and $\bf d_R$. 
Now, for $x, y \in \bb P(W)$ and $g \in G $, the definition of the distance $\bf d$ gives 
\begin{align*}
\bf d(gx, gy) \leq \| \rho(g) \|^2  \| \rho(g)^{-1} \|^{2} \bf d(x, y).
\end{align*}
Since the coefficients of $\rho(g)$ are polynomial functions of the coefficients of $g$ and of ${\rm det}(g)^{-1}$, 
we may find a constant $C>0$ such that, for any $g \in G $, we have 
$$\max\{\| \rho(g) \|,  \|\rho(g)^{-1} \| \} \leq \max\{\|g\|,  \|g^{-1} \| \}^C.$$
We obtain 
\begin{align*}
\bf d(gx, gy) \leq  \max\{\|g\|,  \|g^{-1} \| \}^{2C}  \bf d(x, y). 
\end{align*}

\end{remark}

\begin{proof}[Proof of Proposition \ref{Prop-check-assum-Corollary}]
(1) It follows from Lemma 13.5 in \cite{BQ16b} together with Remark \ref{Rem-Riemannian metric}. 
Indeed, this lemma shows a contraction property in $G/P$ which is even a bit stronger than (1), 
since $G/P$ may be identified with $(G/G^{\circ}) \times (G/N_G(P))$. 
By Remark \ref{Rem-Riemannian metric}, the distance considered in \cite{BQ16b} is equivalent to the Riemannian distance that we have chosen on $B$. 

(2) This is proved in Remark \ref{Rem-Riemannian metric}. 

(3) In the case where $\bf G$ is connected, this follows from the results in \cite{BQ14}.
Indeed, on one hand, Lemma 4.8 (i) and (ii) in \cite{BQ14} says that the orbit closure $\overline{\Gamma_{\mu} \wt b}$ is $\Gamma_{\mu}$-minimal
and contains $\wt b$,
while on the other hand, Proposition 4.12 in \cite{BQ14} says that the set $\wt B = G/AU$ contains a unique $\Gamma_{\mu}$-minimal subset. 
(Note that Lemma 4.8 is written under the assumption that $\Gamma$ is a group, but the proof only uses the fact that $\Gamma$ is a semigroup.)
For any $m \in M$, 
we get $\wt b m \in \overline{\Gamma_{\mu} \wt b m} = \overline{\Gamma_{\mu} \wt b}$ and the conclusion follows. 

In the general case, write $\mu_0$ for the probability measure induced by $\mu$ on the finite index normal subgroup $G^{\circ}$ 
(see \cite[Section 4.2]{BQ16b}). 
Since $\Gamma_{\mu_0}$ is a Zariski dense subsemigroup of $G^{\circ}$, by \cite[Lemma 4.7]{BQ16b}, 
the measure $\nu$ is also the unique $\mu_0$-stationary probability measure on $G^{\circ} / P = G/N_G(P)$. 
Therefore, by the connected case, for $\wt b \in \eta^{-1}(\supp \nu)$, 
we have $\overline{\Gamma_{\mu_0} \wt b} \supset \wt b M_1^{\mu}$, where $M_1^{\mu} = (N_G(P) \cap G^{\circ}) / AU = P/AU$.  
Thus, to conclude, it suffices to prove that the action of $\Gamma_{\mu}$ on $\eta_0^{-1}(\supp \nu)$ is minimal, 
where $\eta_0$ is the quotient map $G/P \to G/N_G(P)$.  
Indeed, as mentioned above, the quotient space $G/P$ may be identified with $G/G^{\circ} \times G/N_G(P)$. 
More precisely, the group $G^{\circ}$ acts transitively on $G/N_G(P)$ (see for example \cite[Proposition 21.12]{Borel12})
and $N_G(P) \cap G^{\circ} = P$, so that the diagonal action of $G$ on $G/G^{\circ} \times G/N_G(P)$ is transitive, 
and the stabilizer of the pair of cosets $(G^{\circ}, N_G(P))$ is $P$. 
This gives an identification $G/P$ with $G/G^{\circ} \times G/N_G(P)$ 
under which the map $\eta_0$ becomes the projection onto the second component. 
Since $\nu$ is the unique $\mu_0$-stationary probability measure on $G/N_G(P) = G^{\circ} /P$, 
the semigroup $\Gamma_{\mu} \cap G^{\circ} = \Gamma_{\mu_0}$ is minimal on the support of $\nu$. 
As $\Gamma_{\mu}$ is Zariski dense in $G$, we have $G = \Gamma_{\mu} G^{\circ}$. 
Thus, we obtain that $\Gamma_{\mu}$ is minimal on $G/G^{\circ} \times \supp \nu$. The conclusion follows.  
\end{proof}

\subsection{Equidistribution of the random walk in $\wt B$}\label{Subsection Equidistribution RW}
We continue to discuss properties which just follow from the fact that $\Gamma_{\mu}$ is Zariski dense in $G$. 
Thanks to Proposition \ref{Prop-check-assum-Corollary}, 
we know that the conclusion of Corollary \ref{Cor-equdistribution} holds.
Now in our situation, we can describe more precisely the group $M_0^{\mu}$ and the operator $Q$ mapping $C^0(\wt B)$ to $C^0(M/M_0^{\mu})$. 

\begin{lemma}\label{Lem quotient group Dirac}
Let $\bf G$ be a real reductive group and let $\mu$ be a Borel probability measure on $G$. 
Suppose that $\Gamma_{\mu}$ is Zariski dense in $G$. 
Then, the group $G$ admits a smallest algebraic closed normal subgroup $G^{\mu}$ such that the quotient $G / G^{\mu}$
is compact and the image of $\mu$ in $G / G^{\mu}$ is a Dirac mass.  
\end{lemma}

\begin{remark}\label{Rem topologically cyclic}
Note that the assumption implies that the group $G / G^{\mu}$ is topologically cyclic.
Indeed, let $\ell_0^{\mu}$ be the element in $G / G^{\mu}$ such that the image of $\mu$ is $\delta_{\ell_0^{\mu}}$. 
Then the image of $\Gamma_{\mu}$ in $G / G^{\mu}$ is the semigroup $(\ell_0^{\mu})^{\bb N} = \{ (\ell_0^{\mu})^n: n \in \bb N \}$.
This semigroup is Zariski dense in the compact algebraic group $G / G^{\mu}$. Hence, by Godement's theorem, it is dense. 
\end{remark}

\begin{proof}[Proof of Lemma \ref{Lem quotient group Dirac}]
Recall that, if $L$ is a compact group and $\mu$ is a Borel probability measure on $L$ whose support spans a dense subgroup of $L$, 
then the group $L$ admits a largest topologically cyclic quotient $L/L_0^{\mu}$ such that the image of $\mu$ on $L/L_0^{\mu}$
is a Dirac mass $\delta_{\ell_0^{\mu}}$, see Example \ref{Example-G-compact}. 
Denote by $\bf S$ the semisimple part of $\bf G$, meaning that $\bf S$ is the Zariski connected component of the identity
in the derived group $[\bf G, \bf G]$ of $\bf G$. 
Then the real algebraic group $\bf G^\circ / \bf S$ is a torus and we denote by $\bf T$ its $\bb R$-split component 
so that the group $\bf G^\circ / \bf S \bf T$ is a $\bb R$-anisotropic torus.  
Finally, the group $\bf L: = \bf G / \bf S \bf T$ is a finite extension of the above $\bb R$-anisotropic torus. 
In particular, the group $L = \bf L(\bb R)$ is a compact group, and we denote by $L / L_0^{\mu}$ its largest topologically cyclic quotient
on which the image of $\mu$ is a Dirac mass. 
Set 
\begin{align*}
S = \bf S(\bb R), \quad  T = \bf T(\bb R)  \quad \mbox{and} \quad G^{\mu} = L_0^{\mu} S T. 
\end{align*}
Since $L$ is a compact algebraic group, by Godement's theorem, $L_0^{\mu}$ is Zariski closed in $L$, 
and hence $G^{\mu}$ is Zariski closed in $G$.

If $H$ is a Zariski closed normal cocompact subgroup of $G$ and the image of $\mu$ in $G/H$ is a Dirac mass, 
reasoning as in Remark \ref{Rem topologically cyclic} shows that $G/H$ is topologically cyclic and hence Abelian. 
In particular, $H$ contains $S$, that is, we may view $H$ as a subgroup of $G/S$. 
The image of $T$ in $G/H$ is a split torus in a compact algebraic group, 
hence it is trivial, that is, $T$ is contained in $H$. 
Thus, we may see $H$ as a subgroup of $L = G/ST$, and by definition, we must have $L_0^{\mu} \subset H$. 
\end{proof}

We keep the notation from the proof of Lemma \ref{Lem quotient group Dirac}.  
Since $AU \subset ST \subset G^{\mu}$, we obtain a continuous $G$-equivariant map 
\begin{align*}
\theta: \wt B = G/AU \to G/G^{\mu} = L / L_0^{\mu}.
\end{align*}

Besides, notice that the group $S$ acts transitively on $G/N_G(P) = G^{\circ} / P$, 
so that the group $G^{\mu} = L_0^{\mu} S T$ also acts transitively on $G/N_G(P)$ and therefore, 
we have an identification of $G/(G^{\mu} \cap N_G(P))$ with $(G/G^{\mu})  \times (G/N_G(P))$. 
Now, the space $G/(G^{\mu} \cap N_G(P))$ may be seen as the quotient of the space $\wt B = G/AU$
by the right action of the compact group $M_0^{\mu} = (G^{\mu} \cap N_G(P)) / AU$.  
We define a nonnegative operator $Q: C^0(\wt B) \to C^0(L/L_0^{\mu})$ as follows. 
Given a continuous function $\varphi \in C^0(\wt B)$, we first set for $\wt b \in \wt B$, 
\begin{align*}
Q_1 \varphi(\wt b) = \int_{M_0^{\mu}} \varphi(\wt b m') dm'. 
\end{align*}
Then the function $Q_1 \varphi$ may be seen as a continuous function on $G/(G^{\mu} \cap N_G(P))$, 
or equivalently, a continuous function on $(G/G^{\mu})  \times (G/N_G(P))$. 
In other words, we have built a nonnegative operator 
\begin{align*}
Q_1: C^0(\wt B) \longrightarrow C^0(G/(G^{\mu} \cap N_G(P))) = C^0((G/G^{\mu})  \times (G/N_G(P))). 
\end{align*}

For any continuous function $\psi$ on $(G/G^{\mu})  \times (G/N_G(P))$, we set,
for $\ell \in L/L_0^{\mu} = G/G^{\mu}$, 
\begin{align*}
Q_2 \psi(\ell) = \int_{G/N_G(P)} \psi(\ell, \xi) \nu(d\xi). 
\end{align*}
Thus, we have defined a nonnegative operator 
\begin{align*}
Q_2: C^0((G/G^{\mu})  \times (G/N_G(P))) \longrightarrow  C^0(L/L_0^{\mu}). 
\end{align*}
Finally, we define the operator $Q: C^0(\wt B) \to C^0(L/L_0^{\mu})$ by setting
$$Q = Q_2 Q_1.$$ 
Corollary \ref{Cor-equdistribution} now translates into the following. 

\begin{proposition}\label{Prop-convergence-operator}
Let $\bf G$ be a real reductive group and let $\mu$ be a Borel probability measure on $G$. 
Suppose that $\Gamma_{\mu}$ is Zariski dense in $G$ and that $\mu$ admits a finite exponential moment \eqref{Exponential-moment}.
Then, for any $\varphi \in C^0(\wt B)$, we have, uniformly in $\wt b \in \wt B$,  
\begin{align*}
 P_{\mu}^n \varphi (\wt b) - Q \varphi \big((\ell_0^{\mu})^n \theta(\wt b) \big) \xrightarrow[n\to\infty]{}  0. 
\end{align*}
\end{proposition}

In the above statement, we have denoted by $\ell_0^{\mu}$ the unique element of $L/L_0^{\mu}$ 
such that $\mu$ has image $\delta_{\ell_0^{\mu}}$ in $L/L_0^{\mu}$. 

\begin{remark}
Using the theory of equicontinuous operators (cf.\ \cite{GR07}), 
one could remove the assumption that $\mu$ admits a finite exponential moment \eqref{Exponential-moment} in Proposition \ref{Prop-convergence-operator}. 
\end{remark}

\begin{remark}
By construction, the natural morphism $M \to L$ is surjective. 
As it will be apparent from the proof of Proposition \ref{Prop-convergence-operator}, 
the group $M_0^{\mu}$ which shows up in Corollary \ref{Cor-equdistribution} is the inverse image of $L_0^{\mu} \subset L$ in $M$. 
\end{remark}

The explicit construction of the operator $Q$ relies on a description of the continuous functions on $\wt B = G/AU$,
which are eigenvectors of the operator $P_{\mu}$ associated with an eigenvalue of modulus one. 
The following lemma is an extension of \cite[Lemma 5.6]{BQ14}. 

\begin{lemma}\label{Lem-S-invariant-function}
Let $\bf G$ be a real reductive group and let $\mu$ be a Borel probability measure on $G$. 
Assume that $\Gamma_{\mu}$ is Zariski dense in $G$. 
Suppose that $\varphi \in C^0(\wt B)$ is an eigenfunction of the operator $P_{\mu}$ associated with an eigenvalue of modulus one.  
Then $\varphi$ is $S$-invariant and therefore, when viewed as a function on $L$, it is $L_0^{\mu}$-invariant.  
\end{lemma}

\begin{proof}
Let $\lambda$ be the complex number with modulus one such that $P_{\mu} \varphi = \lambda \varphi$,
and let $Z \subset \{z \in \bb C: |z| = 1\}$ be the closed subgroup spanned by $\lambda$ which is either a finite group of roots of unity
or the whole group of complex numbers with modulus one. 
We can always write $Z$ as the group of real points $\bf Z(\bb R)$, 
where $\bf Z$ is the complex Zariski closure of the group 
\begin{align*}
\left\{ 
\begin{pmatrix}
\Re z & - \Im z \\
\Im z & \Re z
\end{pmatrix}
: z \in Z  \right\}
\subset {\rm GL}_2(\bb C). 
\end{align*}
Less formally, if $Z$ is finite, $\bf Z$ is just $Z$ viewed as a complex algebraic group. 
If $Z$ is the whole circle, $\bf Z$ is the group of complex matrices of the form $\begin{pmatrix}
a & - b \\
b & a
\end{pmatrix}$, where $a, b \in \bb C$ and $a^2 + b^2 = 1$, this group being a copy of $\bb C^*$. 

On the group $G \times Z = (\bf G \times \bf Z)(\bb R)$, consider the Borel probability measure $\mu_1 = \mu \otimes \delta_{\lambda}$.
We claim that the Zariski closure $\bf G_1$ of the subsemigroup $\Gamma_{\mu_1}$ spanned by the support of $\mu_1$
is a reductive subgroup of $\bf G \times \bf Z$ that contains $\bf S \times \{1\}$. 
Indeed, since $\Gamma_{\mu}$ is Zariski dense in $\bf G$, 
$\bf G_1$ has Zariski dense projection onto $\bf G$. As the image of an algebraic group by a regular morphism is Zariski closed 
(cf.\ \cite[Corollary 1.4]{Borel12} or \cite[Section 7.4]{Hum81}), the projection on $\bf G$ maps $\bf G_1$ onto $\bf G$.
Therefore, $\bf G_1$ is the inverse image in $\bf G \times \bf Z$ of the graph of 
a certain homomorphism $\theta: \bf G \to \bf Z/ (\bf G_1 \cap \bf Z)$. 
In particular, as $\bf Z$ is Abelian, $\theta$ is trivial on $\bf S$, so that $\bf S \times \{1\} \subset \bf G_1$.  
Since $\bf G_1/\bf S \times \{1\}$ is isomorphic to a subgroup of $(\bf G/ \bf S) \times \bf Z$ 
and the latter group is a finite extension of a torus, 
any unipotent element in $\bf G_1$ actually belongs to $\bf S \times \{1\}$. 
Thus the unipotent radical of $\bf G_1$ is a normal unipotent subgroup of $\bf S$, 
which is therefore trivial as $\bf S$ is semisimple. 
Hence the group $\bf G_1$ is reductive.

Define a continuous function $\psi$ on $(G/AU) \times Z$ by setting, for $\wt b \in G/AU$ and $z \in Z$, 
\begin{align*}
\psi(\wt b, z) = z^{-1} \varphi(\wt b). 
\end{align*}
Then we have $P_{\mu_1} \psi = \psi$. 
As $G_1 = \bf G_1(\bb R)$ contains $S$, every $G_1$-orbit in $(G/AU) \times Z$ is compact and therefore, 
by \cite[Proposition 4.12]{BQ14}, every $P_{\mu_1}$-invariant continuous function on $(G/AU) \times Z$ is $G_1$-invariant. 
In particular, the function $\psi$ is $S \times \{1\}$-invariant and hence $\varphi$ is $S$-invariant. 
Therefore, we can view $\varphi$ as a continuous function on $L = G/AS$. Since $P_{\mu} \varphi = \lambda \varphi$,
it is $L_0^{\mu}$-invariant by the definition of $L_0^{\mu}$ (see Example \ref{Example-G-compact}). 
\end{proof}

\begin{proof}[Proof of Proposition \ref{Prop-convergence-operator}]
We use the quotient map $\theta: \wt B = G/AU \to G/G^{\mu} = L/L_0^{\mu}$ to identify $C^0(L/L_0^{\mu})$ 
with a subspace of $C^0(\wt B)$. 
Then, notice that, by construction, for any $\varphi \in C^0(L/L_0^{\mu})$, we have 
$Q \varphi = \varphi$, which is to say that $Q$ is a projection onto $C^0(L/L_0^{\mu})$. 
Notice also that, by Lemma \ref{Lem-S-invariant-function}, the space $C^0(L/L_0^{\mu})$ is the closed subspace of $C^0(\wt B)$
spanned by the eigenspaces of $P_{\mu}$ associated with eigenvalues of modulus one. 

We claim that $P_{\mu}$ and $Q$ commute with each other. 
Indeed, $P_{\mu}$ commutes with $Q_1$ as $Q_1$ commutes with all left translations by elements of $\bf G$. 
Besides, for any $g \in \supp \mu$ and any $(\ell, \xi) \in (L/L_0^{\mu})  \times (G/N_G(P))$, 
since the image of the measure $\mu$ in $L/L_0^{\mu}$ is the Dirac mass concentrated on $\ell_0^{\mu}$, 
we have 
\begin{align*}
g (\ell, \xi) = (g \ell, g \xi) = (\ell_0^{\mu} \ell, g \xi). 
\end{align*}
Therefore, for $\psi \in C^0(L/L_0^{\mu} \times G/N_G(P))$ and $\ell \in L/L_0^{\mu}$, we have 
\begin{align*}
Q_2 P_{\mu} \psi (\ell) 
& = \int_{ G/N_G(P) }  \left( \int_{G} \psi(g(\ell, \xi)) \mu(dg) \right)  \nu(d\xi) \notag\\
& = \int_{ G/N_G(P) }  \left( \int_{G} \psi(\ell_0^{\mu} \ell, g \xi) \mu(dg) \right)  \nu(d\xi)  \notag\\
& = \int_{ G/N_G(P) }  \psi(\ell_0^{\mu} \ell, \xi) \nu(d\xi) 
 = Q_2 \psi(\ell_0^{\mu} \ell)  
= P_{\mu} Q_2 \psi(\ell). 
\end{align*}
That is, $P_{\mu}$ also commutes with $Q_2$. Consequently, $P_{\mu}$ commutes with $Q$. 

Note that Proposition \ref{Prop-check-assum-Corollary} implies that the assumptions of Corollary \ref{Cor-equdistribution} are satisfied. 
We will now check that the operator $Q$ constructed above coincides with the one appearing in the proof of Corollary \ref{Cor-equdistribution}. 

First note that each of the spaces $\wt B = G/AU$, $G/(G^{\mu} \cap N_G(P)) = (G/G^{\mu})  \times (G/N_G(P))$
and $L/L_0^{\mu} = G/G^{\mu}$ is equipped with a right action of the group $M = N_G(P) / AU$. 
Note also that, the operators $Q_1$ and $Q_2$ commute with this action. 
Indeed, on one hand, 
$Q_1$ is defined by averaging functions under the action of $M_0^{\mu}$. 
As $G^{\mu}$ is a normal subgroup of $G$, 
it follows that $M_0^{\mu}$ is also a normal subgroup of $M$. Therefore, $Q_1$ commutes with the action of $M$. 
On the other hand, for $(\ell, \xi) \in L/L_0^{\mu}  \times G/N_G(P)$ and $m \in M$, we have $(\ell, \xi) m = (\ell m, \xi)$, 
so that $Q_2$ also commutes with the action of $M$. 
As a consequence, the operator $Q$ preserves the spaces $\pi_x (C^0(\wt B))$ for $x \in \wh M$, 
which are defined in Lemma \ref{Lem-representation-compact-group}. 
The range $Q \pi_x (C^0(\wt B))$ of $Q$ in $\pi_x (C^0(\wt B))$ is the space $\pi_x (C^0(\wt B)) \cap C^0(L/L_0^{\mu})$.
Since 
$S$ acts transitively on $G/N_G(P) = G^{\circ} /P$, 
the group $N_G(P)$ also acts transitively on $G/S$ and hence, as $G^{\mu}$ contains $AS$, the natural morphism from 
$M = N_G(P)/AU$ to $G/G^{\mu} = L/L_0^{\mu}$ is surjective. 
Consequently, the space $\pi_x (C^0(\wt B)) \cap C^0(L/L_0^{\mu})$ can be identified with 
the $x$-isotypic component $\pi_x (C^0(M/M_0^{\mu}))$ of $C^0(M/M_0^{\mu})$. 
As $M/M_0^{\mu}$ is Abelian, if $x$ is not the representation of $M$ associated with a character of $M/M_0^{\mu}$,
then we get 
\begin{align*}
\pi_x (C^0(\wt B)) \cap C^0(L/L_0^{\mu}) = \{0\}. 
\end{align*}
If $x$ is the representation of $M$ associated with the character $\chi$ of $M/M_0^{\mu}$, 
then we get 
\begin{align*}
\pi_x (C^0(\wt B)) \cap C^0(L/L_0^{\mu}) = \bb C \chi. 
\end{align*}
In other words, by using Lemma \ref{Lem-S-invariant-function}, 
we see that, for $x \in \wh M$, the restriction $Q_x$ of the projection $Q$ to the space $\pi_x (C^0(\wt B))$ 
has the same range as the operator constructed in the proof of Corollary \ref{Cor-equdistribution}, 
which to avoid ambiguity we temporarily 
denote by $Q'_x$. Our goal is to show that $Q'_x = Q_x$. 
Again, if $x$ is not the representation of $M$ associated with a character of $M/M_0^{\mu}$, 
then we have $Q_x = 0 = Q'_x$. 
If $x$ is the representation of $M$ associated with the character $\chi$ of $M/M_0^{\mu}$, 
then the construction of $Q'_x$ shows that, for any $\varphi \in C^0(\wt B)$, we have 
\begin{align*}
Q'_x \varphi = \lim_{n \to \infty} \chi(\ell_0^{\mu})^{-n} P_{\mu}^n \varphi.  
\end{align*}
Since $Q$ commutes with $P_{\mu}$, we obtain that $Q_x$ and $Q'_x$ commute with each other.
Hence, $Q_x$ and $Q'_x$ are commuting projectors of the vector space $\pi_x (C^0(\wt B))$ with the same range $\bb C \chi$. 
Therefore, they are equal. 
Since $Q'_x = Q_x$ for any $x \in \wh M$, we obtain that 
the operator $Q$ is the one constructed in the proof of Corollary \ref{Cor-equdistribution}. 

The statement \eqref{convergence-P-mu-meas} of Corollary \ref{Cor-equdistribution} says that,
for any $\varphi \in C^0(\wt B)$, one has 
\begin{align*}
 P_{\mu}^n (\varphi - Q \varphi)  \xrightarrow[n\to\infty]{}  0,
\end{align*}
which is exactly the statement of Proposition \ref{Prop-convergence-operator}. 
\end{proof}

\subsection{Norms and cocycles}\label{Sec-Norms and cocycles}
We come back to the study of the action of $G$ on $V$, that is, from now on, 
we consider $G$ as a subgroup of ${\rm GL}(V)$. 
We shall now use Theorem \ref{Local-limit-theorem-001} to prove the local limit theorem (Theorem \ref{Thm-LLT-intermediate})
 on $X = G (\bb R \varpi) G \subset \bb P ({\rm End}(V))$ 
for a cocycle $\sigma$ associated with the choice of the norm $\| \cdot \|$ on the space $\End(V)$. 
We will actually work on the space $Y: = (G \times G) /(AU \times AU^-)$ of which $X$ may be seen as a quotient, 
thanks to the orbit map $(g_1, g_2) \mapsto g_1 (\bb R \varpi) g_2^{-1}$. 
For $g \in G$ and $y = (g_1, g_2) (AU \times AU^-) = (g_1 AU) \times (g_2 AU^-)  \in Y$, 
we set 
\begin{align*}
\sigma(g, y) = \log \frac{\| g g_1 \varpi g_2^{-1} \|}{\| g_1 \varpi g_2^{-1} \|}, 
\end{align*}
so that $\sigma: G \times Y \to \bb R$ is a continuous cocycle for the action of $G$ on $Y$ by the multiplication of the first factor. 
Note that, for this action, every $G$-orbit may be identified with the space $G/AU$. 
We will now show that the restriction of $\sigma$ to such an orbit satisfies the assumptions of Theorem \ref{Local-limit-theorem-001}. 
To this aim, we need to build another cocycle $\sigma_0: G \times (G/N_G(P)) \to \bb R$.

To define $\sigma_0$, we introduce a new vector space $W$ and a representation $\rho: G \to {\rm GL}(W)$ as follows. 
Recall that the image of the projection $\varpi$ is the space $V^+$ of $U$-invariant vectors in $V$.  
Let $r$ be the dimension of $V^+$, which is called the proximal dimension of $G$ on $V$. 
We then let $W$ be the linear span of the lines $g_1 \wedge^r V^+$ in the space $\wedge^r V$ when $g_1$ runs over $G$. 
Note that the natural action of $G$ on $\wedge^r V$ preserves $W$. We denote by $\rho$ this representation of $G$.  
By Lemma 4.36 and Remark 4.37 in \cite{BQ16b}, 
the representation $\rho$ is proximal and strongly irreducible. 

We now choose good norms on $V$ and $W$. Fix a maximal compact subgroup $K$ of $G^{\circ}$ such that, 
if $\theta$ is the associated Cartan involution of $G^{\circ}$, we have $\theta(a) = a^{-1}$ for any $a \in A$. 
The normalizer $N_G(K)$ in $G$ is compact and we have the Cartan decomposition $G = K A N_G(K)$. 

\begin{lemma}
We may choose a Euclidean norm $\scr N_{V}$ on $V$
which is $N_G(K)$-invariant and for which the elements of $A$ are symmetric. 
\end{lemma}

\begin{proof}
We recall the proof of this classical fact in the case where $\bf G$ is connected. 
We use freely the language of \cite[Chapter VI]{Hel01}. 
Let $\frak g$ be the Lie algebra of $G$ and let $\frak k \subset \frak g$ be the Lie algebra of $K$.
Also, denote by $\frak p \subset \frak g$ the orthogonal subspace to $\frak k$ with respect to the Killing form. 
The restriction of the Killing form to $\frak k$ is negative definite 
and the restriction of the Killing form to $\frak p$ is positive definite. 
Define $\frak h = \frak k \oplus i \frak p$, which is a real subalgebra of the complexification $\frak g_{\bb C}$ of $\frak g$. 
Since the restriction of the real part of the Killing form to $\frak h$ is negative definite, 
this algebra is the Lie algebra of a compact connected subgroup $H$ of $\bf G$. 
As $H$ is normalized by $N_G(K)$, the set $N_G(K) H$ is a subgroup of $\bf G$, 
which is compact since $K$ has finite index in $N_G(K)$ and $K \subset H$.  
Now the action of $N_G(K) H$ on the complexification $V_{\bb C}$ preserves a Hermitian scalar product. 
With respect to this scalar product, the action of these elements of $\frak h$ is skew Hermitian. 
Hence, the action of the elements of $\frak p$ is Hermitian. 
We take $\scr N_{V}$ to be the induced norm on $V$.  
\end{proof}

We equip $\wedge^r V$ with the associated tensor norm whose restriction to the subspace $W$ is denoted by $\scr N_{W}$. 

\begin{lemma}\label{Lem-Norms-V-W}
For any $g \in G $, we have 
$\scr N_{V}(g) = \scr N_{W}(\rho(g))^{1/r}$.
\end{lemma}

\begin{proof}
Let $\Sigma$ be the set of relative roots, that is, $\Sigma$ is the set of algebraic characters of the group $\bf A$ 
which are non-trivial weights of the adjoint representation of $\bf G$. 
The choice of the minimal parabolic subgroup $\bf P$ fixes the choice of a system of positive roots $\Sigma^+ \subset \Sigma$.
The elements of $\Sigma^+$ are the weights of the restriction of the adjoint representation to the Lie algebra of $\bf U$. 
Recall that we have denoted by $\theta$ the weight of the action of $A$ on $V^+$. 
The other weights of $A$ on $V$ are of the form 
$\theta \prod_{\alpha \in \Sigma^+} \alpha^{-m_{\alpha}}$ for some natural integers $m_{\alpha}$. 
In particular, if 
\begin{align*}
A^+ = \left\{ a \in A: |\alpha(a)| \geq 1, \forall \alpha \in \Sigma^+   \right\}, 
\end{align*}
then we have $\scr N_{V}(a) = |\theta(a)|$ for any $a \in A^+$. 
In the same way, as $V^+$ has dimension $r$ and $\wedge^r V^+ \subset W$, for $a \in A^+$, 
we have $\scr N_{W}(\rho(a)) = |\theta(a)|^r$. 
Hence, the lemma holds for elements in $A^+$. 

The Cartan decomposition says that we have $G = K A^+ N_G(K)$. 
Since the norms $\scr N_{V}$ and $\scr N_{W}$ are $N_G(K)$-invariant, the conclusion of the lemma holds for any $g \in G$. 
\end{proof}

Now we define the continuous cocycle $\sigma_0: G \times (G/N_G(P)) \to \bb R$ as follows. 
The space $V^+ \subset V$ is exactly the space of $U$-invariant vectors in $V$.
Therefore, since $U$ is the unipotent radical of $P$, this space is $N_G(P)$-invariant 
and the line $W^+: = \wedge^r V^+ \subset W$ also is $N_G(P)$-invariant. 
Hence, the orbit map induces a surjection $B = G/N_G(P) \to G W^+ \subset \bb P (W)$. 
For $g \in G $ and $b = g' N_G(P) \in B$, we set 
\begin{align}\label{def-sigma-0}
\sigma_0(g, b) = \frac{1}{r} \log \frac{ \scr N_{W}( \rho(g g') w) }{\scr N_{W}( \rho(g') w)}, 
\end{align}
where $w$ is a non-zero element in $W^+$. 
Note that this definition does not depend on the choices of $g'$ and $w$. 

The next lemma asserts that, on each $G \times \{e\}$-orbit in $Y = (G \times G) /(AU \times AU^-)$, the restriction of the cocycle $\sigma$
is cohomologous to $\sigma_0$, as required in the assumptions of Theorem \ref{Local-limit-theorem-001}. 
We denote by $\eta: Y \to B$ the map $(g_1, g_2) (AU \times AU^-) \to g_1 N_G(P)$.  

\begin{lemma}\label{Lem-cohomology-cocycle}
There exists a continuous function $\varphi: Y \to \bb R$ such that, for any $g \in G$ and $y \in Y$, 
\begin{align*}
\sigma(g, y) = \sigma_0(g, \eta(y)) + \varphi(gy) - \varphi(y). 
\end{align*}
\end{lemma}

\begin{proof}
Note that the representation $\rho: G \to {\rm GL}(W)$ 
may be extended as a continuous morphism from the semigroup $\overline{\bb R G} \subset \End(V)$ to $\End(W)$. 
Indeed, for any $v \in \overline{\bb R G}$, the endomorphism $\wedge^r v$ preserves the space $W \subset \wedge^r(V)$. 
By continuity, Lemma \ref{Lem-Norms-V-W} still applies to all elements in $\overline{\bb R G}$. 

Define, for $y = (g_1, g_2) (AU \times AU^-) \in Y$, 
\begin{align*}
\varphi(y) = \log \frac{ \| g_1 \varpi g_2^{-1} \| }{ \scr N_{V}(g_1 \varpi g_2^{-1}) }. 
\end{align*}
For $g \in G$, a direct computation yields 
\begin{align}\label{cohomological-equ-001}
\sigma(g, y) - \varphi(gy) + \varphi(y) 
& = \log \frac{  \scr N_{V}(g g_1 \varpi g_2^{-1}) }{  \scr N_{V}(g_1 \varpi g_2^{-1}) }  \notag\\
& = \frac{1}{r}  \log \frac{  \scr N_{W} \left( \rho ( g g_1 \varpi g_2^{-1}) \right) }{  \scr N_{W} \left( \rho(g_1 \varpi g_2^{-1}) \right) }, 
\end{align}
where the last equality follows from the fact that $\varpi \in \overline{\bb R A^+}$ and using Lemma \ref{Lem-Norms-V-W}.  
By construction, the endomorphism $\wedge^r \varpi$ of $\wedge^r V$ is a projection onto $W^+$, 
and hence the endomorphism $\rho(\varpi)$ is also a projection onto $W^+$. 
Fix a non-zero vector $w_0 \in W^+$, then we may find a linear form $f$ of the space $W$ such that,
for any $w \in W$, we have $\rho(\varpi) w = f(w) w_0$. 
Consequently, we get that, for any $g_1, g_2 \in G$ and $w \in W$, 
\begin{align*}
\rho \left( g_1 \varpi g_2^{-1} \right) w 
=  \rho(g_1)  \rho(\varpi) \rho(g_2)^{-1}  w
= f \left( \rho (g_2)^{-1} w \right) \rho(g_1) w_0
\end{align*}
and hence the operator norm of $\rho(g_1 \varpi g_2^{-1})$ is given by 
\begin{align*}
\scr N_{W} \left( \rho(g_1 \varpi g_2^{-1}) \right) = \scr N^*_{W} \left( \rho^*(g_2) f \right)  \scr N_{W}( \rho(g_1) w_0),
\end{align*}
where we have denoted by $\scr N^*_{W}$ the norm on the dual space $W^*$ of $W$, which is dual to the norm $\scr N_{W}$,
and by $\rho^*$ the dual representation of $\rho$. 
In the same way, we also have 
\begin{align*}
\scr N_{W} \left( \rho(g g_1 \varpi g_2^{-1}) \right) = \scr N^*_{W}(\rho^*(g_2) f)  \scr N_{W}( \rho(g g_1) w_0). 
\end{align*}
The conclusion now follows by taking the ratio of the last two displayed formulas,
and comparing \eqref{def-sigma-0} with \eqref{cohomological-equ-001}. 
\end{proof}

\subsection{Extended non-arithmeticity}

We have defined in \eqref{def-sigma-0} the cocycle $\sigma_0$ on the base space $B = G/N_G(P)$. 
Thanks to \cite[Corollary 13.4]{BQ16b}, we know that the moment assumptions \eqref{moment-assumption-02} of Theorem \ref{Local-limit-theorem-001} are satisfied. 
Thus, to apply this theorem in the $M$-principal bundle $\wt B = G/AU$, 
we need to check that the group $M_1^{\mu}$ introduced in Proposition \ref{Propo-eigenvalue-one-002}
verifies $M_1^{\mu} = M_0^{\mu} \times \bb R$.

\begin{proposition}\label{Prop-Non-arithmeticity}
Let $\mu$ be a Borel probability measure on ${\rm GL}(V)$. 
Assume that $\Gamma_{\mu}$ is strongly irreducible on $V$ and that $\mu$ admits a finite exponential moment \eqref{Exponential-moment}. 
Assume also that the image of $\Gamma_{\mu}$ in ${\rm PGL}(V)$ is not contained in a compact subgroup of ${\rm PGL}(V)$. 
Then, we have $M_1^{\mu} = M_0^{\mu} \times \bb R$. 
\end{proposition}

We will use the assumption that 
the image of $\Gamma_{\mu}$ in ${\rm PGL}(V)$ is not contained in a compact subgroup of ${\rm PGL}(V)$, 
under the following form. 
Recall from Subsection \ref{Subsection Equidistribution RW} that we defined $S$ as the connected semisimple part of the group $G$.

\begin{lemma}\label{Lem image of G}
Let $G \subset {\rm GL}(V)$ be a Zariski closed subgroup acting strongly irreducibly on $V$. 
Then, the image of $G$ in ${\rm PGL}(V)$ is not compact if and only if the highest weight $\theta$ has non-trivial 
restriction to the group $A \cap S$. 
\end{lemma}

\begin{proof}
Clearly, if $G$ is contained in $\bb R^* K$ for some compact subgroup $K \subset {\rm GL}(V)$,
then $S$ is compact and $A \cap S$ is a trivial group. 
Conversely, since the product of $S$ with the center of $G^{\circ}$ is a finite index subgroup of $G^{\circ}$,
the character $\theta|_{A \cap S}$ is the highest weight of every irreducible $S$-submodule of $V$. 
By \cite{Tit71}, this says that $S$ is compact. 
The center of $G^{\circ}$ is contained in the center of the commutant of $G^{\circ}$, 
which, by Schur's lemma, is a division algebra. 
Therefore, this center has compact image in ${\rm PGL}(V)$. 
Again, as the product of $S$ with the center of $G^{\circ}$ is a finite index subgroup of $G$,
we conclude that $G$ has compact image in ${\rm PGL}(V)$. 
\end{proof}

\begin{proof}[Proof of Proposition \ref{Prop-Non-arithmeticity}]
This result is a translation of the main result of \cite{GQX26}. 
We explain how to deduce it. 
Set $\wt b_0 \in \wt B = G / AU$ to be the base point and $b_0 = \eta(\wt b_0)$ to be the base point of $B = G/N_G(P)$.

Since $M_0^{\mu}$ and $M_1^{\mu}$ are defined through duality, 
we will actually prove the dual statement which is to say that the orthogonal group $(M_1^{\mu})^{\perp}$
is equal to $(M_0^{\mu})^{\perp} \times \{0\}$. 
By Remark \ref{Rem-group-M0-M1}, we know that $(M_0^{\mu})^{\perp} \times \{0\}$ is contained in 
$(M_1^{\mu})^{\perp}$, so it only remains to prove that for any $(\chi, t) \in \wh M \times \bb R$ 
belonging to the orthogonal group of $M_1^{\mu}$, we have $t = 0$.

Now we recall the precise definition of $M_1^{\mu}$ from Proposition \ref{Propo-eigenvalue-one-002}. 
If $(\chi, t) \in \wh M \times \bb R$ belongs to the orthogonal group of $M_1^{\mu}$, 
then there exists a continuous function $\varphi: \wt B \to \bb C$ with constant modulus one on the set $\eta^{-1} (\supp \nu)$ 
and a complex number $w$
with $|w| = 1$ such that, for any $g \in \supp \mu \subset G$, $m \in M$ and $\wt b \in \eta^{-1} (\supp \nu)$, 
we have 
\begin{align*}
\varphi(g \wt b) = w e^{- it \sigma_0(g, \eta(\wt b))} \varphi(\wt b)
\quad \mbox{and} \quad
\varphi(\wt b m) = \chi(m)^{-1} \varphi(\wt b). 
\end{align*} 
For $h \in G$, define
\begin{align*}
\psi(h) = \varphi(h \wt b_0) e^{it \sigma_0(h, b_0)}. 
\end{align*}
If $h b_0 \in \supp \nu$, then we get that, for $g \in \supp \mu$, 
\begin{align}
 |\psi(h)| & = 1,  \label{eigenvalue equation psi00} \\
 \psi(gh)  & = \varphi(gh \wt b_0) e^{it \sigma_0(gh, b_0)}  \notag\\
& = w e^{- it \sigma_0(g, h b_0)} \varphi(h \wt b_0) e^{it \sigma_0(gh, b_0)}  = w \psi(h),  \label{eigenvalue equation psi}
\end{align}
where for the last equality we have used the cocycle identity $\sigma_0(gh, b_0) - \sigma_0(g, h b_0) = \sigma_0(h, b_0)$. 
Note that, by the definition of $\sigma_0$ in \eqref{def-sigma-0}, for $h \in G$ and $a \in A$, we have 
\begin{align}\label{equivariance psi 001}
\psi(ha) & = \varphi(ha \wt b_0) e^{it \sigma_0(ha, b_0)} \notag\\
& = \varphi(h \wt b_0) e^{it \sigma_0(h, b_0)} e^{it \sigma_0(a, b_0)}
= \psi(h)  |\theta(a)|^{it}, 
\end{align}
where $\theta$ is the highest weight of $A$ in $V$. 

We would also like to deduce a property of $\psi$ coming from the fact that 
$\varphi$ is $\chi$-equivariant with respect to the action of $M$ on the right.  
Until now, we have defined the group $M$ as being the quotient $N_G(P)/AU$. 
Note that it is also equal to the quotient $(N_G(P) \cap N_G(A))/A$: 
this follows from conjugacy of the maximal split tori in $P$, see \cite[Theorem 15.14]{Borel12}. 
We set 
\begin{align*}
\wt M = N_G(P) \cap N_G(A) \cap N_G(K), 
\end{align*}
where $K$ is the maximal compact subgroup of $G^{\circ}$ used in Subsection \ref{Sec-Norms and cocycles},
and we have $N_G(P) \cap N_G(A) = \wt M A$. 
Indeed, this follows from the uniqueness of the $A$-invariant flat submanifold in the symmetric space of $G$, 
see \cite[Chapter VI]{Hel01}. 
The group $Z = (\wt M A) \cap G^\circ$ is the centralizer of $A$ in $G^\circ$. 

Denote by $\wt \chi$ the pullback of the character $\chi$ under the natural morphism $\wt M \to M$. 
For $h \in G$ and $m \in \wt M$, 
we have 
\begin{align}\label{equivariance psi 002}
\psi(h m) & = \varphi(h m \wt b_0) e^{it \sigma_0(hm, b_0)}  \notag\\
& = \varphi(h \wt b_0 m) e^{it \sigma_0(h, mb_0)} e^{it \sigma_0(m, b_0)} 
= \wt \chi(m)^{-1} \psi(h), 
\end{align}
where we have used the fact that $\sigma_0(m, b_0) = 0$ 
which follows from the definition of $\sigma_0$ in \eqref{def-sigma-0}, since the norm $\scr N_{W}$ is $N_G(K)$-invariant. 

Since $\wt M \cap A$ is a compact (hence finite) subgroup of $A$, 
for all $a \in \wt M \cap A$, we have $|\theta(a)| = 1$. 
Moreover, the character $\wt \chi$ vanishes on $\wt M \cap A$ since 
$\wt M \cap A$ is the kernel of the natural morphism $\wt M \to M$. 
Therefore, there exists a unique character $\kappa$ of $Z$ whose restriction to $\wt M \cap G^\circ$ is $\wt \chi$
and whose restriction to $A$ is $a \mapsto |\theta(a)|^{-it}$. 
By \eqref{equivariance psi 001} and \eqref{equivariance psi 002}, 
for $h \in G$ and $z \in Z$, we have 
\begin{align}  \label{equivariance psi 003}
\psi(hz) = \kappa(z)^{-1} \psi(h). 
\end{align}

Recall that an element $g \in G^{\circ}$ is said to be loxodromic if it is conjugated to an element of $(\wt M \cap G^\circ) \exp(\frak a^{++})$,
where $\frak a^{++}$ is the interior of the cone $\frak a^{+}$ from \eqref{def a plus},
that is, 
\begin{align*}
\mathfrak a^{++} = \left\{ X \in \mathfrak a: \alpha(X) > 0, \forall \alpha \in \Sigma^+ \right\}. 
\end{align*}
If $g$ is loxodromic and $\gamma \in G^\circ$ is such that $\gamma^{-1} g \gamma$ belongs to $(\wt M \cap G^\circ) \exp(\frak a^{++})$,
then the image of $\gamma^{-1} g \gamma$ in $Z/[Z, Z]$ does not depend on the choice of $\gamma$.
As in \cite{GQX26}, we denote this element of $Z/[Z, Z]$ by $\lambda(g)$. 
By \cite[Lemma 10.3]{BQ16b}, if $g$ is a loxodromic element of $G^\circ$, 
then its attracting fixed point $\xi_g^+$ in $G / N_G(P) = G^\circ / P$ belongs to the support of the $\mu$-stationary probability measure $\nu$. 
In other words, if $\gamma$ is as above, then $\gamma b_0 \in \supp \nu$. 
Therefore, on one hand, by \eqref{eigenvalue equation psi00} and \eqref{eigenvalue equation psi}, if $n \geq 1$ is such that $g \in \supp \mu^{*n}$, 
we get 
\begin{align*}
|\psi(\gamma)| = 1
\quad \mbox{and}  \quad 
\psi(g \gamma) = w^n \psi(\gamma). 
\end{align*}
On the other hand, by \eqref{equivariance psi 003}, 
we get 
\begin{align*}
\psi(g \gamma) = \psi(\gamma \lambda(g)) = \psi(\gamma) \kappa(\lambda(g))^{-1}, 
\end{align*}
so that $\kappa(\lambda(g)) = w^{-n}$. 

Now let $p \geq 1$ and $h \in \supp \mu^{*p} \cap G^\circ$, which is also loxodromic.
In the same way as above, we have $\kappa(\lambda(h)) = w^{-p}$. 
If $gh$ is also loxodromic, we get $\kappa(\lambda(gh)) = w^{-n-p}$. 
Consequently, we obtain 
\begin{align}\label{kappa equation 001}
\kappa \left( \lambda(gh) \lambda(g)^{-1} \lambda(h)^{-1} \right) = 1. 
\end{align}
By \cite[Theorem 1.1]{GQX26} (see \cite{Benoist05, GR07} for related results), 
when $(g, h)$ runs among the set of all pairs of loxodromic elements of $\Gamma_{\mu} \cap G^\circ$
for which $gh$ is loxodromic,   
the closed subgroup spanned by the elements 
$\lambda(gh) \lambda(g)^{-1} \lambda(h)^{-1}$ contains an open subgroup of $(S \cap Z) / [Z, Z]$.  
In this sentence, closed and open are understood in the sense of the locally compact topology. 
Therefore, by \eqref{kappa equation 001}, we have $\kappa = 1$ on $\exp (\frak a \cap {\rm Lie}(S))$, 
and hence $t \theta = 0$ on $A \cap S$. 
Using Lemma \ref{Lem image of G}, we conclude that $t = 0$, as required. 
\end{proof}

\subsection{Proof of Theorem \ref{Thm-LLT-intermediate}}
We can now use the above constructions to conclude that the assumptions in Theorem \ref{Local-limit-theorem-001} 
are satisfied by the restriction of the cocycle $\sigma$ to $(G \times \{e\})$-orbit in $Y = (G \times G) /(AU \times AU^-)$. 
Indeed, any such orbit is a copy of $\wt B = G/AU$, and by Proposition \ref{Prop-check-assum-Corollary}, 
the moment assumption \eqref{moment-assumption-01} holds on $\wt B$, the contraction property \eqref{contraction-property-on-B} holds on $B = G/N_G(P)$,
and for every $\wt b \in \wt B$ with $\eta(\wt b) \in \supp \nu$,  we have $\wt b M \subset \overline{\Gamma_{\mu} \wt b}$,
where $\nu$ is the unique $\mu$-stationary probability measure on $B$.  
Thanks to \cite[Corollary 13.4]{BQ16b} and Remark \ref{Rem-Riemannian metric}, 
the cocycle $\sigma_0$ fulfills the moment assumptions \eqref{moment-assumption-02}. 
By Lemma \ref{Lem-cohomology-cocycle}, the restriction of the cocycle $\sigma$ to 
any $(G \times \{ e \})$-orbit in $Y$ is cohomologous to $\sigma_0$. 
Finally, Proposition \ref{Prop-Non-arithmeticity} says that $M_1^{\mu} = M_0^{\mu} \times \bb R$. 
The conclusion of Theorem \ref{Thm-LLT-intermediate} now follows from Theorem \ref{Local-limit-theorem-001}.

\section{Local limit theorems for the norm}\label{Sec-LLT-approximation}

The aim of this section is to deduce Theorem  \ref{Main-Thm-LLT} from Theorem \ref{Thm-LLT-intermediate}. 
We keep the notation introduced in Section \ref{Sec-action-Strong-irre}. 
Our strategy consists in writing $\log \|g_n \cdots g_1\|$ as 
\begin{align*}
\log \|g_n \cdots g_1\| = \log \Big\|g_n \cdots g_{p+1} \frac{g_p \cdots g_1}{\| g_p \cdots g_1 \|}  \Big\| +  \log \|g_p \cdots g_1\|, 
\end{align*}
where $p$ is much smaller than $n$, so that the term $\log \|g_p \cdots g_1\|$ will be dealt with as a small deviation.
We will show that, with high probability, 
the normalized endomorphism $\frac{g_p \cdots g_1}{\| g_p \cdots g_1 \|}$ of $V$ is very close to an element $h$ of the set $\bb R^* G \varpi G$. 
By using a continuity property, one can then prove that the logarithm of norms $\log \|g_n \cdots g_{p+1} \frac{g_p \cdots g_1}{\| g_p \cdots g_1 \|} \|$
and $\log \|g_n \cdots g_{p+1} h \|$ are also close, and finally conclude by applying Theorem \ref{Thm-LLT-intermediate}.

\subsection{Approximation of the norm}
Recall that we have fixed a maximal compact subgroup $K$ of $G^{\circ}$ so that the associated Cartan involution is the inverse map on $A$. 
Let $\mathfrak a$ be the Lie algebra of $A$ and let $\mathfrak a^+ \subset \mathfrak a$ be the Weyl chamber associated with 
the choice of the minimal parabolic subgroup $\bf P$ of $G^{\circ}$. 

More precisely, the set $\mathfrak a^+$ is a closed convex cone in $\mathfrak a$, which is defined as follows.
We take $\Sigma \subset \mathfrak a^*$ to be the set of roots of $\mathfrak a$ in the Lie algebra $\mathfrak g$ of $G$,
that is, $\Sigma$ is the set of non-zero linear forms $\alpha$ on $\mathfrak a$ for which 
the space 
\begin{align*}
\mathfrak g_{\alpha} = \{ X \in \mathfrak g:  [Z, X] = \alpha(Z) X,  \,  \forall Z \in \mathfrak a \} 
\end{align*}
is not zero. 
The Lie algebra $\mathfrak p$ of $P$ is decomposed as $\mathfrak p = \mathfrak z \oplus \bigoplus_{\alpha \in \Sigma^+} \mathfrak g_{\alpha}$,
where $\mathfrak z$ is the Lie algebra of the centralizer of $A$ in $G$,
and $\Sigma^+ \subset \Sigma$ is a subset satisfying $\Sigma = \Sigma^+ \cup (- \Sigma^+)$. 
We set 
\begin{align}\label{def a plus}
\mathfrak a^+ = \left\{ Z \in \mathfrak a: \alpha(Z) \geq 0,  \, \forall \alpha \in \Sigma^+ \right\}
\end{align}
and $A^+ = \exp (\mathfrak a^+)$. 

Then we have the Cartan decomposition $G = K A^+ N_G(K)$. 
More explicitly, for every $g \in G $, there exists a unique $a_g \in A^+$ 
such that $g$ belongs to $K a_g N_G(K)$. 
Once and for all, for each $g \in G $, we fix $k_g \in K$ and $n_g \in N_G(K)$ such that $g = k_g a_g n_g$.
We define $\pi_g$ as being the endomorphism $k_g \varpi n_g$ of $V$. 
Since the pair $(k_g, n_g)$ is defined up to the action of a closed subgroup of $K$,
by taking Borel sections, we can assume that the map $g \mapsto \pi_g$ is Borel measurable. 

The next proposition says that, when computing norms of big products of elements in $G$, 
we can replace some of them by elements in $\bb R^* G \varpi G$ without significantly changing the result.

\begin{proposition}\label{Prop-norm-comparison}
Assume that $\Gamma_{\mu}$ is strongly irreducible and that $\mu$ admits a finite exponential moment \eqref{Exponential-moment}. 
Then there exist constants $\ee, c, C >0$ such that for any $n \geq 1$ and any $g \in G $, 
\begin{align*}
& \bb P \Big( \Big| \log \| g g_n \cdots g_1 \|  - \log \| g_n \cdots g_1 \|
  - \log \| g \pi_{g_n \cdots g_1} \| +  \log \|  \pi_{g_n \cdots g_1} \| \Big| > c e^{-\ee n} \Big) \notag\\
  &  \leq C e^{-\ee n}. 
\end{align*} 
\end{proposition}

The proof of Proposition \ref{Prop-norm-comparison} relies on the following two lemmas. 

\begin{lemma}\label{Lem-norm-comparison-001}
Assume that $\Gamma_{\mu}$ is strongly irreducible and that $\mu$ admits a finite exponential moment \eqref{Exponential-moment}. 
Then, there exist constants $\ee, c >0$ such that for any $n \geq 1$, 
\begin{align*}
\bb P \left( \left\| \frac{g_n \cdots g_1}{\| g_n \cdots g_1 \|} - \frac{\pi_{g_n \cdots g_1} }{\| \pi_{g_n \cdots g_1} \|}  \right\| > e^{- \ee n} \right) \leq c e^{-\ee n}. 
\end{align*}
\end{lemma}

\begin{proof}
Let $\sigma_{\mu} \in \mathfrak a^+$ be the Lyapunov vector of $\mu$ 
(see Theorem 10.9 of \cite{BQ16b}), so that, almost surely, 
\begin{align*}
 \frac{1}{n} \log a_{g_n \cdots g_1}  \xrightarrow[n\to\infty]{}  \sigma_{\mu}. 
\end{align*}
Still by Theorem 10.9 of \cite{BQ16b}, if $\alpha$ is in $\Sigma^+$, 
then we have $\alpha(\sigma_{\mu}) > 0$. 
Therefore, by Theorem 13.17 of \cite{BQ16b}, we may find constants $\ee, c > 0$ such that for any $n \geq 1$ and $\alpha \in \Sigma^+$, 
\begin{align*}
\bb P \big( \alpha( \log a_{g_n \cdots g_1} ) \leq \ee n \big)
\leq c e^{- \ee n}. 
\end{align*}
Recall that the range $V^+$ of $\varpi$ is the weight space associated to the highest weight $\theta$ of $A$
in $V$.
This means that $\theta$ is an element of $\mathfrak a^*$ and that for any $Z \in \mathfrak a$ and $v \in V^+$,
we have $\exp(Z) v = e^{\theta(Z)} v$. 
Besides, if $\theta'$ is a weight of $A$ in the kernel $V^-$ of $\varpi$, then the difference $\theta - \theta'$ is non-zero and 
is a sum of elements in $\Sigma^+$, see \cite{Tit71}. 
Consequently, for any $Z \in \mathfrak a^+$, we have 
\begin{align*}
\mathscr N_{V} \left(e^{- \theta(Z)} \exp{Z} - \varpi \right) \leq \max_{ \alpha \in \Sigma^+}  e^{- \alpha(Z)}, 
\end{align*}
where $\mathscr N_{V}$ is the Euclidean norm introduced in Section \ref{Sec-Norms and cocycles}. 
In particular, we obtain 
\begin{align*}
\bb P \left( \mathscr N_{V} \bigg(  \frac{a_{g_n \cdots g_1}}{\mathscr N_{V}(a_{g_n \cdots g_1})} - \varpi  \bigg) > e^{-\ee n}  \right)
\leq c e^{-\ee n}. 
\end{align*}
The conclusion follows by using the equivalence of the norms on $\End (V)$. 
\end{proof}

\begin{lemma}\label{Lem-norm-comparison-002}
Assume that $\Gamma_{\mu}$ is strongly irreducible and that $\mu$ admits a finite exponential moment \eqref{Exponential-moment}. 
Then, for any $\ee >0$, there exist constants $a, c  >0$ such that for any $n \geq 1$ and any $g \in G $, 
\begin{align*}
& \bb P \Big( \|g g_n \cdots g_1\| \leq  e^{-\ee n} \|g\| \|g_n \cdots g_1\| \Big) \leq  c e^{- a n},  \notag\\
& \bb P \Big( \|g \pi_{g_n \cdots g_1} \| \leq  e^{-\ee n} \|g\| \|\pi_{g_n \cdots g_1}\| \Big) \leq c e^{- a n}. 
\end{align*}
\end{lemma}

\begin{proof}
By equivalence of norms on the space $\End(V)$, we can assume that $\| \cdot \|$ is the operator norm associated 
with the Euclidean norm $\mathscr N_{V}$ introduced in Section \ref{Sec-Norms and cocycles}. 
In that case, by Lemma \ref{Lem-Norms-V-W}, up to replacing $V$ by $W$,
we can assume $\Gamma_{\mu}$ to be proximal. 
We fix a vector $v \in V$ with norm one. 
Then, by Lemma 14.2 (i) and Proposition 14.3 (inequalities (14.8) and (14.5)) of \cite{BQ16b}, for any $g \in G$, 
on a set whose complement has probability at most $c e^{-a n}$, 
we get 
\begin{align*}
\| g g_n \cdots g_1 \| 
& \geq \| g g_n \cdots g_1 v \| \notag\\
& \geq  e^{- \ee n}  \|g\| \| g_n \cdots g_1 v \|
 \geq   e^{- 2 \ee n} \|g\|  \| g_n \cdots g_1 \|. 
\end{align*}
Replacing $2 \ee$ by $\ee$ (after adjusting the constants) yields the first estimate.
In the same way, still by Lemma 14.2 (i) and Proposition 14.3 (inequality (14.7)) of \cite{BQ16b}, 
for $g \in G $, 
on a set whose complement has probability smaller than $c e^{-a n}$, 
we have 
\begin{align*}
\| g \pi_{g_n \cdots g_1} \|  
 \geq  e^{- \ee n} \|g\|  \| \pi_{g_n \cdots g_1} \|, 
\end{align*}
which gives the second estimate. 
This completes the proof of the lemma. 
\end{proof}

We can now conclude this subsection with the following proof. 

\begin{proof}[Proof of Proposition \ref{Prop-norm-comparison}]
For $g \in G $, we need to dominate 
\begin{align*}
& \Big|  \log \| g g_n \cdots g_1 \|  - \log \| g_n \cdots g_1 \|
 - \log \| g \pi_{g_n \cdots g_1} \|  +  \log \|  \pi_{g_n \cdots g_1} \|  \Big| \notag\\
& =  \bigg| \log \bigg\| \frac{g}{\|g\|} \frac{g_n \cdots g_1}{\| g_n \cdots g_1 \|}  \bigg\| 
 - \log \bigg\| \frac{g}{\|g\|} \frac{\pi_{g_n \cdots g_1}}{\| \pi_{g_n \cdots g_1} \|}  \bigg\|  \bigg|. 
\end{align*} 
Using the mean value theorem for the log-function, we get 
\begin{align*}
& \Big|  \log \| g g_n \cdots g_1 \|  - \log \| g_n \cdots g_1 \|
 - \log \| g \pi_{g_n \cdots g_1} \|  +  \log \|  \pi_{g_n \cdots g_1} \|  \Big| \notag\\
 & \leq  \max \bigg\{ \frac{\|g\| \| g_n \cdots g_1 \|}{ \| g g_n \cdots g_1 \|},  \frac{\|g\| \| \pi_{g_n \cdots g_1} \|}{ \| g \pi_{g_n \cdots g_1} \|} \bigg\}
  \left| \bigg\|  \frac{g}{\|g\|} \frac{g_n \cdots g_1}{\| g_n \cdots g_1 \|} \bigg\|  -   \bigg\| \frac{g}{\|g\|} \frac{\pi_{g_n \cdots g_1}}{\| \pi_{g_n \cdots g_1} \|} \bigg\|  \right|  \notag\\
& \leq  \max \bigg\{ \frac{\|g\| \| g_n \cdots g_1 \|}{ \| g g_n \cdots g_1 \|},  \frac{\|g\| \| \pi_{g_n \cdots g_1} \|}{ \| g \pi_{g_n \cdots g_1} \|} \bigg\}
  \bigg\| \frac{g}{\|g\|} \left( \frac{g_n \cdots g_1}{\| g_n \cdots g_1 \|} - \frac{\pi_{g_n \cdots g_1}}{\| \pi_{g_n \cdots g_1} \|} \right) \bigg\|  \notag\\
&  \leq  c_0  \max \bigg\{ \frac{\|g\| \| g_n \cdots g_1 \|}{ \| g g_n \cdots g_1 \|},  \frac{\|g\| \| \pi_{g_n \cdots g_1} \|}{ \| g \pi_{g_n \cdots g_1} \|} \bigg\}
  \bigg\|  \frac{g_n \cdots g_1}{\| g_n \cdots g_1 \|} -  \frac{\pi_{g_n \cdots g_1}}{\| \pi_{g_n \cdots g_1} \|}  \bigg\|, 
\end{align*} 
where $c_0 = \max_{u, v \in \End (V) \setminus \{0\}} \frac{\|u v\|}{\|u\| \|v\|}$ is a constant
(as the norm $\|\cdot \|$ is not assumed to be an operator norm, we may have $c_0 > 1$). 
Let $\ee, c_1 >0$ be as in Lemma \ref{Lem-norm-comparison-001} and let $a, c_2 >0$ be as in Lemma \ref{Lem-norm-comparison-002}
which we apply to $\ee/2$.
Then, outside an event of probability at most $c_1 e^{-\ee n} + 2 c_2 e^{-a n}$, 
we have 
\begin{align*}
&  \frac{\|g\| \| g_n \cdots g_1 \|}{ \| g g_n \cdots g_1 \|} \leq e^{\frac{\ee}{2} n},  
\quad   \frac{\|g\| \| \pi_{g_n \cdots g_1} \|}{ \| g \pi_{g_n \cdots g_1} \|}   \leq e^{\frac{\ee}{2} n},   \notag\\
& \bigg\| \frac{g_n \cdots g_1}{\| g_n \cdots g_1 \|} -  \frac{\pi_{g_n \cdots g_1}}{\| \pi_{g_n \cdots g_1} \|} \bigg\| \leq e^{ -\ee n}. 
\end{align*} 
The conclusion of the proposition follows. 
\end{proof}

\subsection{Proof of Theorem \ref{Main-Thm-LLT}}\label{Section-Proof-Main-theorem}
We shall in fact prove the following more general statement. 

\begin{theorem}\label{Main-Thm-LLT-general}
Assume that $\Gamma_{\mu}$ is strongly irreducible 
and that the image of $\Gamma_{\mu}$ in the projective linear group 
${\rm PGL}(V)$ is not contained in any compact subgroup of ${\rm PGL}(V)$. 
Assume also that $\mu$ admits a finite exponential moment \eqref{Exponential-moment}. 
Let $f$ be a compactly supported continuous function on $\bb R$. 
Then, uniformly in $w_n \in \bb R$ with $|w_n| = o(\sqrt{n})$,   
\begin{align*}
\upsilon_{\mu} \sqrt{2 \pi n} \,  \bb E  \Big( f (\log \|g_n \cdots g_1 \| - n \lambda_{\mu} - w_n)  \Big) 
\xrightarrow[n\to\infty]{}  \int_{\bb R} f(u) du. 
\end{align*}
\end{theorem}

Here, the asymptotic variance $\upsilon_{\mu}^2$ is defined by Remark \ref{Remark-variance-002}, but we will see below that it coincides with 
the one given by \eqref{def-variance-aa}. 

\begin{proof}[Proof of Theorem \ref{Main-Thm-LLT-general}]
Let $f$ be a compactly supported continuous function on $\bb R$ and for $n \geq 1$, set $p_n = [A \log n]$, where $A>0$ is a fixed constant. 
Applying Proposition \ref{Prop-norm-comparison}, if $A$ is chosen sufficiently large, we get that, uniformly in $w_n \in \bb R$ with $|w_n| = o(\sqrt{n})$,   
\begin{align*}
&  \upsilon_{\mu} \sqrt{2 \pi n} \bigg(   \bb E  \Big[ f(\log \|g_n \cdots g_1 \| - n \lambda_{\mu} -  w_n) \Big]  \notag\\
&   -   
\bb E  \Big[ f \Big(\log  \frac{ \| g_n \cdots g_{p_n + 1} \pi_{g_{p_n} \cdots g_1} \| }{ \| \pi_{g_{p_n} \cdots g_1} \| } 
  - n \lambda_{\mu} - w_n + \log \|g_{p_n} \cdots g_1\| \Big)  \Big] \bigg) 
  \xrightarrow[n\to\infty]{}  0. 
\end{align*}
Now, by Chebyshev's inequality and the finite exponential moment assumption, we have
\begin{align*}
\bb P \left(  | \log \|g_{p_n} \cdots g_1\| | > n^{1/4} \right) = o \left( n^{-1/2} \right). 
\end{align*}
Therefore, applying Theorem \ref{Thm-LLT-intermediate} with the perturbation $w_n' = w_n - \log \|g_{p_n} \cdots g_1\| + p_n \lambda_{\mu}$, 
we obtain that, uniformly in $w_n \in \bb R$ with $|w_n| = o(\sqrt{n})$,   
\begin{align*}
 \upsilon_{\mu} \sqrt{2 \pi n} 
\,  \bb E  \Big[ f(\log \|g_n \cdots g_1 \| - n \lambda_{\mu} - w_n) \Big]  
\xrightarrow[n\to\infty]{}   \int_{\bb R} f(u) du, 
\end{align*}
as claimed. 
\end{proof}

\begin{remark}\label{Remark-variance-003}
From Remark \ref{Remark-variance-002} and Proposition \ref{Prop-norm-comparison}, 
the same argument as in the proof above shows that the asymptotic variance $\upsilon_{\mu}^2$
introduced in Remark \ref{Remark-variance-002} coincides with the expression given in \eqref{def-variance-aa}. 
\end{remark}

\begin{proof}[Proof of Theorem \ref{Main-Thm-LLT}]
The assertion of Theorem \ref{Main-Thm-LLT} follows directly from Theorem \ref{Main-Thm-LLT-general} 
by taking $w_n = 0$ and using standard approximation arguments. 
\end{proof}


\end{document}